\documentclass[11pt, twoside]{article}

\def\titlerunning#1{\gdef\titrun{#1}}
\makeatletter
\def\author#1{\gdef\autrun{\def\and{\unskip, }#1}\gdef\@author{#1}}
\def\address#1{{\def\and{\\\hspace*{18pt}}\renewcommand{\thefootnote}{}%
\footnote {#1}}%
\markboth{\autrun}{\titrun}}
\makeatother
\def\email#1{e-mail: #1}
\def\subjclass#1{{\renewcommand{\thefootnote}{}%
\footnote{\emph{Mathematics Subject Classification (2020):} #1}}}
\def\keywords#1{\par\medskip
\noindent\textbf{Keywords.} #1}

\input{GeoWebPaper-Macros}

\begin{document}




\titlerunning{Geodesics in LPP}

\title{Geometry of geodesics through Busemann measures\\ in directed last-passage percolation}

\author{Christopher Janjigian
\and
Firas Rassoul-Agha
\and  
Timo Sepp\"al\"ainen}

\date{September 3, 2021} 

\maketitle

\address{C.\ Janjigian: Purdue University,  Mathematics Department,  150 N.\ University Street,  West Lafayette IN 47907, USA; \email{cjanjigi@purdue.edu}
\and
F.\ Rassoul-Agha: University of Utah, Mathematics Department, 155S 1400E, Salt Lake City UT 84112, USA; \email{firas@math.utah.edu}
\and
T.\ Sepp\"al\"ainen: University of Wisconsin-Madison, Mathematics Department, Van Vleck Hall, 480 Lincoln Dr., Madison WI 53706, USA; \email{seppalai@math.wisc.edu}}



\begin{abstract}
We consider  planar directed last-passage percolation on the square lattice with general i.i.d.\ weights and study the geometry of the full set of semi-infinite geodesics in a typical realization of the random environment. The structure of the geodesics is studied through the properties of  the Busemann functions viewed as a stochastic process indexed by the asymptotic direction. Our results are further connected to the ergodic program for and stability properties of random Hamilton-Jacobi equations. 
In the exactly solvable exponential model, our results specialize to give the first complete characterization of the uniqueness and coalescence structure of the entire family of semi-infinite geodesics for any model of this type. Furthermore, we compute statistics of locations of instability, where we discover an unexpected connection to simple symmetric random walk.

\end{abstract}

\tableofcontents



\section{Introduction} 

\subsection{Random growth models} Irregular or random growth is a ubiquitous phenomenon in nature,   from the growth of tumors, crystals, and bacterial colonies to the propagation of forest fires and the spread of water through a porous medium. Models of random growth have been a driving force in probability theory over the last sixty years and a wellspring of important ideas \cite{Auf-Dam-Han-17}. 

The mathematical analysis of such models began in the early 1960s with the introduction of the {\it Eden model} by Eden \cite{Ede-61} and {\it first-passage percolation} (FPP) by Hammersley and Welsh \cite{Ham-Wel-65}.  
About two decades later, early forms of a directed variant of FPP,  {\it directed last-passage percolation}  (LPP), appeared in a paper by Muth \cite{Mut-79} in connection with {\it series of queues in tandem}. 
Soon after, Rost \cite{Ros-81} introduced a random growth model, now known as the {\it corner growth model} (CGM), in connection with the {\it totally asymmetric simple exclusion process} (TASEP), a model of interacting particles. 
A decade later, the CGM arose naturally from LPP in queueing theory in the work of Szczotka and Kelly \cite{Szc-Kel-90} and Glynn and Whitt \cite{Gly-Whi-91}.
Around the same time, the third author \cite{Sep-98-mprf-1} connected the CGM and LPP to  Hamilton-Jacobi equations and Hopf-Lax-Oleinik semigroups.

 Much of this early work was primarily concerned with the deterministic asymptotic shape and large deviations of the randomly growing interface.   The breakthrough  of Baik, Deift, and Johansson  \cite{Bai-Dei-Joh-99}  showed that   the  fluctuations of the Poissonian LPP model   have the same  limit as the fluctuations of the largest eigenvalue of  the {\it Gaussian unitary ensemble}  derived by Tracy and Widom in \cite{Tra-Wid-94}.   This result was extended to  the exactly solvable versions of the CGM  by  Johansson  in \cite{Joh-00}.  These results  marked the CGM and the related LPP and TASEP models as members of the {\it Kardar-Parisi-Zhang}  (KPZ) universality class.  This  universality class is conjectured to describe the statistics of a growing interface   observed when a rapidly mixing stable state invades a rapidly mixing metastable state.  This subject  has been a major focus of   probability theory and statistical physics over the last three decades.  Recent surveys appear in  \cite{Cor-12,Cor-16,Hal-Tak-15,Qua-Spo-15,Qua-12}.

\subsection{Geodesics} A common feature of many models of random growth is  a natural metric-like interpretation  in which there exist paths that can be thought of as geodesics. In these interpretations, the growing interface can be viewed as a sequence of balls of increasing radius and centered at the origin. This connection is essentially exact in the case of FPP, which  describes a random pseudo-metric on $\bbZ^d$. Related models like the CGM and stochastic Hamilton-Jacobi equations have natural extremizers through their Hopf-Lax-Oleinik semi-groups, which share many of the properties of geodesics. For this reason and following the convention in the field, we will call all such paths geodesics.

Considerable effort has been devoted to understanding the geometric structure of semi-infinite geodesics in models of random growth. In the mathematical literature, this program was largely pioneered in the seminal work of Newman and co-authors \cite{New-95, Lic-New-96, How-New-97, How-New-01}, beginning with his paper in the 1994 Proceedings of the ICM \cite{New-95}. Under strong hypotheses on the curvature of the limit shape, that early work showed that all such geodesics must be asymptotically directed and that for Lebesgue-almost every fixed direction, from each site of the lattice, there exists a unique semi-infinite geodesic with that asymptotic direction and all these geodesics coalesce.  In special cases where the curvature  hypotheses are met, Newman's program was subsequently implemented in LPP models  \cite{Cat-Pim-11,Cat-Pim-12,Cat-Pim-13,Fer-Pim-05, Wut-02} and  certain stochastic Hamilton-Jacobi equations \cite{Bak-Li-19,Bak-16,Bak-Cat-Kha-14}. In all  the results of the last twenty-five years, the obstruction of needing to work on direction-dependent events of full probability  has been a persistent issue. A description of the overall geometric structure of semi-infinite geodesics has remained elusive. 

It is known that the picture described by these now-classical methods cannot be complete, because uniqueness fails for countably infinitely many random directions \cite{Geo-Ras-Sep-17-ptrf-2,Fer-Mar-Pim-09,Cou-11}. In the CGM, these special directions are the asymptotic directions of {\it competition interfaces}.    These  are dual lattice paths that separate geodesics rooted at a fixed site. Competition interface directions are distinguished by the existence of (at least)  two geodesics that emanate from the same site,  have  the same asymptotic direction, but  separate immediately in their first step. Once these two geodesics separate they never intersect again.  So in these directions  coalescence also fails. 


Borrowing ideas from classical metric geometry, Newman introduced the tool of Busemann functions into the field in \cite{New-95}. In Newman's work, these Busemann functions are defined as directional limits of differences of metric distances or passage times. Following Newman's work and the subsequent seminal work of Hoffman in \cite{Hof-08}, Busemann functions have become a  principal tool for studying semi-infinite geodesics. The existence of the Busemann limits, however, relies on strong hypotheses on the limit shape.  Modern work primarily uses generalized Busemann functions, which exist without assumptions on the limit shape \cite{Dam-Han-14,Dam-Han-17,Geo-Ras-Sep-17-ptrf-1,Geo-Ras-Sep-17-ptrf-2,Ahl-Hof-16-}.

\subsection{Busemann measures} The present paper  introduces a new framework that relates geometric properties of geodesics to analytic properties of a measure-valued stochastic process called the {\it Busemann process} or {\it Busemann measures}. 
These Busemann measures are  Lebesgue-Stieltjes measures of generalized Busemann functions on the space of spatial directions, and  the Busemann process is the associated family of distribution functions.   This  approach enables a study of the entire family of semi-infinite geodesics on a single event of full probability. 

We  describe, in terms of the supports of the Busemann measures, the random exceptional directions in which uniqueness or coalescence of geodesics fails. Many of these  results hold without further assumptions on the weight distribution. 
This work also identifies key hypotheses 
that are  equivalent to desirable coalescence and uniqueness properties of geodesics. 
We expect that our methods will apply in related models including FPP and stochastic Hamilton-Jacobi equations.

In the exactly solvable case with i.i.d.\ exponential weights, when the new results  are combined with previous work from   \cite{Cou-11, Fer-Pim-05, Geo-Ras-Sep-17-ptrf-2}, this yields  a complete characterization of the uniqueness and coalescence structure of all semi-infinite geodesics    on a single event of full probability.  Here is a  summary: 
\begin{enumerate}[label={\rm(\roman*)}, ref={\rm\roman*}]  \itemsep=3pt 
\item Every semi-infinite geodesic has an asymptotic direction.
\item 
 There exists a random countably infinite dense set of interior directions in which  there are {\it exactly two} geodesics  from each lattice site, a left geodesic and a right geodesic. These two  families of left and right geodesics can be constructed from the Busemann process.   Each family   forms a tree of coalescing geodesics.
  
  \item   In every other interior direction there is a unique geodesic from each lattice point, which again can be constructed from the Busemann process. In each such  direction these geodesics coalesce to form a tree. 
\item\label{summary:v} The countable set of  directions of non-uniqueness is exactly  the set of asymptotic directions of competition interfaces 
from all lattice points, in addition to being the set of  discontinuity directions of the Busemann process.
\item In a direction $\xi$ of non-uniqueness, finite geodesics out of a site $x$ with endpoints going in direction $\xi$ converge to the left (resp.\ right) semi-infinite geodesic out of $x$ with asymptotic direction $\xi$ if and only if the endpoints eventually stay to the left (resp.\ right) 
of the competition interface rooted at the point where the left and right semi-infinite geodesics out of $x$ split.  
\item In a direction $\xi$ of uniqueness, finite geodesics out of a site $x$ with endpoints going in direction $\xi$ converge to the  semi-infinite geodesic out of $x$ with asymptotic direction $\xi$.
\end{enumerate}
This gives the first complete accounting of semi-infinite geodesics in a model which lies in the KPZ class.\smallskip

\subsection{Instability points} Passage times in LPP solve a variational problem  that is a   discrete version of the stochastic Burgers  Hopf-Lax-Oleinik semigroup.  
Through this connection, this paper  is also related  to the ergodic program for the 
  stochastic Burgers equation  initiated by Sinai in \cite{Sin-92}. 
  As mentioned in point \eqref{summary:v} above, the exceptional directions in which coalescence fails correspond to directions at which the Busemann process has jump discontinuities.  
  This means that the Cauchy problem at time $-\infty$ is not well-posed for certain initial conditions  that correspond to these exceptional directions. 
  In this case, it is reasonable to expect that solving the Cauchy problem with the initial condition given at time $t_0$ and sending $t_0\to-\infty$ gives multiple limits at 
  the space-time locations where the Busemann process has jump discontinuities. Thus we call these locations points of instability.
  In situations where the Cauchy problem is well-posed, points of instability correspond to shock locations. 
  The structure of shocks in connection with the Burgers program has been a major line of research \cite{E-etal-00,Bec-Kha-07,Bak-16}, with a conjectured relationship between shock statistics and the KPZ universality phenomenon (Bakhtin and Khanin \cite{Bak-Kha-18}). These conjectures are open. 
  



 Past works \cite{E-etal-00,Bec-Kha-07,Bak-16,Bak-Kha-18} considered shocks in  fixed deterministic directions, where the Cauchy problem at time $-\infty$ is shown to be well-posed almost surely and  these shocks are the only points of instability.  
  Our model is in a non-compact space setting, where these problems have been especially difficult to study. In exceptional directions, 
  points of instability turn out to have a markedly different structure from what has been seen previously in fixed directions.  Among the new phenomena are that points of instability form bi-infinite paths that both branch and coalesce. 
Bi-infinite shock paths have previously been observed only when  the space is compact  and the asymptotic direction is fixed. 
Branching shocks have not been observed.

 In the exponential model we  compute non-trivial  statistics of points of instability. 
 Among  our  results  is an unexpected connection with simple symmetric random walk: conditional on a $\xi$-directed path of instability points passing through the origin, the distribution of the locations of $\xi$-points of instability on the $x$-axis  
has the same law as the zero set of simple symmetric random walk sampled at even times.

\subsection{Organization of the paper}
Section \ref{sec:busgeo} defines the model and summarizes the currently known results  on Busemann functions and existence, uniqueness, and coalescence of geodesics.
Section \ref{s:Bmeas}  contains  our main results on Busemann measures and the geometry of geodesics for general weight distributions.  Section \ref{sec:LPPDS}  connects  our general results to dynamical systems and studies
the web of instability  defined by the discontinuities of the Busemann process. Section \ref{sec:exp} specializes to the  exponential case  to compute non-trivial statistics of the Busemann process.  Proofs come  in Sections \ref{sec:bus}--\ref{sec:exp-pf}, with some auxiliary results relegated to  Appendices \ref{a:aux}--\ref{a:RW}.  Appendix \ref{app:busgeo} collects the inputs we need from previous work.

\subsection{Setting and notation} \label{sub:setting}
Throughout this paper, $(\Omega, \sF,\bbP)$ is a Polish probability space equipped with a group $T=\{T_x\}_{x\in\Z^2}$ of $\sF$-measurable $\bbP$-preserving bijections $T_x:\Omega\to\Omega$
such that $T_0=$ identity and  $T_xT_y=T_{x+y}$.
$\bbE$ is expectation relative to $\bbP$.
A generic point in this space is denoted by $\w \in \Omega$. We assume that there exists a family $\{\w_x(\w) : x \in \bbZ^2\}$ of real-valued random variables called {\it weights} such that  
	\begin{align}\label{main-assump}
	\begin{split}
	&\text{$\{\w_x\}$ are i.i.d.\ with a continuous distribution under $\bbP$,  $\Var(\w_0)>0$,}\\[-2.5pt]
	&\text{and $\exists p>2$\,: $\bbE[\abs{\w_0}^p] < \infty$.}
	\end{split}
	\end{align}
We require further that  $\w_y(T_x \w) = \w_{x+y}(\w)$ for all $x,y \in \bbZ^2$.   $\P_0$  denotes the marginal distribution of $\{\w_x:x\in\Z^2\}$ under $\P$.   
Continuous distribution means that  $\P_0(X\le r)$ is a continuous function of $r\in\R$.   $X\sim\text{ Exp}(\alpha)$ means that the random variable $X$ satisfies $P(X>t)=e^{-\alpha t}$ for $t>0$ (rate $\alpha$ exponential distribution). 

The canonical setting is the one where $\Omega=\R^{\Z^2}$ is endowed with the product topology, Borel $\sigma$-algebra $\sF$, and the natural shifts, $\w_x$ are the coordinate projections, and $\P=\P_0$ is a product shift-invariant measure.

The standard basis vectors of $\bbR^2$ are  $e_1= e_+ = (1,0)$ and $e_2=e_-=(0,1)$. The $e_\pm$ notation will   conveniently  shorten some statements.  
Additional special vectors are   $\et = e_1 + e_2$, $\etstar=\et/2$, $\ex=e_2-e_1$, and $\exstar=\ex/2$.  In the dynamical view of LPP,  $\et$ is the  time coordinate and  $\ex$  the  space coordinate.  See Figure \ref{fig:evecs}.
The spatial  level at time  $t\in \bbZ$ is denoted by  $\level_t = \{x \in \bbZ^2 : x \cdot \et = t\}$.     The half-vectors $\etstar$ and $\exstar$ connect   $\Z^2$  with its dual lattice $\Z^{2*}=\etstar+\Z^2$.

\begin{figure}[h]
 	\begin{center}
 		 \begin{tikzpicture}[scale=1,>=stealth]
			\begin{scope}
			\draw(0,-0.5)--(0,1.5);
			\draw(1,-0.5)--(1,1.5);
			\draw(-1,-0.5)--(-1,1.5);
			\draw(-1.5,0)--(1.5,0);
			\draw(-1.5,1)--(1.5,1);
			\draw[line width=1.5pt,->](0,0)--(0,1);		 
			\draw(0.2,1)node[above]{\small$e_2=e_-$};
			\draw[line width=1.5pt,->](0,0)--(1,0);		 
			\draw(1.65,0)node[above]{\small$e_1=e_+$};
			\draw[line width=1.5pt,->](0,0)--(1,1);		 
			\draw(1.2,1)node[above]{\small$\et$};
			\draw[line width=1.5pt,->](0,0)--(-1,1);		 
			\draw(-1.2,1)node[above]{\small$\ex$};
			\draw[fill=black](0,0) circle(3pt);
			\draw(0,-.2)node[left]{\small$0$};
			\end{scope}

			\begin{scope}[shift={(4,0)}]
			\draw(0,-0.5)--(0,1.5);
			\draw(1,-0.5)--(1,1.5);
			\draw(-1,-0.5)--(-1,1.5);
			\draw(-1.5,0)--(1.5,0);
			\draw(-1.5,1)--(1.5,1);
			\draw[dashed](-1.5,0.5)--(1.5,0.5);
			\draw[dashed](-.5,-0.5)--(-.5,1.5);
			\draw[dashed](.5,-0.5)--(.5,1.5);
			\draw[line width=1.5pt,->](0,0)--(0.5,0.5);		 
			\draw(0.8,0.5)node[above]{\small$\etstar$};
			\draw[line width=1.5pt,->](0,0)--(-0.5,0.5);		 
			\draw(-0.7,0.5)node[above]{\small$\exstar$};
			\draw[fill=black](0,0) circle(3pt);
			\draw(0,-.2)node[left]{\small$0$};
			\end{scope}
			
			\begin{scope}[shift={(7,0)}]
			\draw(0,-0.5)--(0,1.5);
			\draw(1,-0.5)--(1,1.5);
			\draw(-.5,0)--(1.5,0);
			\draw(-.5,1)--(1.5,1);
			\draw[line width=1.5pt](1,0)--(0,1);
			\draw(0.7,0.7)node{\small$\Uset$};
			\draw[fill=black](0,0) circle(3pt);
			\draw(0,-.2)node[left]{\small$0$};
			\draw(0.2,1)node[above]{\small$e_2$};
			\draw(1.2,0)node[above]{\small$e_1$};
			\end{scope}
			
		\end{tikzpicture}
 	 \end{center}
 	\caption{\small  An illustration of the  vectors $e_1$, $e_2$, $e_{\pm}$, $\et$, $\ex$, $\etstar$,  $\exstar$, and the set $\Uset$. The dashed lines in the middle plot are edges of the dual lattice $\Z^{2*}=\bbZ^2+\etstar$.}
 \label{fig:evecs}
 \end{figure}

A statement with $\pm$ and possibly also $\mp$  is a conjunction of two statements: one for the top signs, and another one for the bottom signs.  We employ $\sigg$ to represent an arbitrary element of $\{-,+\}$.




We  use  $\R_+=[0,\infty)$, $\Z_+=\Z\cap\R_+$  and $\N=\{1,2,3,\dotsc\}$.   For $x,y\in\R^2$, inequalities such as $x\le y$ and $x< y$, 
and operations such as $x\wedge y=\min(x,y)$ and $x\vee y=\max(x,y)$
are understood coordinatewise.  (In particular, $x<y$ means $x\cdot e_i<y\cdot e_i$ for both $i=1,2$.)  For $x\le y$ in $\Z^2$, $\lzb x, y\rzb$ denotes the rectangle $\{z\in\Z^2:x\le z\le y\}$. 
For integers $i\le j$, $\lzb i,j\rzb$ denotes the interval $[i,j]\cap\Z$.  For $m\le n$ in $\Z\cup\{-\infty,\infty\}$ we denote a sequence $\{a_i:m\le i\le n\}$ by $a_{m,n}$. 

A path $\geod{m,n}$ in $\bbZ^2$ with $\geod{i+1}- \geod{i} \in \{e_1,e_2\}$ for all $i$ is called an {\it up-right} path. 
Throughout, paths are indexed so that $\geod{k}\cdot\et=k$.

For vectors $\zeta,\eta \in \bbR^2$, denote  open and closed line segments by  $]\zeta,\eta[\, = \{t\zeta + (1-t)\eta : 0 < t < 1\}$ and  $[\zeta,\eta] = \{t \zeta + (1-t) \eta : 0 \leq t \leq 1\}$, with the consistent definitions for  $]\zeta,\eta]$ and $[\zeta,\eta[$.   $\Uset = [e_2,e_1]$ with  relative interior $\ri \Uset = ]e_2,e_1[$. See Figure \ref{fig:evecs}. A left-to-right ordering of points $\zeta,\eta\in\R^2$ with $\zeta\cdot\et=\eta\cdot\et$ is defined by $\zeta \prec \eta$ if $\zeta \cdot e_1 < \eta \cdot e_1$ and $\zeta \preceq \eta$ if $\zeta \cdot e_1  \leq \eta \cdot e_1$.   This leads to notions of left and right limits: if  $\zeta_n \to \xi$ in $\Uset$, then $\zeta_n \nearrow \xi$ if $\zeta_n \prec \zeta_{n+1}$ for all $n$, while $\zeta_n \searrow \xi$ if $\zeta_{n+1} \prec \zeta_n$  for all $n$. 

The support $\supp\mu$ of a signed Borel measure $\mu$  is the smallest closed set whose complement has zero measure under the total variation measure $\abs\mu$.   
%
\medskip


\section{Preliminaries on last-passage percolation}\label{sec:busgeo}

This section introduces   the background required for the main results in Sections \ref{s:Bmeas}--\ref{sec:exp}.  To avoid excessive technical detail at this point,  precise  statements of  previous  results needed for the proofs later in the paper  are deferred to Appendix \ref{app:busgeo}.

\subsection{The shape function} \label{s:shape} 
Recall the assumption \eqref{main-assump}.
For $x \leq y$ in $\bbZ^2$ satisfying $x \cdot \et = k$ and $y \cdot \et = m$,  denote by $\Pi_x^y$ the collection of up-right paths $\geod{k,m}$ which satisfy $\geod{k} = x$ and $\geod{m} = y$. 
The {\it last-passage time} from $x$ to $y$ is defined by 
\be\label{G1} 
G_{x,y} = G(x,y)=\max_{\geod{k,m} \in \Pi_x^y} \sum_{i=k}^{m-1} \w_{\geod{i}}. 
\ee
A maximizing path is called a (point-to-point or finite) {\it geodesic} and denoted by $\gamma^{x,y}$.
Under the i.i.d.\ continuous distribution assumption \eqref{main-assump},  $\gamma^{x,y}$   is almost surely  unique. 

The shape theorem \cite{Mar-04} says there exists a non-random  function $\gpp:\bbR_+^2 \to \bbR$ such that with probability one
\be\label{sh-th} 
\lim_{n\to\infty} \;\max_{x\, \in \,\bbZ^2_+ : \, |x|_1 = n} \frac{\abs{G_{0,x} - \gpp(x)}}{n} = 0.
\ee
This {\it shape function} $\gpp$  is symmetric, concave, and homogeneous of degree one. By homogeneity, $\gpp$  is determined by its values on $\Uset$. 
%
Concavity implies the existence of one-sided derivatives:
\begin{align*}
\nabla \gpp(\xi \pm) \cdot e_1 = \lim_{\e \searrow 0} \frac{\gpp(\xi \pm \e e_1) - \gpp(\xi)}{\pm \e} \qquad{ and }\qquad \nabla \gpp(\xi \pm) \cdot e_2 = \lim_{\e \searrow 0} \frac{\gpp(\xi \mp \e e_2) - \gpp(\xi)}{\mp \e}.
\end{align*}
By \cite[Lemma 4.7(c)]{Jan-Ras-18-arxiv} 
differentiability of $\gpp$ at $\xi\in\ri\Uset$ is the same as $\nabla \gpp(\xi+) = \nabla \gpp(\xi -)$.
Denote the directions of differentiability by
\be\label{df:D} 
\Diff = \{\xi \in \ri \sU : \gpp \text{ is differentiable at }\xi\}.
\ee

For $\xi \in \ri \Uset$,  define the maximal linear segments of $g$ with slopes given by  the    right $(\sigg=+)$ and the  left $(\sigg=-)$ derivatives of $g$ at $\xi$  to be
\begin{align*}
\Uset_{\xi\sig} = \bigl\{\zeta \in \ri\Uset : \gpp(\zeta) - \gpp(\xi) = \nabla \gpp(\xi\sigg) \cdot (\zeta - \xi)\bigr\},\quad\sigg\in\{-,+\}. 
\end{align*}
We say $g$ is {\it strictly concave} at $\xi\in\ri\Uset$ if $\Uset_{\xi-}=\Uset_{\xi+}=\{\xi\}$.  Geometrically this means that $\xi$ does not lie on a nondegenerate closed linear segment of $\gpp$. 
The usual notion of strict concavity on an open  subinterval of $\Uset$ is the same as having this pointwise 
strict concavity at   all $\xi$ in the interval.

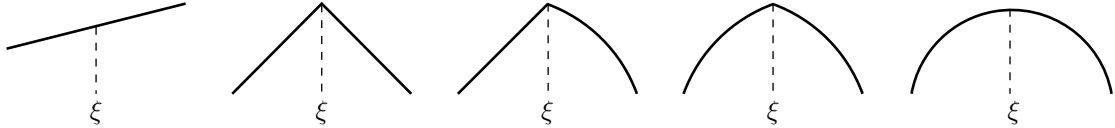
\begin{figure}[h]
\begin{center}
\begin{tikzpicture}[>=latex,scale=1]
\begin{scope}
\draw[line width=1pt] (-1.19,0.6)--(1.19,1.2);
\draw[ line width=0.5pt,dashed](0,0.89)--(0,0);
\draw(0,-.25)node{$\xi$};
\end{scope}
\begin{scope}[shift={(3,0)}]
\draw[line width=1pt] (-1.19,0)--(0,1.2)--(1.19,0);
\draw[ line width=0.5pt,dashed](0,1.19)--(0,0);
\draw(0,-.25)node{$\xi$};
\end{scope}
\begin{scope}[shift={(6,0)}]
\draw[line width=1pt] (-1.19,0)--(0.01,1.2);
\draw[line width=1pt] (1.19,0) arc (20:70:2cm);
\draw[ line width=0.5pt,dashed](0.01,1.2)--(0.01,0);
\draw(0,-.25)node{$\xi$};
\end{scope}
 \begin{scope}[shift={(9,0)}]
\draw[line width=1pt] (1.19,0) arc (20:70:2cm);
\draw[line width=1pt] (-1.19,0) arc (160:110:2cm);
\draw[ line width=0.5pt,dashed](0,1.2)--(0,0);
\draw(0,-.25)node{$\xi$};
\end{scope}
\begin{scope}[shift={(11.5,0)}]
\draw[line width=1pt] (2,0) arc (10:170:1.35cm);
\draw[ line width=0.5pt,dashed](0.65,1.13)--(0.65,0);
\draw(0.7,-.25)node{$\xi$};
\end{scope}
\end{tikzpicture}
\end{center}
\caption{\small   In the first three graphs $\gpp$ is not strictly concave at $\xi$ while in the last two it is.}
\label{fig:conc} \end{figure}

For a given $\xi\in\ri\Uset$, let $\ximin\preceq\ximax$ denote the endpoints of the (possibly degenerate) interval 
\begin{align*}
\Uset_{\xi} &= \Uset_{\xi-} \cup \Uset_{\xi +} = [\hspace{0.5pt}\ximin,\ximax\hspace{0.8pt}].  
\end{align*}
If $\xi\in\Diff$ then $\Uset_{\xi-}=\Uset_{\xi+}=\Uset_{\xi}$ while if $\xi\notin\Diff$ then $\Uset_{\xi-}\cap\,\Uset_{\xi+}=\{\xi\}$. Set $\Uset_{e_i}=\{e_i\}$, for $i\in\{1,2\}$.\smallskip

Additional control over the geometry of geodesics is provided by this regularity condition:
\begin{align}\label{g-reg}
\begin{split}
&\text{The shape function $\gpp$ is strictly concave at all $\xi\notin\Diff$ or, equivalently, $\gpp$ is differentiable}\\[-2.5pt]
&\text{at the endpoints of its linear segments.}
\end{split}
\end{align}

Condition \eqref{g-reg} holds obviously  if $\gpp$ is either differentiable or strictly concave. Both of these latter properties are true  for exponential  weights   and are conjectured to be valid more generally for continuously distributed  weights. Under 
\eqref{g-reg},   if both $ \Uset_{\xi-}$  and $\Uset_{\xi +}$ are nondegenerate intervals,  then   $ \Uset_{\xi-}=\Uset_{\xi +}=\Uset_{\xi}$ (leftmost graph in Figure \ref{fig:conc}). 


\subsection{The Busemann process}
\label{sec:bus-brief}
Under regularity condition \eqref{g-reg}, it is known that for each fixed 
$\xi \in \Diff$     
 and $x,y \in \bbZ^2$, there is a $\xi$-dependent event of full probability on which the limit
\begin{align}
\B{\xi}(x,y) = \lim_{n\to\infty} (G_{x,v_n} - G_{y,v_n})  \label{eq:Busfn}
\end{align}
exists and agrees for all sequences $v_n \in \bbZ^2$ such that $\abs{v_n}\to\infty$ and  $v_n/n \to \xi$.  Similar  limits appear in metric geometry under the name of {\it Busemann functions}.

The goal of this paper  is to  study the LPP model  without a priori hypotheses on the shape function.  Hence  the  limit in \eqref{eq:Busfn} cannot serve as a starting point. Instead we work with a stochastic process of {\it generalized} Busemann functions, indexed by $\xi \in \ri \Uset$,  constructed through a weak limit procedure on an extended probability space. See Remark \ref{rmk:weak-lim} for a brief discussion of the construction of this process in \cite{Jan-Ras-20-aop}, which is based in part on ideas from \cite{Dam-Han-14, Geo-Ras-Sep-17-ptrf-1}. This  process agrees with  \eqref{eq:Busfn} when the limit in  \eqref{eq:Busfn}  exists. 

The construction in \cite{Jan-Ras-20-aop} produces  a probability space  $(\Omega, \sF, \bbP)$ with a group of shifts $T =\{T_x : x \in \bbZ^2\}$ that satisfies  the requirements of Section \ref{sub:setting} and a stochastic process  $\bigl\{\B{\xi\sigg}(x,y): x,y \in \bbZ^2,\, \xi \in \ri \Uset, \,\sigg \in \{-,+\}\bigr\}$ on $\Omega$, which we   call the {\it Busemann process}. We record here those properties of this process that  are needed for Sections \ref{s:Bmeas}--\ref{sec:exp}. 

In general, there is a $T$-invariant full probability  event on which the following hold. For all $\xi \in \ri \Uset$, $x,y,z \in \bbZ^2$, and $\sigg \in \{-,+\}$:
 \begin{align}
	&\B{\xi\sig}(x+z,y+z,\w)=\B{\xi\sig}(x,y,T_z\w),\label{cov-prop}\\
	&\B{\xi\sig}(x,y,\w)+\B{\xi\sig}(y,z,\w)=\B{\xi\sig}(x,z,\w),\label{coc-prop}\\
	&\min\bigl\{\B{\xi\sig}(x,x+e_1,\w),\B{\xi\sig}(x,x+e_2,\w)\bigr\}=\w_x,\quad\text{and}\label{rec-prop2}\\
	&\E\bigl[ \B{\xi\sig}(x,x+e_i)\bigr] = \nabla\gpp(\xi\sigg)\cdot e_i.\label{E[h(B)]}
\end{align}

Properties \eqref{cov-prop}--\eqref{coc-prop} express that each $\B{\xi\sig}$ is a {\it covariant cocycle}.   The {\it weights recovery} property \eqref{rec-prop2}  is the key  that relates these cocycles to the LPP process. 
 \eqref{E[h(B)]} shows that the Busemann process is naturally parametrized by the superdifferential of the shape function $\gpp$.
 The following monotonicity  is inherited from the path structure: for all $x \in \bbZ^2$ and $\xi,\xi' \in \ri\Uset$ with $\xi\prec\xi'$,
\begin{align}\label{mono}
\begin{split}
&\B{\xi-}(x,x+e_1,\w) \ge \B{\xi+}(x,x+e_1,\w)\ge \B{\xi'-}(x,x+e_1,\w)\ge \B{\xi'+}(x,x+e_1,\w)\quad\text{and} \\
&\B{\xi-}(x,x+e_2,\w) \le \B{\xi+}(x,x+e_2,\w)\le \B{\xi'-}(x,x+e_2,\w)\le \B{\xi'+}(x,x+e_2,\w).
\end{split}
\end{align} 

As a consequence of monotonicity and the cocycle property \eqref{coc-prop}, left and right limits exist. The signs in $\B{\xi\pm}$   correspond to   left and  right continuity: for all $x,y \in \bbZ^2$,   $\xi \in \ri \Uset$, and  $\sigg \in \{-,+\}$, 
	\begin{align}\label{Busemann-limits}
	\begin{split}
	\B{\xi-}(x,y,\w)=\lim_{\ri\Uset\,\ni\, \zeta \nearrow \xi}\B{\zeta\sig}(x,y,\w)\quad\text{and}\quad
	\B{\xi+}(x,y,\w)=\lim_{\substack{\ri\Uset\,\ni\,\zeta \searrow \xi}}\B{\zeta\sig}(x,y,\w).
	\end{split}
	\end{align}
When $\B{\xi+}(x,y,\w)=\B{\xi-}(x,y,\w)$ we drop the $+/-$ distinction and 
write $\B{\xi}(x,y,\w)$.


Theorem \ref{thm:Bus}  in Appendix \ref{app:busgeo} contains the complete list of the properties of the Busemann process that  are  used in the proofs in Sections \ref{sec:bus}--\ref{sec:webpf}. 

\subsection{Semi-infinite geodesics}\label{sec:geod-brief}

A path $\geod{k,\infty}$ with $\geod{i+1}- \geod{i} \in \{e_1,e_2\}$ for all $i\ge k$ is called a {\it semi-infinite geodesic emanating from, or rooted at}, $x$ if $\geod{k}=x$ and for any $m,n \in \bbZ_+$ with $k\le m \leq n$, the restricted path $\geod{m,n}$ is a geodesic between $\geod{m}$ and $\geod{n}$.  A path $\geod{-\infty,\infty}$ with $\geod{i+1}- \geod{i} \in \{e_1,e_2\}$ for all $i$ is called a {\it bi-infinite geodesic} if $\geod{m,n}$ is a geodesic for any $m\le n$ in $\Z$.
Due to the fact that the set of admissible steps is $\{e_1,e_2\}$, from each site $x$,  there are always two trivial semi-infinite geodesics, namely $x+\Z_+ e_1$, which we denote by $\geo{}{x}{e_1}$, and $x+\Z_+ e_2$, which we denote by $\geo{}{x}{e_2}$. There are two trivial bi-infinite geodesics going through $x$, namely $x+\Z e_1$ and $x+\Z e_2$, which we do not introduce notation for. 

A semi-infinite geodesic $\geod{k,\infty}$, or a bi-infinite geodesic $\geod{-\infty,\infty}$, is  {\it directed}  into a set $\cA\subset\Uset$ if the limit points of $\geod{n}/n$ as $n\to\infty$  are all in $\cA$.  When $\cA=\{\xi\}$ the condition becomes $\lim_{n\to\infty} \geod{n}/n=\xi$ and we  say $\geod{k,\infty}$ is {\it $\xi$-directed}.\smallskip

Using the Busemann process, we construct a semi-infinite path $\geo{}{x}{\xi\sig}$  for each $\xi\in\ri\Uset$, both signs $\sigg\in\{-,+\}$,  and all  $x \in \bbZ^2$,  via these rules: the initial point is  $\geo{m}{x}{\xi\sig}=x$ where $m=x\cdot\et$, and for $n\ge m$
\be\label{d:bgeo}\begin{aligned}
\geo{n+1}{x}{\xi\sig} &= \begin{cases}
\geo{n}{x}{\xi\sig} + e_1, & \text{ if } \B{\xi\sig}(\geo{n}{x}{\xi\sig},\geo{n}{x}{\xi\sig} + e_1) < \B{\xi\sig}(\geo{n}{x}{\xi\sig},\geo{n}{x}{\xi\sig} + e_2),\\
 \geo{n}{x}{\xi\sig} + e_2, & \text{ if } \B{\xi\sig}(\geo{n}{x}{\xi\sig},\geo{n}{x}{\xi\sig} + e_1) > \B{\xi\sig}(\geo{n}{x}{\xi\sig},\geo{n}{x}{\xi\sig}+ e_2),\\
\geo{n}{x}{\xi\sig} + e_{\sig},& \text{ if } \B{\xi\sig}(\geo{n}{x}{\xi\sig},\geo{n}{x}{\xi\sig} + e_1) = \B{\xi\sig}(\geo{n}{x}{\xi\sig},\geo{n}{x}{\xi\sig}+ e_2).
\end{cases}
\end{aligned}
\ee
As above, we dispense with the $\pm$ distinction when $\geo{}{x}{\xi+}=\geo{}{x}{\xi-}$.  These geodesics inherit an ordering from \eqref{mono}:   for all $x\in\Z^2$, $n\ge x\cdot\et$, and $\zeta\prec\eta$ in $\ri\Uset$,
	\begin{align}\label{path-ordering}
	\geo{n}{x}{\zeta-}\preceq\geo{n}{x}{\zeta+}\preceq\geo{n}{x}{\eta-}\preceq\geo{n}{x}{\eta+}.
	\end{align}

Similarly, the geodesics inherit one-sided continuity from \eqref{Busemann-limits}  in the sense of convergence of finite length segments: 
for all $x \in \bbZ^2$,   $\xi \in \ri \Uset$ and $\sigg\in\{-,+\}$, if $k = x \cdot \et$ and $m \geq k$ is an integer, then
\be\label{pmgeolim}
\lim_{\ri\Uset\,\ni\,\eta \searrow \xi} \geo{k,m}{x}{\eta\sig}  =  \geo{k,m}{x}{\xi+}\quad \text{and}\quad\lim_{\ri\Uset\,\ni\,\zeta \nearrow \xi} \geo{k,m}{x}{\zeta\sig} =  \geo{k,m}{x}{\xi-}.
\ee 

 An elementary  argument given in 
\cite[Lemma 4.1]{Geo-Ras-Sep-17-ptrf-2}
shows that properties \eqref{coc-prop} and \eqref{rec-prop2} combine to imply that these paths are all semi-infinite geodesics and that, moreover, for all choices of $x \in \bbZ^2$,  $n\ge x\cdot\et$,  $\sigg\in\{-,+\}$, and $\xi\in\ri\Uset$, we have 
\begin{align} \label{eq:buspass}
G(x,\geo{n}{x}{\xi\sig}) &= \B{\xi\sig}(x,\geo{n}{x}{\xi\sig}).
\end{align}



Below  are the  main properties of these  Busemann geodesics  $\geo{}{x}{\xi\sig}$ under assumption \eqref{main-assump}, from  article \cite {Geo-Ras-Sep-17-ptrf-2}.  (Theorem \ref{thm1}  in Appendix \ref{app:busgeo} provides  a more  precise accounting.) 
\begin{enumerate}[label={\rm(\roman*)}, ref={\rm\roman*}]  \itemsep=3pt 
\item Every semi-infinite geodesic is $\Uset_\xi$-directed for some $\xi\in\Uset$.
\item $\geo{}{x}{\xi\sig}$ is $\Uset_{\xi\sig}$-directed for each $x \in \bbZ^2$ and each $\xi \in \Uset$.
\item \label{pt:unique} If $\xi,\ximin,\ximax \in \Diff$, then there is a $\xi$-dependent event of full probability on which $\geo{}{x}{\xi-} = \geo{}{x}{\xi+}$ for all $x \in \bbZ^2$.
\item\label{pt:coal}  There is a $\xi$-dependent event of full probability on which $\geo{}{x}{\xi\sig}$ and $\geo{}{y}{\xi\sig}$ coalesce for each $\sig \in \{+,-\}$. That is, for each $x,y \in \bbZ^2$, there exists an $\w$-dependent $K \in \bbN$ such that for all $k \geq K$,  $\geo{k,\infty}{x}{\xi\sig}= \geo{k,\infty}{y}{\xi\sig}$. 
\end{enumerate}


The regularity condition \eqref{g-reg} guarantees that $\geo{}{x}{\xi-}$ and $\geo{}{x}{\xi+}$ are extreme among the $\Uset_\xi$-directed geodesics out of $x$ in the sense that for any $x\in\Z^2$,  $\xi\in\ri\Uset$,  and any $\Uset_\xi$-directed semi-infinite geodesic $\geod{}$ emanating from $x$, we have 
\begin{equation}\label{eq:geoorder}	
	\geo{n}{x}{\xi-} \preceq \geod{n} \preceq \geo{n}{x}{\xi+} 
\end{equation}
for all $n\ge x\cdot\et$. We record this fact as Theorem \ref{thm:extreme} in Appendix \ref{app:busgeo}. 

Under the regularity condition \eqref{g-reg} and $\xi,\ximin,\ximax\in\Diff$, part \eqref{pt:unique} combined with \eqref{eq:geoorder} implies that there is a $\xi$-dependent event of full probability on which there is a unique $\Uset_\xi$-directed geodesic from each $x \in \bbZ^2$. Moreover, by part \eqref{pt:coal}, all of these geodesics coalesce. On the other hand, under the same condition, it is known that 
there are exceptional random directions at which both uniqueness and coalescence fail. We discuss these directions in the next subsection.






\subsection{Non-uniqueness  of directed semi-infinite geodesics}
\label{s:mult-geo} 
For a fixed site $x \in \bbZ^2$, a natural direction in which non-uniqueness occurs is  the {\it competition interface direction}, which we denote by $\cid(T_x\w)$.  At the origin, $\cid(\w)\in\ri\Uset$ is   the unique direction such that 
	\begin{align}\label{cid}
	\B{\zeta\pm}(e_1,e_2)<0\quad\text{if}\quad\zeta\prec\cid(\w)
	\quad\text{and}\quad
	\B{\zeta\pm}(e_1,e_2)>0\quad\text{if}\quad\zeta\succ\cid(\w).
	\end{align}
Theorem \ref{thm:cif} records the main properties of competition interface directions, including the existence and uniqueness of such a direction.

Under  the regularity condition \eqref{g-reg}, we also have the following alternative description of $\cid(\w)$. Fix a site $x\in\Z^2$.  The uniqueness of finite geodesics implies that the collection of geodesics from $x$ to all points  $y \in x+ \bbZ^2_+$ form a tree $\sT_x$ rooted
at $x$ and spanning $x+\bbZ_+^2$. The subtree rooted at $x+e_1$ is separated from the subtree rooted at $x+e_2$  by a path $\{\varphi_n^x:n\ge x\cdot\et\}$ on the dual lattice $\etstar+\Z^2$, known as the {\it competition interface}. See Figure \ref{fig:cif1}.

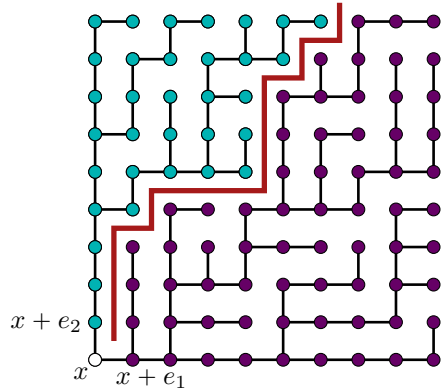
\begin{figure}[H]
\begin{center}

\begin{tikzpicture}[>=latex,scale=0.5]

\draw(-0.37,-0.35)node{$x$};
\draw(-1.3,1)node{$x+e_2$};
\draw(1.5,-0.5)node{$x+e_1$};


\draw[line width=1pt](9,0)--(0,0)--(0,9);     
\draw[line width=1pt](2,0)--(2,4)--(5,4)--(5,7)--(7,7)--(7,9)--(9,9);
\draw[line width=1pt](1,0)--(1,3);
\draw[line width=1pt](0,4)--(1,4)--(1,5)--(2,5)--(2,7);
\draw[line width=1pt](2,5)--(4,5)--(4,6);
\draw[line width=1pt](3,5)--(3,6)--(3,8)--(5,8)--(5,9)--(6,9);
\draw[line width=1pt](3,7)--(4,7);
\draw[line width=1pt](6,7)--(6,8);
\draw[line width=1pt](0,6)--(1,6)--(1,8)--(2,8)--(2,9)--(3,9);
\draw[line width=1pt](4,8)--(4,9);
\draw[line width=1pt](0,9)--(1,9);
\draw[line width=1pt](5,4)--(7,4)--(7,5)--(9,5)--(9,8);
\draw[line width=1pt](6,4)--(6,6)--(7,6);
\draw[line width=1pt](8,5)--(8,8);
\draw[line width=1pt](2,1)--(8,1)--(8,2)--(9,2);
\draw[line width=1pt](2,2)--(4,2)--(4,3)--(6,3);
\draw[line width=1pt](9,0)--(9,1);

\draw[line width=1pt](5,0)--(5,1);
\draw[line width=2pt,color=white](4,1)--(5,1);

\draw[line width=1pt](4,3)--(4,4);
\draw[line width=2pt,color=white](3,4)--(4,4);

\draw[line width=1pt](5,1)--(5,2)--(7,2)--(7,3);
\draw[line width=1pt](8,2)--(8,4)--(9,4);
\draw[line width=1pt](3,2)--(3,3);
\draw[line width=1pt](8,3)--(9,3);


\draw[line width=2.0pt,color=nicosred](.5,.5)--(0.5,3.5)--(1.5,3.5)--(1.5, 4.5)--(4.5,4.5)--(4.5,7.5)--(5.5,7.5)--(5.5,8.5)--(6.5,8.5)--(6.5, 9.5);   
\foreach \x in {0,...,9}{
             \foreach \y in {0,...,9}{
\draw[ fill =sussexp](\x,\y)circle(1.7mm);    
}}

\foreach \x in{0,...,6}{
		\draw[ fill =sussexg](\x, 9)circle(1.7mm);    
				}

\foreach \x in{0,...,5}{
		\draw[ fill =sussexg](\x, 8)circle(1.7mm);    
				}

\foreach \x in{0,...,4}{
		\foreach \y in {5,6,7}{
		\draw[fill =sussexg](\x, \y)circle(1.7mm);    
				}}

\foreach \x in{0,...,1}{
		\draw[ fill =sussexg](\x, 4)circle(1.7mm);    
				}

\foreach \x in{0,...,0}{
		\foreach \y in {1,2,3}{
		\draw[fill =sussexg](\x, \y)circle(1.7mm);    
				}}

\draw[fill=white](0,0)circle(1.7mm); 

\end{tikzpicture}

\end{center}
\caption{\small   The geodesic tree $\mathcal{T}_x$ rooted at $x$.  
The competition interface  (solid line)  emanates from $x+\etstar$ and separates 
the   subtrees of $\mathcal{T}_x$ rooted at $x+e_1$   and $x+e_2$.}
\label{fig:cif1}
\end{figure}

Under condition \eqref{g-reg}, the competition interface satisfies $\varphi_n^x/n \to \cid(T_x\w)$, given by \eqref{cid}. Moreover, each of these two trees contains at least one semi-infinite geodesic with asymptotic direction $\cid(T_x\w)$. Indeed, $\cid(T_x\w)$ is the unique direction with the property that there exist at least two semi-infinite geodesics rooted at $x$, with asymptotic direction $\cid(T_x\w)$, and which differ in their first step. See Figure \ref{fig:cif2}. Theorem \ref{thm-Coupier} records the fact that when the weights are exponentially distributed, there are no directions $\xi$ with three $\xi$-directed geodesics emanating from the same point.

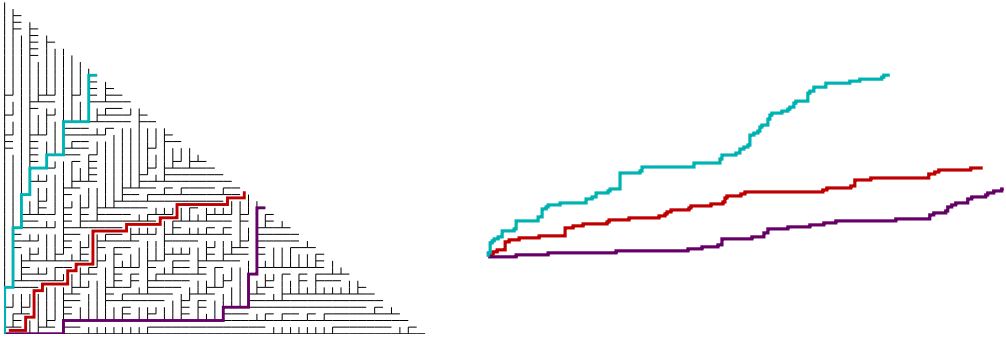
\begin{figure}[H]
\centering
\begin{subfigure}{0.38\textwidth}
\includegraphics[width=\textwidth]{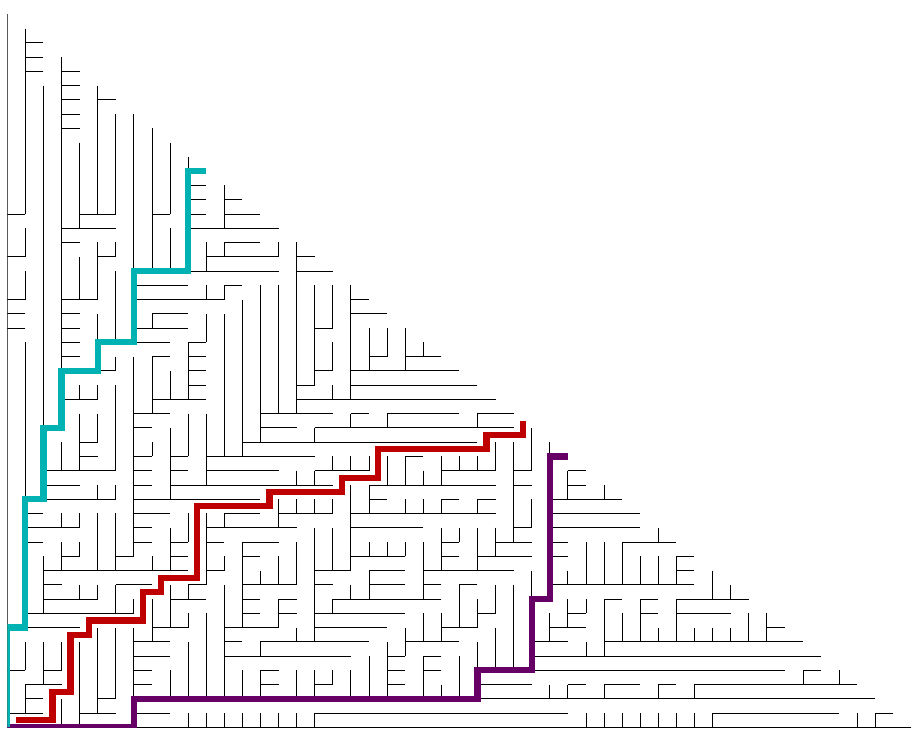}
\end{subfigure}
\begin{subfigure}{0.59\textwidth}
%
%
%
\definecolor{sussexg}{rgb}{0,0.7,0.7}
\definecolor{sussexp}{rgb}{0.4,0,0.4}
\definecolor{nicos-red}{rgb}{0.75,0.0,0.0}
\begin{tikzpicture}[scale=1.2]

\begin{axis}[%
unit vector ratio*=1 1 1,
axis line style={draw=none},
ticks=none,
axis x line*=bottom,
axis y line*=left
]

\addplot [color=nicos-red,  line width=1.0pt, forget plot]
  table[row sep=crcr]{%
0.5	0.5\\
1.5	0.5\\
1.5	1.5\\
2.5	1.5\\
2.5	2.5\\
3.5	2.5\\
4.5	2.5\\
4.5	3.5\\
5.5	3.5\\
6.5	3.5\\
7.5	3.5\\
8.5	3.5\\
8.5	4.5\\
8.5	5.5\\
8.5	6.5\\
8.5	7.5\\
9.5	7.5\\
10.5	7.5\\
10.5	8.5\\
11.5	8.5\\
12.5	8.5\\
13.5	8.5\\
14.5	8.5\\
15.5	8.5\\
15.5	9.5\\
16.5	9.5\\
17.5	9.5\\
18.5	9.5\\
19.5	9.5\\
20.5	9.5\\
21.5	9.5\\
22.5	9.5\\
23.5	9.5\\
24.5	9.5\\
25.5	9.5\\
25.5	10.5\\
26.5	10.5\\
27.5	10.5\\
28.5	10.5\\
29.5	10.5\\
30.5	10.5\\
31.5	10.5\\
32.5	10.5\\
33.5	10.5\\
34.5	10.5\\
35.5	10.5\\
36.5	10.5\\
37.5	10.5\\
38.5	10.5\\
38.5	11.5\\
38.5	12.5\\
38.5	13.5\\
38.5	14.5\\
39.5	14.5\\
40.5	14.5\\
41.5	14.5\\
42.5	14.5\\
42.5	15.5\\
43.5	15.5\\
44.5	15.5\\
45.5	15.5\\
46.5	15.5\\
47.5	15.5\\
47.5	16.5\\
48.5	16.5\\
49.5	16.5\\
50.5	16.5\\
51.5	16.5\\
52.5	16.5\\
53.5	16.5\\
54.5	16.5\\
55.5	16.5\\
56.5	16.5\\
57.5	16.5\\
58.5	16.5\\
59.5	16.5\\
59.5	17.5\\
60.5	17.5\\
60.5	18.5\\
61.5	18.5\\
62.5	18.5\\
63.5	18.5\\
64.5	18.5\\
65.5	18.5\\
66.5	18.5\\
67.5	18.5\\
68.5	18.5\\
69.5	18.5\\
70.5	18.5\\
70.5	19.5\\
71.5	19.5\\
72.5	19.5\\
73.5	19.5\\
74.5	19.5\\
75.5	19.5\\
76.5	19.5\\
77.5	19.5\\
78.5	19.5\\
79.5	19.5\\
80.5	19.5\\
81.5	19.5\\
82.5	19.5\\
83.5	19.5\\
84.5	19.5\\
85.5	19.5\\
85.5	20.5\\
86.5	20.5\\
87.5	20.5\\
88.5	20.5\\
88.5	21.5\\
89.5	21.5\\
89.5	22.5\\
90.5	22.5\\
91.5	22.5\\
91.5	23.5\\
92.5	23.5\\
93.5	23.5\\
94.5	23.5\\
95.5	23.5\\
96.5	23.5\\
97.5	23.5\\
98.5	23.5\\
99.5	23.5\\
100.5	23.5\\
100.5	24.5\\
101.5	24.5\\
101.5	25.5\\
102.5	25.5\\
103.5	25.5\\
104.5	25.5\\
105.5	25.5\\
106.5	25.5\\
107.5	25.5\\
108.5	25.5\\
109.5	25.5\\
110.5	25.5\\
111.5	25.5\\
111.5	26.5\\
112.5	26.5\\
113.5	26.5\\
114.5	26.5\\
115.5	26.5\\
116.5	26.5\\
117.5	26.5\\
117.5	27.5\\
118.5	27.5\\
118.5	28.5\\
118.5	29.5\\
119.5	29.5\\
119.5	30.5\\
120.5	30.5\\
121.5	30.5\\
122.5	30.5\\
123.5	30.5\\
124.5	30.5\\
125.5	30.5\\
126.5	30.5\\
126.5	31.5\\
127.5	31.5\\
128.5	31.5\\
128.5	32.5\\
129.5	32.5\\
130.5	32.5\\
131.5	32.5\\
132.5	32.5\\
133.5	32.5\\
134.5	32.5\\
135.5	32.5\\
136.5	32.5\\
137.5	32.5\\
138.5	32.5\\
139.5	32.5\\
140.5	32.5\\
141.5	32.5\\
142.5	32.5\\
143.5	32.5\\
144.5	32.5\\
145.5	32.5\\
146.5	32.5\\
147.5	32.5\\
148.5	32.5\\
149.5	32.5\\
150.5	32.5\\
151.5	32.5\\
152.5	32.5\\
153.5	32.5\\
154.5	32.5\\
155.5	32.5\\
156.5	32.5\\
157.5	32.5\\
158.5	32.5\\
159.5	32.5\\
160.5	32.5\\
161.5	32.5\\
162.5	32.5\\
163.5	32.5\\
164.5	32.5\\
165.5	32.5\\
166.5	32.5\\
167.5	32.5\\
167.5	33.5\\
168.5	33.5\\
169.5	33.5\\
169.5	34.5\\
170.5	34.5\\
171.5	34.5\\
172.5	34.5\\
173.5	34.5\\
174.5	34.5\\
175.5	34.5\\
176.5	34.5\\
177.5	34.5\\
178.5	34.5\\
179.5	34.5\\
180.5	34.5\\
181.5	34.5\\
181.5	35.5\\
182.5	35.5\\
183.5	35.5\\
183.5	36.5\\
183.5	37.5\\
183.5	38.5\\
184.5	38.5\\
185.5	38.5\\
186.5	38.5\\
187.5	38.5\\
188.5	38.5\\
189.5	38.5\\
190.5	38.5\\
191.5	38.5\\
191.5	39.5\\
192.5	39.5\\
193.5	39.5\\
194.5	39.5\\
195.5	39.5\\
196.5	39.5\\
197.5	39.5\\
198.5	39.5\\
199.5	39.5\\
200.5	39.5\\
201.5	39.5\\
202.5	39.5\\
203.5	39.5\\
204.5	39.5\\
205.5	39.5\\
206.5	39.5\\
207.5	39.5\\
208.5	39.5\\
209.5	39.5\\
210.5	39.5\\
211.5	39.5\\
212.5	39.5\\
213.5	39.5\\
214.5	39.5\\
215.5	39.5\\
216.5	39.5\\
217.5	39.5\\
218.5	39.5\\
219.5	39.5\\
220.5	39.5\\
220.5	40.5\\
220.5	41.5\\
221.5	41.5\\
222.5	41.5\\
223.5	41.5\\
223.5	42.5\\
224.5	42.5\\
225.5	42.5\\
225.5	43.5\\
226.5	43.5\\
227.5	43.5\\
228.5	43.5\\
229.5	43.5\\
230.5	43.5\\
231.5	43.5\\
232.5	43.5\\
233.5	43.5\\
234.5	43.5\\
235.5	43.5\\
236.5	43.5\\
237.5	43.5\\
238.5	43.5\\
239.5	43.5\\
240.5	43.5\\
241.5	43.5\\
241.5	44.5\\
242.5	44.5\\
243.5	44.5\\
244.5	44.5\\
245.5	44.5\\
246.5	44.5\\
247.5	44.5\\
};
\addplot [color=sussexp,  line width=1.0pt, forget plot]
  table[row sep=crcr]{%
0	0\\
1	0\\
2	0\\
3	0\\
4	0\\
5	0\\
6	0\\
7	0\\
8	0\\
9	0\\
10	0\\
11	0\\
12	0\\
13	0\\
14	0\\
14	1\\
15	1\\
16	1\\
17	1\\
18	1\\
19	1\\
20	1\\
21	1\\
22	1\\
23	1\\
24	1\\
25	1\\
26	1\\
27	1\\
28	1\\
29	1\\
30	1\\
30	2\\
31	2\\
32	2\\
33	2\\
34	2\\
35	2\\
36	2\\
37	2\\
38	2\\
39	2\\
40	2\\
41	2\\
42	2\\
43	2\\
44	2\\
45	2\\
46	2\\
47	2\\
48	2\\
49	2\\
50	2\\
51	2\\
52	2\\
53	2\\
54	2\\
55	2\\
56	2\\
57	2\\
58	2\\
59	2\\
60	2\\
61	2\\
62	2\\
63	2\\
64	2\\
64	3\\
65	3\\
66	3\\
67	3\\
68	3\\
69	3\\
70	3\\
71	3\\
72	3\\
73	3\\
74	3\\
75	3\\
76	3\\
77	3\\
78	3\\
79	3\\
80	3\\
81	3\\
82	3\\
83	3\\
84	3\\
85	3\\
86	3\\
87	3\\
88	3\\
89	3\\
90	3\\
91	3\\
92	3\\
93	3\\
94	3\\
95	3\\
96	3\\
97	3\\
98	3\\
99	3\\
100	3\\
100	4\\
101	4\\
102	4\\
103	4\\
104	4\\
105	4\\
106	4\\
107	4\\
107	5\\
108	5\\
109	5\\
110	5\\
111	5\\
112	5\\
113	5\\
114	5\\
115	5\\
115	6\\
116	6\\
117	6\\
117	7\\
117	8\\
117	9\\
118	9\\
119	9\\
120	9\\
121	9\\
122	9\\
123	9\\
124	9\\
125	9\\
126	9\\
127	9\\
128	9\\
129	9\\
130	9\\
131	9\\
132	9\\
132	10\\
133	10\\
134	10\\
135	10\\
136	10\\
137	10\\
138	10\\
138	11\\
138	12\\
139	12\\
140	12\\
140	13\\
140	14\\
141	14\\
142	14\\
143	14\\
144	14\\
145	14\\
146	14\\
147	14\\
148	14\\
149	14\\
150	14\\
150	15\\
151	15\\
152	15\\
153	15\\
154	15\\
155	15\\
156	15\\
157	15\\
158	15\\
159	15\\
160	15\\
161	15\\
161	16\\
162	16\\
163	16\\
164	16\\
165	16\\
166	16\\
167	16\\
168	16\\
168	17\\
169	17\\
170	17\\
171	17\\
172	17\\
173	17\\
174	17\\
174	18\\
175	18\\
176	18\\
177	18\\
178	18\\
179	18\\
180	18\\
181	18\\
182	18\\
183	18\\
184	18\\
185	18\\
186	18\\
187	18\\
188	18\\
189	18\\
190	18\\
191	18\\
192	18\\
193	18\\
194	18\\
195	18\\
196	18\\
197	18\\
198	18\\
199	18\\
200	18\\
201	18\\
202	18\\
203	18\\
204	18\\
204	19\\
205	19\\
206	19\\
207	19\\
208	19\\
209	19\\
210	19\\
211	19\\
212	19\\
213	19\\
214	19\\
215	19\\
216	19\\
217	19\\
218	19\\
219	19\\
220	19\\
221	19\\
221	20\\
221	21\\
222	21\\
223	21\\
223	22\\
224	22\\
225	22\\
226	22\\
227	22\\
228	22\\
229	22\\
229	23\\
230	23\\
230	24\\
230	25\\
231	25\\
232	25\\
232	26\\
233	26\\
234	26\\
234	27\\
235	27\\
236	27\\
237	27\\
238	27\\
239	27\\
240	27\\
240	28\\
240	29\\
241	29\\
242	29\\
242	30\\
243	30\\
244	30\\
245	30\\
246	30\\
247	30\\
248	30\\
249	30\\
249	31\\
250	31\\
250	32\\
251	32\\
252	32\\
253	32\\
253	33\\
254	33\\
255	33\\
256	33\\
257	33\\
257	34\\
258	34\\
};
\addplot [color=sussexg,  line width=1.0pt, forget plot]
  table[row sep=crcr]{%
0	0\\
0	1\\
0	2\\
1	2\\
1	3\\
1	4\\
1	5\\
1	6\\
1	7\\
2	7\\
2	8\\
3	8\\
3	9\\
4	9\\
5	9\\
5	10\\
6	10\\
7	10\\
7	11\\
7	12\\
7	13\\
8	13\\
9	13\\
10	13\\
11	13\\
12	13\\
13	13\\
13	14\\
14	14\\
14	15\\
14	16\\
14	17\\
14	18\\
15	18\\
16	18\\
17	18\\
18	18\\
19	18\\
20	18\\
21	18\\
22	18\\
23	18\\
24	18\\
25	18\\
25	19\\
25	20\\
26	20\\
27	20\\
27	21\\
27	22\\
27	23\\
27	24\\
28	24\\
29	24\\
29	25\\
30	25\\
30	26\\
31	26\\
32	26\\
33	26\\
34	26\\
35	26\\
36	26\\
36	27\\
37	27\\
38	27\\
39	27\\
40	27\\
41	27\\
42	27\\
43	27\\
44	27\\
45	27\\
46	27\\
47	27\\
48	27\\
49	27\\
49	28\\
49	29\\
50	29\\
51	29\\
52	29\\
52	30\\
53	30\\
54	30\\
54	31\\
54	32\\
55	32\\
56	32\\
57	32\\
58	32\\
59	32\\
60	32\\
60	33\\
61	33\\
61	34\\
62	34\\
63	34\\
64	34\\
65	34\\
66	34\\
66	35\\
66	36\\
66	37\\
66	38\\
66	39\\
66	40\\
66	41\\
66	42\\
67	42\\
68	42\\
69	42\\
70	42\\
71	42\\
72	42\\
73	42\\
74	42\\
75	42\\
76	42\\
76	43\\
77	43\\
77	44\\
77	45\\
78	45\\
79	45\\
80	45\\
81	45\\
82	45\\
83	45\\
84	45\\
85	45\\
86	45\\
87	45\\
88	45\\
89	45\\
90	45\\
91	45\\
92	45\\
93	45\\
94	45\\
95	45\\
96	45\\
97	45\\
98	45\\
99	45\\
100	45\\
101	45\\
102	45\\
103	45\\
103	46\\
103	47\\
104	47\\
105	47\\
106	47\\
107	47\\
108	47\\
109	47\\
110	47\\
111	47\\
112	47\\
113	47\\
114	47\\
115	47\\
116	47\\
116	48\\
116	49\\
117	49\\
117	50\\
117	51\\
118	51\\
119	51\\
120	51\\
121	51\\
122	51\\
123	51\\
124	51\\
124	52\\
125	52\\
126	52\\
126	53\\
126	54\\
127	54\\
128	54\\
128	55\\
129	55\\
130	55\\
130	56\\
131	56\\
131	57\\
131	58\\
131	59\\
131	60\\
131	61\\
132	61\\
133	61\\
133	62\\
134	62\\
135	62\\
135	63\\
136	63\\
136	64\\
136	65\\
137	65\\
137	66\\
138	66\\
139	66\\
140	66\\
140	67\\
140	68\\
140	69\\
141	69\\
141	70\\
141	71\\
142	71\\
142	72\\
143	72\\
144	72\\
145	72\\
146	72\\
147	72\\
147	73\\
148	73\\
149	73\\
150	73\\
150	74\\
151	74\\
151	75\\
152	75\\
153	75\\
153	76\\
153	77\\
154	77\\
154	78\\
155	78\\
156	78\\
157	78\\
158	78\\
159	78\\
160	78\\
160	79\\
160	80\\
160	81\\
160	82\\
161	82\\
161	83\\
162	83\\
163	83\\
163	84\\
163	85\\
164	85\\
165	85\\
166	85\\
167	85\\
168	85\\
169	85\\
169	86\\
169	87\\
170	87\\
171	87\\
172	87\\
173	87\\
174	87\\
175	87\\
176	87\\
177	87\\
178	87\\
179	87\\
180	87\\
181	87\\
181	88\\
182	88\\
183	88\\
184	88\\
185	88\\
186	88\\
186	89\\
187	89\\
188	89\\
189	89\\
190	89\\
191	89\\
192	89\\
193	89\\
194	89\\
195	89\\
196	89\\
197	89\\
197	90\\
198	90\\
198	91\\
199	91\\
200	91\\
201	91\\
};
\end{axis}
\end{tikzpicture}%
\end{subfigure}
\caption{\small  The competition interface (middle path) separating the two $\cid$-directed geodesics. The left picture is a small portion of the right one. In the picture on the right the $x$-axis appears to be stretched, but the scales of the axes are in fact identical.}
\label{fig:cif2}
\end{figure}


\section{Busemann measures, exceptional directions, and coalescence points}\label{s:Bmeas} 
The central theme of   this paper is the relationship between analytic properties of the Busemann process and the geometric properties of the geodesics $\geo{}{\abullet}{\xi\sig}$ for   $\xi \in \ri \Uset$ and $\sigg\in\{-,+\}$. 
It will be convenient in what follows to have a bookkeeping tool for the locations at which the Busemann processes are not locally constant. A natural way to record this information is through the supports of the associated Lebesgue-Stieltjes measures.

%
As functions of the direction parameter $\xi$, $\B{\xi-}_{x,x+e_i}$  and $\B{\xi+}_{x,x+e_i}$ are respectively left- and right-continuous versions of the same monotone function and satisfy the cocycle property \eqref{coc-prop}. As a consequence, 
for each $x,y\in\Z^2$, $\sigg\in\{-,+\}$, $\xi\mapsto\B{\xi\sig}(x,y)$ has locally bounded total variation.  
Hence on each compact subset  $K$ of $\ri \Uset$ there exists a signed Lebesgue-Stieltjes measure  $\mu_{x,y}^{K}$  with the property that whenever 
$\zeta \prec \eta$ and $[\zeta, \eta]\subset K$,
\be\label{Bmeas}\begin{aligned}
\mu_{x,y}^{K}(\,]\zeta, \eta])  = \B{\eta+}_{x,y} - \B{\zeta +}_{x,y}
\quad\text{and}\quad
\mu_{x,y}^{K}([\zeta , \eta [\,)= \B{\eta-}_{x,y} - \B{\zeta -}_{x,y}.
\end{aligned}\ee

The restriction to compact sets is a technical point: in general, $\B{\xi+}_{x,y}$ and $\B{\xi-}_{x,y}$  are signed sums of monotone functions and thus correspond to formal linear combinations of positive measures. 
By the limit in \eqref{B:inf-lim}, 
each of these positive measures assigns infinite mass to the interval $\ri\Uset$ and if any two of the measures come with different signs, the formal linear combination will not define a signed measure on all of $\ri\Uset$. We will ignore this technical point in what follows and write $\mu_{x,y}(\bbullet)$ for the value of this measure and $|\mu_{x,y}|(\bbullet)$ for the value of the total variation measure whenever they are unambiguously defined. In that vein, we define the \emph{support of the measure $\mu_{x,y}$} on $\ri\Uset$ as 
\begin{align}
\supp{\mu_{x,y}} &= \bigcup_{\zeta,\,\eta\,\in\,\ri\Uset: \, \zeta\prec\eta}
\supp{\mu_{x,y}^{[\zeta,\eta]}},
\end{align}
where $\supp{\mu_{x,y}^{[\zeta,\eta]}}$ is, as usual, the 
support of the (well-defined) total variation measure $\abs{\mu_{x,y}^{[\zeta,\eta]}}$. Naturally, this definition agrees with the standard notion of the support of a measure when $\mu_{x,y}$ is a well-defined positive or negative measure on $\Uset$. 

\subsection{Coalescence and the Busemann measures}

The first  result below relates membership in the support with the existence of disjoint Busemann geodesics. 
\begin{theorem}\label{thm:nonint}
The following  holds with $\P$-probability one. For all $x\ne y$ in $\Z^2$ and $\xi\in\ri\Uset$  statements \eqref{thm:nonint.a} and \eqref{thm:nonint.b} below are equivalent: 
\begin{enumerate}[label={\rm(\roman*)}, ref={\rm\roman*}] \itemsep=3pt 
\item\label{thm:nonint.a} $\xi \in \supp{\mu_{x,y}}$.
\item\label{thm:nonint.b}  Either $\geo{}{x}{\xi-} \cap \geo{}{y}{\xi+} = \varnothing$ or $\geo{}{x}{\xi+} \cap \geo{}{y}{\xi-} = \varnothing$.
\end{enumerate}
Under the regularity condition \eqref{g-reg}, \eqref{thm:nonint.a} and \eqref{thm:nonint.b} are equivalent to
\begin{enumerate}[label={\rm(\roman*)}, ref={\rm\roman*},resume] \itemsep=3pt 
\item\label{thm:nonint.c}There exist $\Uset_\xi$-directed semi-infinite geodesics $\geod{}^x$ and $\geod{}^y$ out of $x$ and $y$, respectively,
such that $\geod{}^x \cap \geod{}^y = \varnothing$.
\end{enumerate}
\end{theorem}

The difference between statements \eqref{thm:nonint.b} and \eqref{thm:nonint.c} is that if $\xi\not\in\supp{\mu_{x,y}}$ then \eqref{thm:nonint.b} leaves open the possibility that even though $\geo{}{x}{\xi-}$ and $\geo{}{y}{\xi+}$ intersect and   $\geo{}{x}{\xi+}$ and $\geo{}{y}{\xi-}$ intersect,  there may be other $\Uset_\xi$-directed geodesics out of $x$ and $y$  that do not intersect. This is because without the regularity condition \eqref{g-reg}, we currently do not know whether  \eqref{eq:geoorder} holds, that is, whether $\geo{}{x}{\xi+}$ is 
the rightmost and $\geo{}{x}{\xi-}$ the  leftmost $\Uset_\xi$-directed geodesic out of $x$.

The subsequent  several results relate the support of Busemann measures to the coalescence geometry of geodesics.   For $x,y\in\Z^2$,   $\xi\in\ri\Uset$, and signs   $\sigg\in\{-,+\}$, define 
  the {\it coalescence point}  of  the geodesics $\geo{}{x}{\xi\sig}$ and $\geo{}{y}{\xi\sig}$
by 
\be\label{df:coal-z} \begin{aligned}
\coal{\xi\sig}(x,y) &= \begin{cases}
\text{first point in $\geo{}{x}{\xi\sig}\cap \geo{}{y}{\xi\sig}$}, & \text{if } \geo{}{x}{\xi\sig}\cap \geo{}{y}{\xi\sig} \neq \emptyset \\[3pt] 
\infty, &  \text{if } \geo{}{x}{\xi\sig}\cap \geo{}{y}{\xi\sig}= \emptyset.  
\end{cases}
\end{aligned}\ee
The  first point $z$ in $\geo{}{x}{\xi\sig}\cap \geo{}{y}{\xi\sig}$ is identified uniquely by choosing  the common point $z=\geo{k}{x}{\xi\sig}=\geo{k}{y}{\xi\sig}$ that  minimizes $k$. 
In the expression above,  $\infty$ is the point added in the one-point compactification of $\Z^2$. 
If the two geodesics  $\geo{}{x}{\xi\sig}$ and  $\geo{}{y}{\xi\sig}$ ever meet, they coalesce due to the local rule in \eqref{d:bgeo}. We write $\coal{\xi}{(x,y)}$ when $\coal{\xi-}(x,y)=\coal{\xi+}(x,y)$. 

As $\Z^2\cup\{\infty\}$-valued functions,  $\xi\mapsto\coal{\xi+}(x,y)$ is right-continuous and $\xi\mapsto\coal{\xi-}(x,y)$ is left-continuous.  Namely, 
a consequence of \eqref{pmgeolim} is that for $\xi \in \ri \Uset$ and $\sigg\in\{-,+\}$, 
\begin{align}\label{coal-lim}
\lim_{\ri\Uset\,\ni\,\eta \searrow \xi} \coal{\eta \sig}(x,y) = \coal{\xi+}(x,y). 
\end{align}
If $\coal{\xi+}(x,y) = \infty$ this limit still holds in the sense that then $\abs{\coal{\eta \sig}(x,y)}\to\infty$.   The analogous statement holds for convergence from the left to  $\coal{\xi-}(x,y)$. 

The next theorem states that an interval of directions outside the support of a Busemann measure corresponds to geodesics following common initial segments to a common coalescence point. 

\begin{theorem}\label{thm:1path}   With probability one,   simultaneously  for all $\zeta\prec\eta$ in $\ri\Uset$ and all $x,y\in\Z^2$,  statements  \eqref{thm:1path.a}, \eqref{thm:1path.b}, and \eqref{thm:1path.c} below are equivalent:   
\begin{enumerate}[label={\rm(\roman*)}, ref={\rm\roman*}] \itemsep=3pt 
\item\label{thm:1path.a} $\abs{\mu_{x,y}}(\,]\zeta, \eta[\,)=0$.
\item\label{thm:1path.b}   Letting $k=x\cdot\et$ and $\ell=y\cdot\et$,  there exist a point $z$ with $z\cdot\et=m\ge k\vee\ell$  and path segments $\pi_{k,m}$ and $\wt\pi_{\ell,m}$  with these properties:   $\pi_k=x$, $\wt\pi_\ell=y$,  $\pi_m=\wt\pi_m=z$, and  for all $\xi\in\,]\zeta, \eta[$ and $\sigg\in\{-,+\}$ we have   $\geo{k,m}{x}{\xi\sig}=\pi_{k,m}$ and $\geo{\ell,m}{y}{\xi\sig}=\wt\pi_{\ell,m}$.  
\item\label{thm:1path.c} Letting $k=x\cdot\et$ and $\ell=y\cdot\et$,  there exists a point $z$ with $z\cdot\et=m\ge k\vee\ell$ so that for all $\xi\in\,]\zeta, \eta[$ and $\sigg\in\{-,+\}$, $\coal{\xi\sig}(x,y)=z$.
\end{enumerate}
\end{theorem} 

 
 The next lemma verifies that intervals that satisfy statement \eqref{thm:1path.a} of Theorem \ref{thm:1path}  almost surely make up a  random  dense open subset of $\ri\Uset$. 

\begin{lemma}\label{lm:isolated}   
 Let $\Udense\subset\ri\Uset$ be a fixed countable dense set of points of differentiability of $\gpp$.
Then the  following holds with $\P$-probability one: 
for every $x,y\in\Z$ and every $\xi\in\Udense$, there exist $\zeta\prec\xi\prec\eta$ in $\ri\Uset$ such that  $\abs{\mu_{x,y}}(\,]\zeta, \eta[\,)=0$. 
\end{lemma}



A  natural question is whether the measure is Cantor-like with no isolated points of support, or   if the support   consists entirely of isolated points, or if both are possible.  These features also turn out to have counterparts in coalescence properties. 
  For a set $\cA\subset\Uset$  say that $\xi$ is a {\it  limit point of $\cA$ from the right} 
if $\cA$ intersects $]\xi, \eta[$ for each $\eta\succ\xi$, with a similar definition for limit points from the left.

\begin{theorem}\label{thm:supp-tri}
The following  holds with probability one. For all $x,y \in \bbZ^2$ and $\xi \in \ri \Uset$:
\begin{enumerate} [label={\rm(\alph*)}, ref={\rm\alph*}]   \itemsep=3pt 
\item\label{thm:supp-tri:1} $\xi\notin\supp{\mu_{x,y}}$ $\Longleftrightarrow$ $\coal{\xi+}(x,y)=\coal{\xi-}(x,y) \in \bbZ^2$.
\item\label{thm:supp-tri:2} $\xi$ is an isolated point of $\supp{\mu_{x,y}}$ $\Longleftrightarrow$ $\coal{\xi+}(x,y) \neq \coal{\xi-}(x,y)$ but both $\coal{\xi\pm}(x,y) \in \bbZ^2$.
\item\label{thm:supp-tri:3} $\xi$ is a limit point of $\supp{\mu_{x,y}}$ from the right $\Longleftrightarrow$ $\coal{\xi+}(x,y)=\infty$.  Similarly,  $\xi$ is a limit point of $\supp{\mu_{x,y}}$ from the left  $\Longleftrightarrow$ $\coal{\xi-}(x,y)=\infty$. 
 \end{enumerate}
\end{theorem}

%
%

 This motivates  
the following condition on the Busemann process   which  will be invoked in some  results in the sequel: 
\begin{align}\label{cond:jumpcond}
\begin{minipage}{0.9\textwidth}
There exists a full $\P$-probability event on which every point of $\supp{\mu_{x,y}}$ is isolated, for all $x,y\in\Z^2$. 
\end{minipage}
\end{align}\smallskip
Equivalently, condition \eqref{cond:jumpcond} says that   $\xi \mapsto \B{\xi\pm}(x,y)$ is a  jump process whose jumps do not accumulate on $\ri \Uset$. For this reason, we  refer to \eqref{cond:jumpcond} as the  {\it jump process condition}.
It is shown in \cite[Theorem 3.4]{Fan-Sep-20} that \eqref{cond:jumpcond} holds 
when the weights $\w_x$ are i.i.d.\ exponential random variables.  
In addition to Lemma \ref{lm:isolated}, this is a further  reason to expect that \eqref{cond:jumpcond} holds very generally. 

Under condition \eqref{cond:jumpcond}  Theorem \ref{thm:supp-tri} extends to a global  coalescence statement.  
 
\begin{theorem}\label{thm:+-coal}
Statements \eqref{thm:+-coal.a} and  \eqref{thm:+-coal.b} below are equivalent.
\begin{enumerate}[label={\rm(\roman*)}, ref={\rm\roman*}] \itemsep=3pt 
\item\label{thm:+-coal.a} The jump process condition \eqref{cond:jumpcond} holds.
\item\label{thm:+-coal.b} This holds with $\P$-probability one: for all $x,y\in\Z^2$, all $\xi\in\ri\Uset$, and both signs $\sigg\in\{-,+\}$, the geodesics $\geo{}{x}{\xi\sig}$ and $\geo{}{y}{\xi\sig}$ coalesce.
\end{enumerate}
\end{theorem}


We introduce  the random  set of exceptional directions obtained by taking the union of the supports of  the Busemann measures:
\begin{align} 
\aUset &= \bigcup_{x,\,y\,\in\,\bbZ^2} \supp{\mu_{x,y}} \;\subset\;\ri\Uset. \label{df:aUset}
\end{align}
It turns out that not all pairs $x,y$ are necessary for the union.  It suffices  to take  pairs of adjacent points  along horizontal or vertical lines, or along  any  bi-infinite path  with nonpositive local slopes.  
\begin{lemma}\label{lm:downright}
The following holds for $\P$-almost every $\w$.  
Let $x_{-\infty,\infty}$ be any bi-infinite path in $\Z^2$ such that  $\forall i\in\Z$,  $(x_{i+1}-x_i)\cdot e_1\ge 0$ and $(x_{i+1}-x_i)\cdot e_2\le 0$ and not both are zero. Then 
	\[\aUset= \bigcup_{i\in\Z} \supp\mu_{x_i,x_{i+1}}.\]
\end{lemma}


  The remainder of this section addresses   (i) characterizations of $\aUset$ and  (ii) its significance for  uniqueness and coalescence of geodesics. 
  The first item relates the exceptional directions to asymptotic directions of competition interfaces. 
  
  \begin{theorem}\label{thm:Vcid} 
  The following hold for $\P$-almost every $\w$.
\begin{enumerate}[label={\rm(\alph*)}, ref={\rm\alph*}] \itemsep=3pt 
\item\label{thm:Vcid.1} 
 For   all $x\in\Z^2$, $  \supp \mu_{x,x+e_1} \cap \supp \mu_{x,x+e_2}=\{\cid(T_x\w)\} $.   In particular, 
 $\aUset\supset\{\cid(T_x\w): x\in\Z^2\}$.  
   \item\label{thm:Vcid.2} 
 Under the jump process condition \eqref{cond:jumpcond},  
$\aUset=\{\cid(T_x\w): x\in\Z^2\}$. 
\end{enumerate}

  \end{theorem}

The next issue is the relationship between   $\aUset$ and  regularity properties of $\gpp$.   Recall the definition \eqref{df:D} of $\Diff$ as the set of differentiability points of $\gpp$.   Let $\fUset$ be the subset of $\ri\Uset$ that remains after removal of all open linear segments of $\gpp$ and removal of  those endpoints of linear segments that are differentiability points.  Equivalently, $\fUset$ consists of those $\xi\in\ri\Uset$ at which $\gpp$ is either non-differentiable or strictly concave. 

\begin{theorem}\label{th:V1}  $ $ 
\begin{enumerate}[label={\rm(\alph*)}, ref={\rm\alph*}] \itemsep=3pt 
\item\label{th:V1.a} Let  $\xi\in\ri\Uset$. Then  $\xi\in\Diff$ if and only if $\P(\xi\in\aUset)=0$. 
If  $\xi\notin\Diff$ then 
\[  \P(\exists x:  \,\cid(T_x\w) =\xi) = \P(\xi\in\aUset)= 1.  \]  

\item\label{th:V1.b} For $\P$-almost every $\w$, the set  $\{\cid(T_x\w): x\in\Z^2\}$ and the set  $\aUset$ are  dense subsets of $\fUset$. 

\end{enumerate}
\end{theorem}


%

The next  theorem identifies  $\aUset$ as   the set of
directions with multiple semi-infinite geodesics.  
As before,  the regularity condition \eqref{g-reg} allows us to talk about general $\Uset_\xi$-directed semi-infinite geodesics, instead of only the Busemann geodesics $\geo{}{x}{\xi\sig}$.  


\begin{theorem}\label{th:V2}
The following hold for $\P$-almost every $\w$.
\begin{enumerate}[label={\rm(\alph*)}, ref={\rm\alph*}] \itemsep=3pt 
\item\label{th:V2.a}  
$\xi\in(\ri\Uset)\setminus\aUset$ if and only if the following is true: $\geo{}{x}{\xi+}=\geo{}{x}{\xi-}$ for all $x\in\Z^2$ and all these geodesics coalesce.
\item\label{th:V2.b} Under the regularity condition \eqref{g-reg}, $\xi\in(\ri\Uset)\setminus\aUset$ if and only if the following is true: there exists a unique $\Uset_\xi$-directed semi-infinite geodesic out of every $x\in\Z^2$ and all these geodesics coalesce.
\item\label{th:V2.c}  Under  the jump process condition \eqref{cond:jumpcond} the existence of $x\in\Z^2$ such that $\geo{}{x}{\xi+}=\geo{}{x}{\xi-}$ 
implies that $\geo{}{y}{\xi+}=\geo{}{y}{\xi-}$ for all $y\in\Z^2$,  all these geodesics coalesce, and $\xi\in(\ri\Uset)\setminus\aUset$.
\item\label{th:V2.d} Assume  both the regularity condition \eqref{g-reg} and the jump process condition \eqref{cond:jumpcond}.  Suppose there exists   $x\in\Z^2$  such that $\geo{}{x}{\xi+}=\geo{}{x}{\xi-}$.  Then  there is a unique $\Uset_\xi$-directed semi-infinite geodesic out of every $x\in\Z^2$,  all these geodesics coalesce, and $\xi\in(\ri\Uset)\setminus\aUset$.   
\end{enumerate}
\end{theorem}

By the uniqueness of finite geodesics, 
two   geodesics emanating from the same site $x$ cannot intersect after they separate. Consequently, non-uniqueness of semi-infinite directed geodesics implies the existence of non-coalescing semi-infinite directed geodesics.   
 When both  conditions \eqref{g-reg} and \eqref{cond:jumpcond} hold, Theorem \ref{th:V2}\eqref{th:V2.d}  
shows the converse:   uniqueness     implies   coalescence. 

We close this section with a theorem that collects those previously established properties of geodesics which hold when both the regularity condition \eqref{g-reg} and the jump process condition \eqref{cond:jumpcond} are in force. 
Lemma \ref{lm:capst-aux} justifies that the geodesics in part \eqref{th:capst.d} are $\xi$-directed rather than merely   $\Uset_\xi$-directed.  
 
\begin{theorem}\label{th:capst}   Assume  the regularity condition \eqref{g-reg} and the jump process condition \eqref{cond:jumpcond}.   The following hold for $\P$-almost every $\w$.
\begin{enumerate}     [label={\rm(\alph*)}, ref={\rm\alph*}]   \itemsep=3pt  
\item\label{th:capst.a} $\xi \in \aUset$ if and only if there exist $x,y \in \bbZ^2$ with $\B{\xi-}(x,y) \neq \B{\xi+}(x,y)$.
\item\label{th:capst.b} $\xi \in \aUset$ if and only if there exists $x \in \bbZ^2$ such that $\xi = \cid(\T_x\w)$.
\item\label{th:capst.c} If $\xi \in (\ri\Uset) \backslash \aUset$, then for each $x \in \bbZ^2$,    
$\geo{}{x}{\xi}=\geo{}{x}{\xi-}=\geo{}{x}{\xi+}$ and this   is the  unique  $\Uset_\xi$-directed semi-infinite geodesic out of $x$. 
For any $x,y \in \bbZ^2$, $\geo{}{x}{\xi}$ and $\geo{}{y}{\xi}$ coalesce.
\item\label{th:capst.d} If   $\xi \in \aUset$, then from   each  $x \in \bbZ^2$  there exist at least two   $\xi$-directed semi-infinite  geodesics   that  separate eventually, namely  $\geo{}{x}{\xi-}$ and $\geo{}{x}{\xi+}$. 
  For each pair $x,y \in \bbZ^2$,  $\geo{}{x}{\xi-}$ and $\geo{}{y}{\xi-}$ coalesce and $\geo{}{x}{\xi+}$ and $\geo{}{y}{\xi+}$ coalesce. 
\end{enumerate}

\end{theorem} 

\smallskip


\subsection{Exponential case}
We   specialize to the case where 
	\begin{align}\label{exp-assump}
	\text{$\{\w_x : x \in \bbZ^2\}$ are i.i.d.\ mean one exponential random variables.}
	\end{align} 
 Rost's classic result \cite{Ros-81} gives the shape function 
	\begin{align}\label{exp-shape}
	\gpp(\xi)=
	\bigl(\sqrt{\xi\cdot e_1}+\sqrt{\xi\cdot e_2}\,\bigr)^2,\quad\xi\in\R_+^2.
	\end{align}
 The regularity condition \eqref{g-reg} is satisfied as  $\gpp$ is strictly concave and differentiable on $\ri \Uset$.   The supports $\supp\mu_{x,y}$ are unions of inhomogeneous Poisson processes and hence   the jump process condition \eqref{cond:jumpcond} is satisfied.   This comes from  \cite[Theorem 3.4]{Fan-Sep-20} and is described in Section \ref{sub:palm} below. These two observations imply that the conclusions of Theorem \ref{th:capst} hold. With some additional work, we can go beyond the conclusions of Theorem \ref{th:capst} in this solvable setting.

Let $\spt{\xi}{x}$ denote  the location where the $\xi+$ and $\xi-$ geodesics out of $x$ split:
\begin{align}
\spt{\xi}{x} &= \begin{cases}
\text{last point in  } \geo{}{x}{\xi-} \cap \geo{}{x}{\xi+} & \text{ if }\geo{}{x}{\xi-} \neq \geo{}{x}{\xi+}, \\
\infty & \text{ if } \geo{}{x}{\xi-} = \geo{}{x}{\xi+}.
\end{cases}
\end{align}

For part \eqref{th:exp1.c} in the next theorem, recall the finite geodesic $\gamma^{x,y}$ defined below \eqref{G1} and the  competition interface path $\varphi^x$ introduced in Section \ref{s:mult-geo}.   Convergence of paths means that any finite segments eventually coincide. 
  
\begin{theorem}\label{th:exp1}
Assume \eqref{exp-assump}.  Then the conclusions of Theorem \ref{th:capst} hold with $\Uset_\xi=\{\xi\}$ for all $\xi \in \ri \Uset$.  Additionally,  the following  hold $\P$-almost surely. 
\begin{enumerate}     [label={\rm(\alph*)}, ref={\rm\alph*}]   \itemsep=3pt  
\item\label{th:exp1.a} If  $\xi \in \aUset$ then from  each  $x \in \bbZ^2$ there emanate exactly two semi-infinite $\xi$-directed  geodesics that eventually separate, namely  $\geo{}{x}{\xi-}$ and $\geo{}{x}{\xi+}$. 
\item\label{th:exp1.b} For any $\xi \in \ri \Uset$ and any three $\xi$-directed semi-infinite geodesics rooted at any three points, at least two of the geodesics coalesce.
\item\label{th:exp1.c} Let  $x\in\Z^2$, $\xi\in\ri\Uset$, and let $\{v_n\}_{n\ge m} $ be any sequence on $\Z^2$ such that $v_n\cdot\et=n$ and  $v_n/n \to \xi$.  If  $v_n \prec \varphi_n^{\spt{\xi}{x}}$ for all sufficiently large $n$, then $\gamma^{x,v_n} \to \gamma^{x,\xi-}$ as $n\to\infty$. If $\varphi_n^{\spt{\xi}{x}} \prec v_n$ for all sufficiently large $n$ then $\gamma^{x,v_n} \to \gamma^{x,\xi+}$ as $n\to\infty$.
\item\label{th:exp1:d} For each $x \in \bbZ^2$, the entire collection of semi-infinite geodesics emanating from $x$ is exactly $\bigl\{\geo{}{x}{e_1},\geo{}{x}{e_2},\geo{}{x}{\xi \sigg} : \xi \in \ri \Uset, \sigg \in \{+,-\}\bigr\}$. 
\end{enumerate}
\end{theorem} 


Theorem \ref{th:exp1} resolves a number of previously open problems on the geometry of geodesics in the exponential model. It shows that in all but countably many exceptional directions, the collection of geodesics with that asymptotic direction  coalesce and form a tree. These exceptional directions are identified both with the directions of discontinuity of the Busemann process and the asymptotic directions of competition interfaces.   Moreover, in each exceptional direction $\xi\in \aUset$, ahead of each lattice site $x$, there is a $\xi$-directed  competition interface at which the $\xi-$ and $\xi+$ geodesics out of $x$ split. These are the only two $\xi$-directed geodesics rooted at $x$. Strikingly, each of the two families of $\xi-$ and $\xi+$ geodesics has the same structure as the collection of  geodesics in a typical direction: each family forms a tree of coalescing semi-infinite paths.


 Theorem \ref{th:exp1} utilizes  Theorem \ref{thm-Coupier}, due to Coupier \cite{Cou-11}, that rules out three geodesics that have the same direction,  emanate from a common vertex, and eventually separate.    It appears that the modification argument of \cite{Cou-11} cannot rule out three non-coalescing geodesics from distinct roots, and so  Theorem \ref{th:exp1}\eqref{th:exp1.b}   significantly extends Theorem \ref{thm-Coupier}. 
 
Finally, Theorem \ref{th:exp1} gives a complete description of the coalescence structure of finite geodesics to semi-infinite geodesics in the exponential model. Part \eqref{th:exp1.c} says that if we consider a sequence of lattice sites $v_n$ with asymptotic direction $\xi$, then the geodesic from $x$ to $v_n$ will converge to the $\xi-$ geodesic out of $x$ if and only if $v_n$ eventually stays to the left of the competition interface emanating from the site $\spt{\xi}{x}$ where the $\xi-$ and $\xi+$ geodesics out of $x$ separate. Similarly it will converge to the $\xi+$ geodesic if and only if it stays to the right of that path. The competition interface lives on the dual lattice, so for large $n$ every point  $v_n$ is either to the left or to the right of the competition interface. The coalescence structure of semi-infinite geodesics to arbitrary sequences $v_n$ with $v_n/n \to \xi$ then follows by passing to subsequences.

 The results of Section \ref{s:Bmeas}  are proved in Section \ref{sec:bus}, except   Lemma \ref{lm:downright} which is proved  at the end of Section \ref{shockgraphs:pfs}.


\section{Last-passage percolation as a dynamical system} \label{sec:LPPDS}

After the general description of uniqueness and coalescence of Section \ref{s:Bmeas}, we take a closer look at the spatial structure of the set of lattice points where particular values or ranges of values from the set  $\aUset$ of exceptional directions appear. (Recall its definition \eqref{df:aUset}.)  
As mentioned in the introduction, there is a connection to instability in noise-driven conservation laws.  The next section explains this  point of view.  

\subsection{Discrete Hamilton-Jacobi equations}\label{s:HJ}
We take a dynamical point of view of LPP.   Time proceeds in the negative diagonal direction $-\et=-e_1-e_2$ and the spatial axis is $\ex=e_2-e_1$.  For each $t\in\Z$, the spatial level at time $t$ is $\level_{t}=\{x\in\Z^2: x\cdot \et=t\}$.    For $x\in\Z^2$ and $A\subset\Z^2$ let $\Pi_x^A$ denote the set of up-right paths $\pi_{k,m}$ such  that  $\pi_k=x$ and $\pi_m\in A$, where $k=x\cdot\et$ and $m$ is any integer $\ge k$ such that $A\cap\level_m\ne\varnothing$.  For each $\xi\in\ri\Uset$ and sign $\sigg\in\{-,+\}$,    the Busemann function $\B{\xi\sig}$ satisfies the following equation:   for all $t\le t_0$ and $x\in\level_{t}$, 
\be\label{hj8} 
\B{\xi\sig}(x,0)= \max\Bigl\{ \, \sum_{i=t}^{t_0-1} \w_{\pi_i} +  \B{\xi\sig}(\pi_{t_0}, 0)  :  \pi\in\Pi_x^{\level_{t_0}} \Bigr\} . 
\ee
The unique   maximizing path  in \eqref{hj8} is the geodesic segment  $\geo{t,t_0}{x}{\xi\sig}$.  

Equation  \eqref{hj8} can be viewed as a discrete Hopf-Lax-Oleinik semigroup.   
For example, equation \eqref{hj8} is an obvious  discrete analogue of the variational formula (1.3) of \cite{Bak-Cat-Kha-14}.  At first blush  the two formulas appear different because  (1.3) of \cite{Bak-Cat-Kha-14} contains a  kinetic energy term.  However, this term is not needed in \eqref{hj8} above  because all admissible steps are of size one and all paths between  levels $\level_{t}$ and $\level_{t_0}$ have equal length (number of steps).   

Through this analogy with a Hopf-Lax-Oleinik semigroup we can regard  $\B{\xi\sig}(\bbullet\,,0)$   as a global solution of a discrete stochastic Hamilton-Jacobi equation started in the infinite past ($t_0\to\infty$) and driven by the noise $\w$.  The spatial difference $\B{\xi\sig}(x+e_1, x+e_2)=\B{\xi\sig}(x+e_1, 0)-\B{\xi\sig}(x+e_2,0)$ can then be viewed as a global solution of a discretized stochastic Burgers equation.  


By Lemma \ref{lem:buselln}, if $\gpp$ is differentiable on $\ri\Uset$, then  
$\B{\xi+}$ and $\B{\xi-}$ both satisfy for each $x \in \bbZ^2$
\begin{align*}
\lim_{\abs{n}\to\infty} \frac{\B{\xi\pm}\left(x,x+n\ex \right)}{n} &= \nabla \gpp(\xi)\cdot\ex.
\end{align*} 
Thus, $\B{\xi\pm}$ are two solutions with the same value of the conserved quantity.
Under the jump process condition \eqref{cond:jumpcond}, $\xi \in \supp \mu_{x+e_1,x+e_2}$ if and only if $\B{\xi+}(x+e_1,x+e_2) \neq \B{\xi-}(x+e_1,x+e_2)$.  
This means that the locations $x$ where $\xi \in \supp \mu_{x+e_1,x+e_2}$ are precisely the space-time points at which the two solutions $\B{\xi\pm}$ differ. 
It is reasonable to expect then that these points are locations of instability in the following sense. The spatial differences of the solution to the stochastic Hamilton-Jacobi equation started at time $t_0$ with a linear initial condition dual to $\xi$,
	\[\max\{G_{x+e_1,y}-y\cdot\nabla\gpp(\xi):y\cdot\et=t_0\}-\max\{G_{x+e_2,y}-y\cdot\nabla\gpp(\xi):y\cdot\et=t_0\},\]
has at least two limit points $\B{\xi\pm}(x+e_1,x+e_2)+(e_1-e_2)\cdot\nabla\gpp(\xi)$ as $t_0\to\infty$.  
This is supported by simulations and is hinted at by Theorem \ref{th:exp1}\eqref{th:exp1.c}.

With these points in mind, we now define what we mean by instability points and then turn to studying their geometric structure. Proofs of the results of this section appear in Section \ref{sec:webpf}.

\subsection{Webs of instability}\label{sec:web} 
For a direction  $\xi \in \ri \Uset$ and a sign $\sigg\in\{-,+\}$,   let $\G{\xi\sig}$ be the directed graph whose vertex set is  $\bbZ^2$ and  whose  edge  set includes  $(x, x+e_i)$  whenever $\geo{m+1}{x}{\xi\sig} = x + e_i$. Here $m=x\cdot\et$  and we consider both $i\in\{1,2\}$.   These are the directed graphs of $\xi\sigg$ geodesics defined by  \eqref{d:bgeo}. By construction, each $\G{\xi\sig}$ is a disjoint union of trees, i.e.\ a forest, and for each $x \in \Z^2$, the geodesic $\geo{}{x}{\xi\sig}$ follows the directed edges of   $\G{\xi\sig}$. 

Recall the vectors    $\etstar=\et/2= (e_1 + e_2)/2$  and $\exstar=\ex/2=(e_2-e_1)/2$. 
Let $\dG{\xi\sig}$ be the directed graph whose vertex set is the dual lattice $\Z^{2*}=\etstar+\bbZ^2$ and whose edge set is defined by this rule:  for each $x\in\bbZ^2$,  on the dual lattice $x+\etstar$ points to $x+\etstar-e_i$ in $\dG{\xi\sig}$ if and only if on the original lattice  $x$ points to $x+e_i$ in $\G{\xi\sig}$.   Pictorially this means that $\dG{\xi\sig}$ contains all the  south and west  directed nearest-neighbor  edges of  $\Z^{2*} $  that do not  cross an edge of $\G{\xi\sig}$.     See Figure \ref{fig:trees} for an illustration.

\begin{figure}[h]
\centering
 		 \begin{tikzpicture}[scale=1,>=stealth]
			\begin{scope}
			\draw(0,-.5)--(0,1.5);
			\draw(1,-.5)--(1,1.5);
			\draw(-.5,0)--(1.5,0);
			\draw(-.5,1)--(1.5,1);
			\draw[color=darkblue,line width=2pt,->](0,0)--(0,1);
			\draw[color=nicos-red,line width=1pt,->](.5,.5)--(.5,-.5);
			\draw[fill=black](0,0) circle(2pt);
			\draw(-.05,-0.15)node[left]{$x$};
			\draw[fill=white](.5,.5) circle(2pt);
			\draw(0.5,0.62)node[right]{$x^*=x+\etstar$};
			\end{scope}
			
			\begin{scope}[shift={(0,2.5)}]
			\draw(0,-.5)--(0,1.5);
			\draw(1,-.5)--(1,1.5);
			\draw(-.5,0)--(1.5,0);
			\draw(-.5,1)--(1.5,1);
			\draw[color=darkblue,line width=2pt,->](0,0)--(1,0);
			\draw[color=nicos-red,line width=1pt,->](.5,.5)--(-.5,.5);
			\draw[fill=black](0,0) circle(2pt);
			\draw(-.05,-.15)node[left]{$x$};
			\draw[fill=white](.5,.5) circle(2pt);
			\draw(0.5,0.62)node[right]{$x^*=x+\etstar$};
			\end{scope}			
		\end{tikzpicture}
		\hspace{2cm}
\includegraphics[trim={5.4cm 15.6cm 7.5cm 3.6cm},clip,width=5cm]{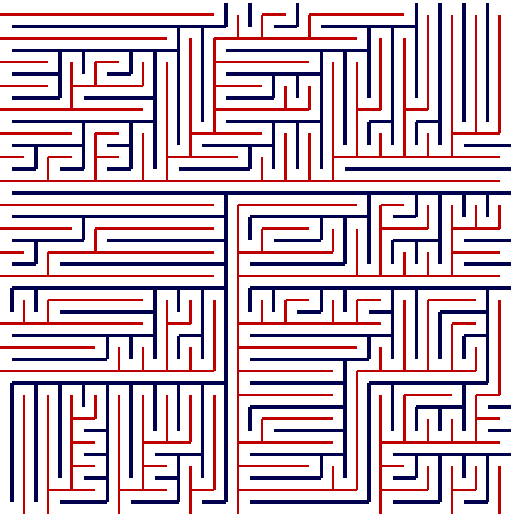}
\caption{\small  Left plot: An illustration of the duality relation between the edges of $\G{\xi{\scriptscriptstyle{\boxempty}}}$ (black/thick) and those of $\dG{\xi{\scriptscriptstyle{\boxempty}}}$ (red/thin). Right plot: An illustration of a (blue/thick) north-east directed  geodesic graph $\G{\xi{\scriptscriptstyle{\boxempty}}}$   and its (red/thin) south-west directed  dual $\dG{\xi{\scriptscriptstyle{\boxempty}}}$.}
\label{fig:trees}\end{figure}
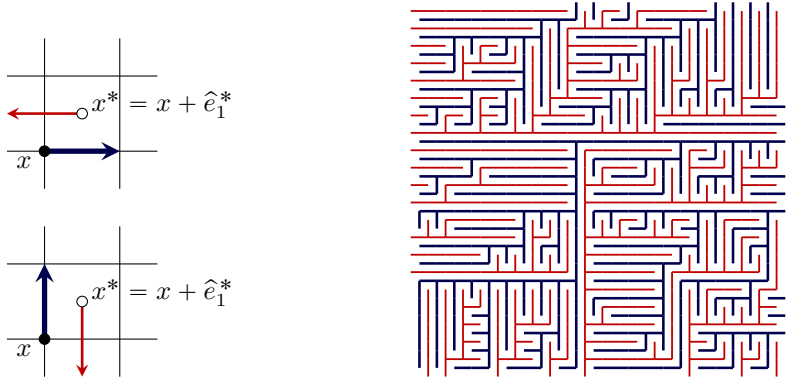

For  $\zeta\preceq\eta$ in $\ri\Uset$ let  the graph 
$\dS{[\zeta,\eta]}$ be the union of  the graphs $\dG{\xi\sig}$ over $\xi\in[\zeta,\eta]$ and $\sigg\in\{-,+\}$.
 That is, the vertex set of  $\dS{[\zeta,\eta]}$ is $\Z^{2*}$, and the edge set of  $\dS{[\zeta,\eta]}$ is the union of the  edge sets of $\dG{\xi\pm}$ over $\xi\in[\zeta,\eta]$.   From each point $x^*\in\Z^{2*}$   a directed edge of $\dS{[\zeta,\eta]}$ points to $x^*-e_1$ or $x^*-e_2$ or both.  Due to  the monotonicity \eqref{mono} of the Busemann functions,   $\dS{[\zeta,\eta]}$ is  the union of just the two  graphs $\dG{\zeta-}$ and $\dG{\eta+}$.  In particular,   $x^*$ points to $x^*-e_2$ in $\dS{[\zeta,\eta]}$ if and only if $x^*-\etstar$ points to $x^*+\exstar$ in $\G{\zeta-}$, and 
 $x^*$ points to $x^*-e_1$ in $\dS{[\zeta,\eta]}$ if and only if $x^*-\etstar$  points to $x^*-\exstar$ in $\G{\eta+}$.

Identify the space-time point 
$x+ \etstar \in\bbZ^{2*}$
on the dual lattice with the diagonal edge that connects $x+e_1$ and $x+e_2$
on the primal lattice (see Figure \ref{fig:dual}).
Call the dual lattice point $x^*=x+\etstar$  a $[\zeta, \eta]$-{\it instability point}  if $[\zeta, \eta]\cap\supp\mu_{x+e_1,x+e_2}\neq\varnothing$.   
If $\zeta=\eta=\xi$,  call  $x^*$  a {\it $\xi$-instability point}.   
Denote the set of $[\zeta,\eta]$-instability points   by $\shock{[\zeta,\eta]}$,  with  $\shock{\xi}=\shock{[\xi,\xi]}$.  
$\shock{[\zeta,\eta]}$ is the union of $\shock{\xi}$ over $\xi\in[\zeta,\eta]$.
    Theorem \ref{thm:nonint} and the ordering \eqref{path-ordering}  of geodesics give the following characterization in terms of disjoint geodesics, alluded to in Section \ref{s:HJ}. 

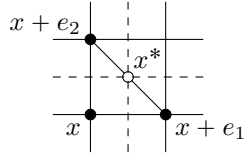
\begin{figure}
 	\begin{center}
 		 \begin{tikzpicture}[scale=1,>=stealth]
			\draw(0,1)--(1,0);
			\draw[fill=black](1,0) circle(2pt);
			\draw[fill=black](0,1) circle(2pt);
			\draw[fill=black](0,0) circle(2pt);
			\draw[fill=white](.5,.5) circle(2pt);
			\draw(0,-0.5)--(0,1.5);
			\draw(1,-0.5)--(1,1.5);
			\draw(-.5,0)--(1.5,0);
			\draw(-.5,1)--(1.5,1);
			\draw[dashed](-.5,0.5)--(1.5,0.5);
			\draw[dashed](.5,-0.5)--(.5,1.5);
			\draw(0,-.2)node[left]{$x$};
			\draw(.45,.75)node[right]{$x^*$};
			\draw(0,1.2)node[left]{$x+e_2$};
			\draw(1,-.2)node[right]{$x+e_1$};
			
		\end{tikzpicture}
 	 \end{center}
 	\caption{\small  The edge $(x+e_1,x+e_2)$ is identified with the dual point $x^*=x+\etstar$.}
 \label{fig:dual}
 \end{figure}

\begin{lemma}\label{lem:shockequiv} 
The following holds for $\bbP$-almost every $\w$.   Let $\zeta\preceq\eta$, including the case $\zeta=\eta=\xi$.     Let $x\in\Z^2$ and   $x^*=x+\etstar$.   Then $x^*\in\shock{[\zeta,\eta]}$    if and only if $\geo{}{x+e_2}{\zeta-}\cap\geo{}{x+e_1}{\eta+}=\varnothing$. 
\end{lemma}



  Let  the {\it instability graph} $\shockG{[\zeta,\eta]}$ be the subgraph of $\dS{[\zeta,\eta]}$  with vertex set $\shock{[\zeta,\eta]}$ and those  directed edges of $\dS{[\zeta,\eta]}$  that point  from some  $x^*\in\shock{[\zeta,\eta]}$ to  a point 
$x^*-e_i\in\shock{[\zeta,\eta]}$, for either  $i\in\{1,2\}$.  (The proof of  Theorem  \ref{thm:shock1} in Section \ref{shockgraphs:pfs}  shows that {\it every} edge of $\dS{[\zeta,\eta]}$ that emanates from a point of $\shock{[\zeta,\eta]}$ is in fact an edge of  $\shockG{[\zeta,\eta]}$.)   

In the  case $\zeta=\eta=\xi$ write  $\shockG{\xi}$ for   $\shockG{[\xi,\xi]}$.  Explicitly,  the vertices of $\shockG{\xi}$ are dual points $x+\etstar$ such that $\xi\in\supp\mu_{x+e_1,x+e_2}$ and the edges are those of $\dG{\xi-}\cup\dG{\xi+}$ that connect these points.  

The graph $\shockG{[\zeta,\eta]}$ is also the edge union of the graphs $\shockG{\xi}$  over $\xi\in[\zeta, \eta]$. To see this, let $x^*=x+\etstar$.  If $\xi=\cid(T_x\w)\in[\zeta, \eta]$ then  $\shockG{\xi}$ contains both edges from $x^*$ to $x^*-e_1$ and $x^*-e_2$, as does $\shockG{[\zeta,\eta]}$.    If $\cid(T_x\w)\notin[\zeta, \eta]$ then in  $\shockG{\xi}$ and in $\shockG{[\zeta,\eta]}$, $x^*$ points to   the same vertex $x^*-e_i$.  


 \begin{remark}
By the continuity \eqref{pmgeolim} and the fact that the support of a measure is a closed set we get that almost surely, 
for any $\zeta\preceq\eta$ in $\ri\Uset$ and for any finite box $[-L,L]^2\cap\Z^2$, $\shockG{[\zeta',\eta']}=\shockG{[\zeta,\eta]}$ on the entire box, for $\zeta'\prec\zeta$ close enough to $\zeta$ and $\eta'\succ\eta$ close enough to $\eta$.
This explains why the two top graphs in Figure \ref{fig:shocks2b} are identical and are in fact equal to $\shockG{\xi}$.
 \end{remark}

The message of the  next theorem is that  instability points exist for all exceptional directions in $\aUset$, and these  instability points arrange themselves on  bi-infinite directed paths  in the instability graphs.

\begin{theorem}\label{thm:shock1}
The following holds for $\bbP$-almost every $\w$. Pick any $\zeta\preceq\eta$ in $\ri\Uset$ such that  $[\zeta,\eta]\cap\,\aUset\ne\varnothing$, including   the case $\zeta=\eta=\xi$.    Then the  instability graph $\shockG{[\zeta,\eta]}$ is an infinite directed graph.   Furthermore,  $\shockG{[\zeta,\eta]}$ equals the union of the bi-infinite directed paths of the graph $\dG{\cup[\zeta,\eta]}$.   In the backward {\rm(}north and east{\rm)} orientation, each such path is $[\zetamin,\etamax]$-directed. 
\end{theorem} 

In particular,   if $x^*\in\shock{[\zeta,\eta]}$ and $m=x^*\cdot\et$,  there exists a bi-infinite sequence $\{x^*_n\}_{n\in\Z}\subset\shock{[\zeta,\eta]}$ such that $x^*_m=x^*$ and for each $n$,   $x^*_n\cdot\et=n$ and  $x^*_n$ points to  $x^*_{n-1}$   in the graph $\shockG{[\zeta,\eta]}$.  As $n\to\infty$,  the limit points of $n^{-1}x^*_n$ lie in $[\zetamin,\etamax]$.

\begin{figure}[h]
\centering
\includegraphics[width=0.45\textwidth]{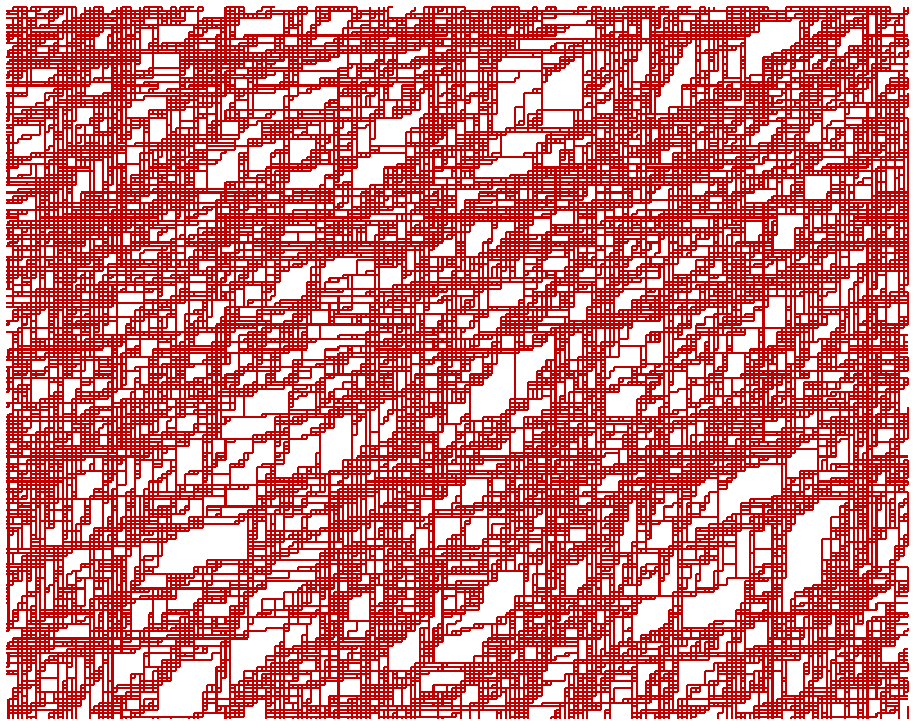}\ \ \ 
\includegraphics[width=0.45\textwidth]{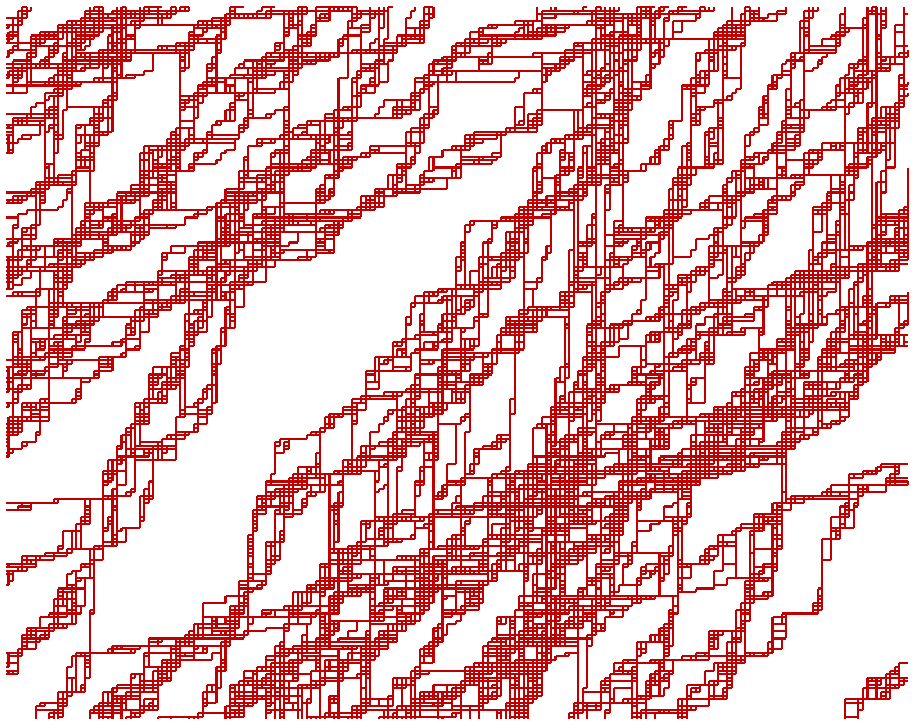}
\includegraphics[width=0.45\textwidth]{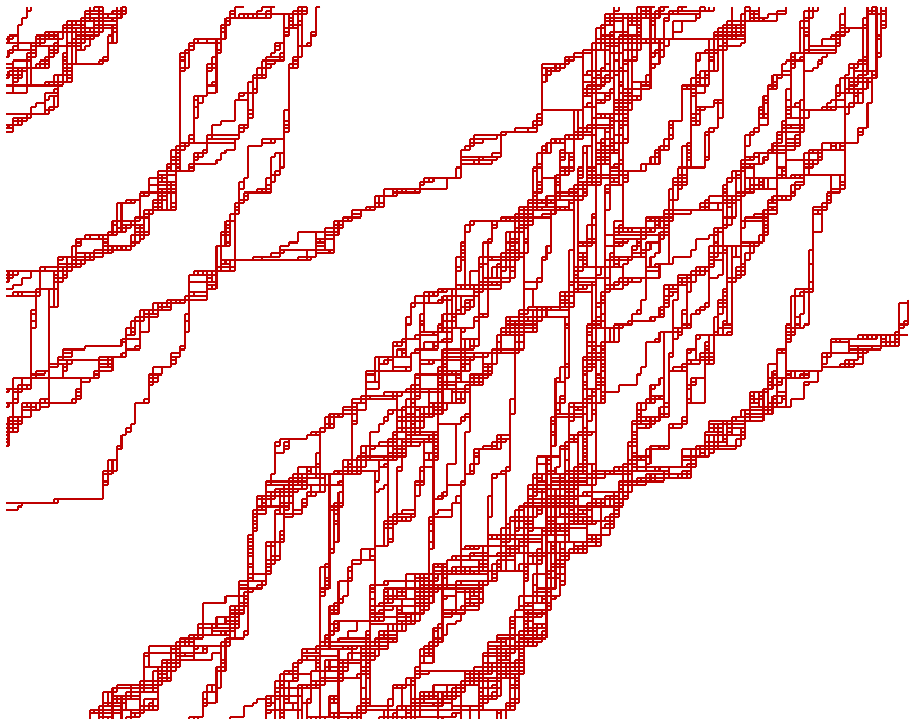}\ \ \ 
\includegraphics[width=0.45\textwidth]{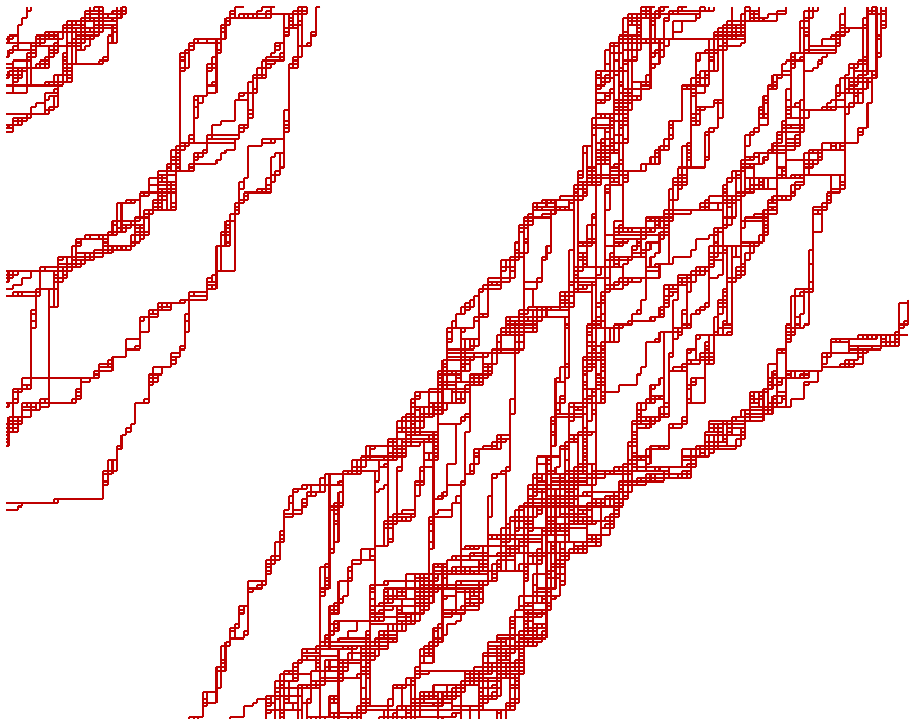}
\caption{\small  Four nested down-left pointing  $\shockG{[\zeta,\eta]}$ graphs in the square  $[-100,100]^2$. Top to bottom, left to right, in reading order, 
$[\zeta\cdot e_1,\eta\cdot e_1]$ equals  $[0.096,0.772]$, $[0.219,0.595]$, $[0.318,0.476]$, and $[0.355,0.436]$.  
Two further nested subgraphs appear in Figure \ref{fig:shocks2b}.  
In the simulation the weights were exponentially distributed and we chose the direction $\xi$ to be a jump point of the Busemann process on the edge $(0,e_1)$.}
\label{fig:shocks2a}\end{figure}

\begin{figure}[h]
\centering
\includegraphics[width=0.45\textwidth]{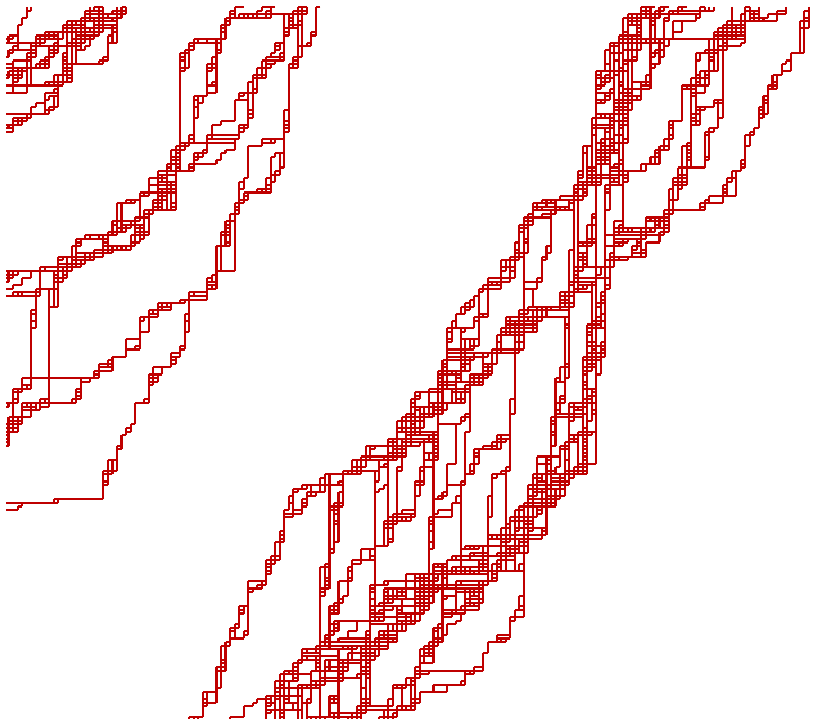}\ \ \ 
\includegraphics[width=0.45\textwidth]{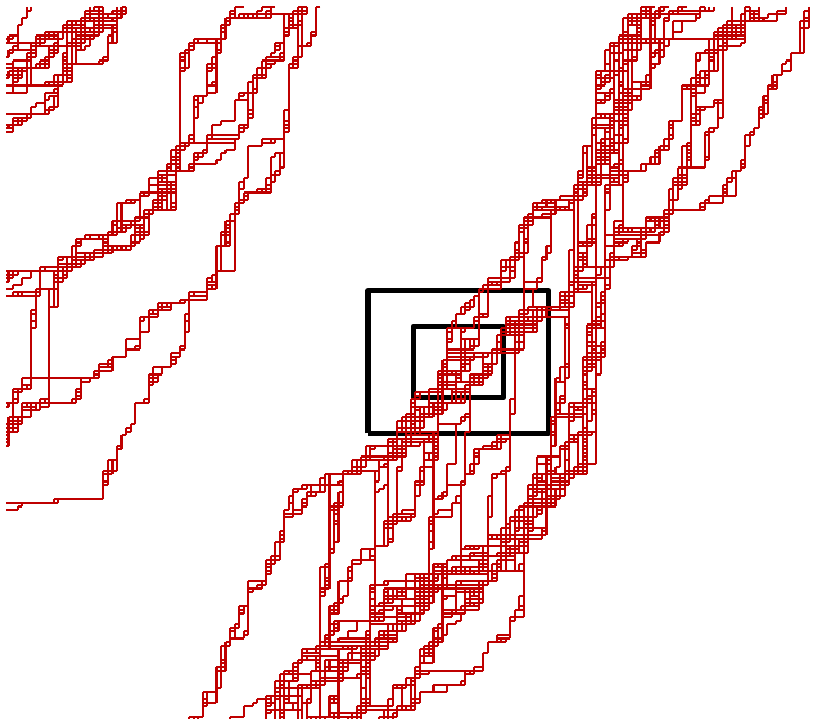}
\includegraphics[width=0.45\textwidth]{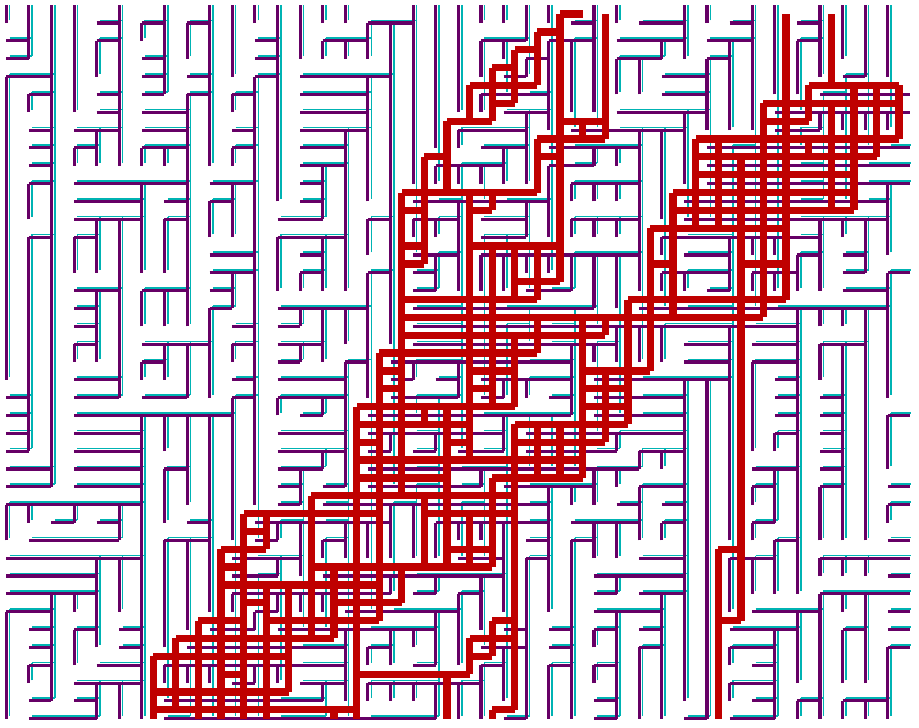}\ \ \  
\includegraphics[width=0.45\textwidth]{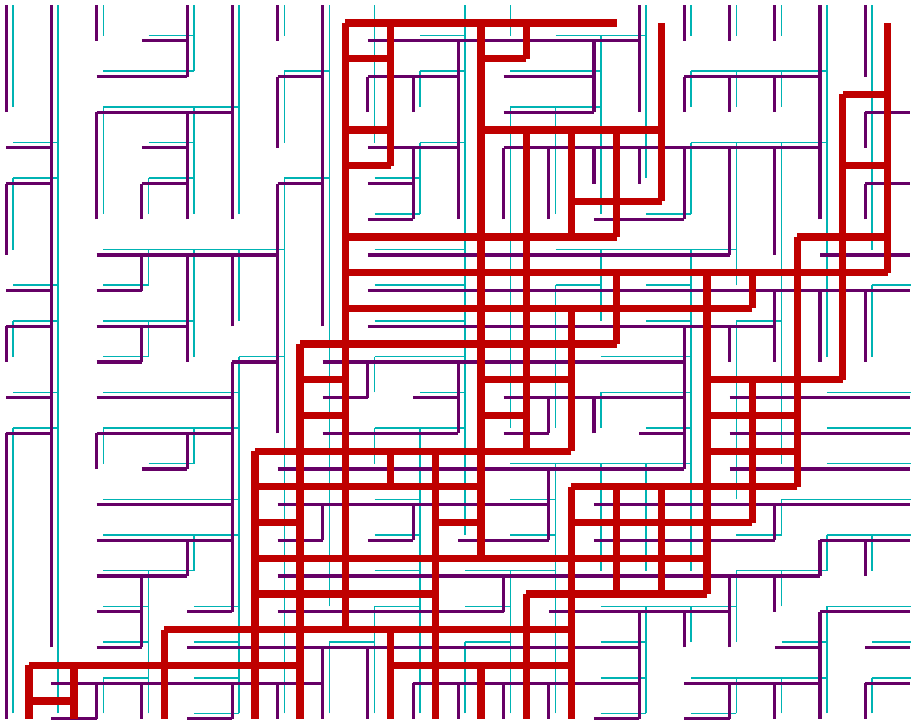}
\caption{\small  Continuing with the simulation setting of Figure \ref{fig:shocks2a}, the top two pictures  are $\shockG{[\zeta,\eta]}$ graphs in $[-100,100]^2$ with 
$[\zeta\cdot e_1,\eta\cdot e_1]=[0.374,0.417]$ (left) and $[0.393,0.397]$ (right). 
The two graphs are in fact identical.  The  pictures on the second row zoom into the framed squares of the  top right picture, the left one into the square $[-20,20]^2$ and  the right  one into $[-10,10]^2$.  Besides the down-left pointing red $\shockG{[\zeta,\eta]}$ graphs, the bottom pictures include   the up-right pointing graphs $\G{\zeta-}$ (green/lighter) and $\G{\eta+}$ (purple/darker).  Whenever $\G{\zeta-}$  and $\G{\eta+}$ separate at $x$,    green points up and purple points right, and  $\shockG{[\zeta,\eta]}$ has a branch point at $x+\etstar$.   The blue/green trees that occupy the islands surrounded by red paths are described in Section \ref{sec:Busflow}. 
}
\label{fig:shocks2b}\end{figure}


Next we   describe  the branching and coalescing of the bi-infinite directed paths that make up   the  graph $\shockG{[\zeta,\eta]}$.  
If there is a  directed path  in the graph $\shockG{[\zeta,\eta]}$  from $y^*$ to $x^*$, then   $y^*$ is an {\it ancestor} of $x^*$ and  equivalently  $x^*$ is a {\it descendant} of $y^*$.  
Let $\ans{[\zeta,\eta]}{x^*}$ denote the set of ancestors of $x^*$ in the graph $\shockG{[\zeta,\eta]}$.  Abbreviate again  $\ans{\xi}{x^*}=\ans{[\xi,\xi]}{x^*}$.\smallskip  

A point $x^*\in\shock{[\zeta,\eta]}$ is   a {\it branch point}  in the  graph $\shockG{[\zeta,\eta]}$  if  $x^*$  is an ancestor of both $x^*-e_1$ and $x^*-e_2$.  Branch points are dual to those where $\zeta-$ and $\eta+$ geodesics separate. Similarly, $x^*\in\shock{[\zeta,\eta]}$ is a {\it coalescence point} if both $x^*+e_1$ and $x^*+e_2$ are ancestors of $x^*$.   Figures \ref{fig:shocks2a} and \ref{fig:shocks2b} display  simulations that illustrate the branching and coalescing.

For the sharpest branching and coalescing properties in the next theorem, we invoke again the regularity condition \eqref{g-reg}  and  the jump process condition \eqref{cond:jumpcond}, and additionally the non-existence of non-trivial bi-infinite geodesics: 

\begin{align}\label{cond:biinf}
\begin{minipage}{0.9\textwidth}
There exists a full $\bbP$-probability event on which the only bi-infinite geodesics are the trivial ones: $x+\Z e_i$ for $x\in\Z^2$ and $i\in\{1,2\}$.
\end{minipage}
\end{align}
Condition  \eqref{cond:biinf} is known to hold  in the exponential case 
  \cite{Bal-Bus-Sep-20,Bas-Hof-Sly-18-}.

%
 
\begin{theorem}\label{thm:shocks}
The following hold for $\bbP$-almost every $\w$ and all $\zeta\preceq\eta$ in $\ri\Uset$ such that  $[\zeta,\eta]\cap\,\aUset\ne\varnothing$.  The case $\zeta=\eta=\xi$ is included unless otherwise stated.
\begin{enumerate}[label={\rm(\alph*)}, ref={\rm\alph*}] \itemsep=3pt 
\item\label{thm:shocks.c} $x^*$ is a branch point in $\shockG{[\zeta,\eta]}$
if and only if $\cid(T_{x^*-\etstar}\w)\in[\zeta,\eta]$.
\item\label{thm:shocks.d} If $\zeta\prec\eta$, then any $x^*,y^*\in\shock{[\zeta,\eta]}$ have a common descendant: $\exists z^*\in\shock{[\zeta,\eta]}$ such that $x^*,y^*\in\ans{[\zeta,\eta]}{z^*}$. If we assume the no bi-infinite geodesics condition \eqref{cond:biinf}, then the same statement holds also for the case $\zeta=\eta=\xi$.
\item\label{thm:shocks.e} Assume the jump process condition \eqref{cond:jumpcond}.  
Then any $x^*,y^*\in\shock{[\zeta,\eta]}$ have a common ancestor 
$z^*\in\ans{[\zeta,\eta]}{x^*}\cap\ans{[\zeta,\eta]}{y^*}$. 
\item\label{thm:shocks.f} Suppose  $\zeta\prec\eta$ are such that $]\zeta,\eta[\,\cap\aUset\ne\varnothing$.  Then  for any $z\in\Z^2$ there is a coordinatewise strictly ordered infinite sequence $z<z_1^*<z_2^*< \dotsm<z_n^*<\dotsm$   such that each $z_n^*$ is a   branch point in $\shockG{[\zeta,\eta]}$.  
There are also  infinitely many coalescence points in $\shockG{[\zeta,\eta]}$. 
\item\label{thm:shocks.g} If the jump process condition \eqref{cond:jumpcond} holds and $\xi\in\aUset$, then  for any $z\in\Z^2$ there is a coordinatewise strictly ordered infinite sequence $z<z_1^*<z_2^*< \dotsm<z_n^*<\dotsm$   such that each $z_n^*$ is a   branch point in $\shockG{\xi}$. 
If additionally the no bi-infinite geodesics condition \eqref{cond:biinf} holds, then there are infinitely many coalescence points in $\shockG{\xi}$.
\end{enumerate}
\end{theorem}

\begin{remark}\label{rk:erg2}
If regularity condition \eqref{g-reg} holds, then part \eqref{thm:shocks.f} holds for $\zeta\prec\eta$ with $[\zeta,\eta]\cap\aUset\ne\varnothing$. 
The proof of this is given right after that of Theorem \ref{thm:shocks} in Section \ref{shockgraphs:pfs}.
\end{remark}

Given that there are infinitely many instability points when instability points exist, it is natural to wonder what their density on the lattice is. We identify the following trichotomy.

\begin{proposition} \label{prop:trichotomy}
Assume the regularity condition \eqref{g-reg}.
Then for $\bbP$-almost every $\w$ and all $\xi\in\ri\Uset$, exactly one of the following three scenarios happens: 
\begin{enumerate}[label={\rm(\alph*)}, ref={\rm\alph*}]
\item $\xi\not\in\aUset$ and hence there are no $\xi$-instability points.
\item $\xi\in\aUset\cap\Diff$ and there are infinitely many $\xi$-instability points but they have zero density.
\item $\xi\not\in\Diff$ and $\xi$-instability points have positive density.
\end{enumerate}
\end{proposition}

%
We return to this question in Section \ref{sec:exp} in the solvable case of exponential weights, where we can say significantly more.

\subsection{Flow of Busemann measure} \label{sec:Busflow}
 This section views the instability graph $\shockG{[\zeta,\eta]}$ as a description of the south-west directed  flow of Busemann measure on the dual lattice. 
 As discussed in Section \ref{s:HJ}, we can think of the function $\B{\xi\sig}(x+e_1, x+e_2)$ 
  as a global solution of a discretized stochastic Burgers equation.  We can assign the value $\B{\xi\sig}(x+e_1, x+e_2)$ to the dual point $x^*=x+\etstar$ that represents the diagonal edge $(x+e_1, x+e_2)$.  Then the cocycle property \eqref{coc-prop} gives us a flow of Busemann measure along the south and west pointing edges of the dual lattice $\Z^{2*}$. 
 First decompose   the Busemann measure of the  edge $(x+e_1, x+e_2)$ as a sum  $\mu_{x+e_1, x+e_2}=\mu_{x+e_1, x}+\mu_{x,x+e_2}$ of two positive measures.  This is justified by  the cocycle property \eqref{coc-prop}. 
Then  stipulate that measure $\mu_{x+e_1, x}$ flows south from $x^*$ to $x^*-e_2$ and contributes to Busemann measure $\mu_{x-\ex, x}$,  while measure $\mu_{x,x+e_2}$ flows west from $x^*$ to $x^*-e_1$ and contributes to Busemann measure $\mu_{x, x+\ex}$.   See Figure \ref{fig:flow}.

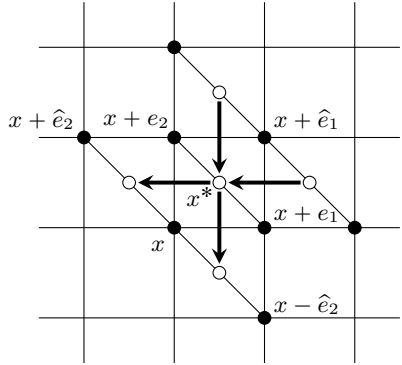
\begin{figure}[h]
 	\begin{center}
 		 \begin{tikzpicture}[scale=1.2,>=stealth]
			\draw(0,-.5)--(0,3.5);		 
			\draw(1,-.5)--(1,3.5);		 
			\draw(2,-.5)--(2,3.5);		 
			\draw(3,-.5)--(3,3.5);		 
			\draw(-.5,0)--(3.5,0);		 
			\draw(-.5,1)--(3.5,1);		 
			\draw(-.5,2)--(3.5,2);		 
			\draw(-.5,3)--(3.5,3);	
			\draw(0,2)--(2,0);
			\draw(1,2)--(2,1);
			\draw(1,3)--(3,1);
			\draw[line width=1.5pt,->](1.4,1.5)--(0.6,1.5);	 
			\draw[line width=1.5pt,->](1.5,1.4)--(1.5,0.6);	 
			\draw[line width=1.5pt,->](2.4,1.5)--(1.6,1.5);	 
			\draw[line width=1.5pt,->](1.5,2.4)--(1.5,1.6);	 
			\draw[fill=white](1.5,1.5) circle(2pt);
			\draw(1.55,1.35)node[left]{\small$x^*$};
			\draw[fill=white](0.5,1.5) circle(2pt);
			\draw[fill=white](1.5,0.5) circle(2pt);
			\draw[fill=white](2.5,1.5) circle(2pt);
			\draw[fill=white](1.5,2.5) circle(2pt);
			\draw[fill=black](1,1) circle(2pt);
			\draw(1,0.8)node[left]{\small$x$};
			\draw[fill=black](2,0) circle(2pt);
			\draw(2,.15)node[right]{\small$x-\ex$};
			\draw[fill=black](0,2) circle(2pt);
			\draw(0,2.2)node[left]{\small$x+\ex$};
			\draw[fill=black](1,2) circle(2pt);
			\draw(1.05,2.2)node[left]{\small$x+e_2$};
			\draw[fill=black](2,1) circle(2pt);
			\draw(2,1.15)node[right]{\small$x+e_1$};
			\draw[fill=black](2,2) circle(2pt);
			\draw(2,2.2)node[right]{\small$x+\et$};
			\draw[fill=black](3,1) circle(2pt);
			\draw[fill=black](1,3) circle(2pt);
		
		\end{tikzpicture}
 	 \end{center}
 	\caption{\small  The flow of Busemann measures follows the arrows.  The antidiagonal edge $(x+e_1, x+e_2)$  is identified with the dual point $x^*=x+\etstar$.   The Busemann measure  $\mu_{x+e_1,x+e_2}$ on this edge is composed of the mass flowing from the north and east, and it in turns divides its mass between the flows south and west.}
 \label{fig:flow}
 \end{figure}

The cocycle property also tells us that $\mu_{x+e_1, x+e_2}=\mu_{x+e_1, x+\et}+\mu_{x+\et,x+e_2}$.  This represents $\mu_{x+e_1, x+e_2}$ as the sum of the contributions it receives from the next level up: $\mu_{x+e_1, x+\et}$ comes from the east from dual vertex $x+e_1+\etstar$, while $\mu_{x+\et,x+e_2}$ comes from the north  from dual vertex $x+e_2+\etstar$. 

Now pick a pair of directions $\zeta\preceq\eta$ in $\ri\Uset$,  and consider the graph $\cB^*_{[\zeta, \eta]}$  on the dual lattice $\Z^{2*}$ obtained as follows.  Include vertex $x^*=x+\etstar$ if $[\zeta, \eta]\cap\supp\mu_{x+e_1, x+e_2}\ne\varnothing$.  For $i\in\{1,2\}$,  include dual edge $(x^*, x^*-e_i)$ if $[\zeta, \eta]$ intersects $\supp\mu_{x, x+e_{3-i}}$, or somewhat pictorially, if some of the support in $[\zeta, \eta]$ flows along the dual edge  $(x^*, x^*-e_i)$. \smallskip 

The results of this section hold $\P$-almost surely simultaneously for all $\zeta\preceq\eta$ in $\ri\Uset$, including for the case $\zeta=\eta=\xi$.

\begin{theorem}\label{th:flowS}
  The graphs $\cB^*_{[\zeta, \eta]}$ and $\shockG{[\zeta,\eta]}$ are the same.  
\end{theorem} 

Under  the jump condition \eqref{cond:jumpcond}, a closed set cannot intersect the support without actually having nonzero measure.  Thus under \eqref{cond:jumpcond},   Theorem \ref{th:flowS} tells us that $\shockG{[\zeta,\eta]}$ is precisely the graph along which positive Busemann measure in the interval $[\zeta, \eta]$  flows. 

 \medskip
 
 Next we describe the ``islands'' on $\Z^2$ carved out by the paths of the graph $\shockG{[\zeta, \eta]}$  (islands surrounded by red paths in Figures \ref{fig:shocks2a} and \ref{fig:shocks2b}).   These islands are trees, they are the connected components of an intersection of geodesic graphs, and they are the equivalence classes of an equivalence  relation defined in terms of the supports of Busemann measures. 
 
Define the graph $\G{\cap[\zeta, \eta]}=\bigcap_{\xi\in[\zeta, \eta]} (\G{\xi-}\cap\G{\xi+})$ on the  vertex set  $\Z^2$ by keeping only those edges that lie in each geodesic graph $\G{\xi\sig}$ as $\xi$ varies over $[\zeta, \eta]$ and $\sigg$ over $\{-,+\}$.  
Also, directly from the definitions  follows that an edge of $\Z^2$ lies in $\G{\cap[\zeta, \eta]}$ if and only if the dual edge it crosses does not lie in the graph $\dS{[\zeta, \eta]}$ introduced in Section \ref{sec:web}.   Since each $\G{\xi\sig}$ is a forest, $\G{\cap[\zeta, \eta]}$ is a forest, that is, a union of disjoint trees.  

Define an equivalence relation $\zesim$ on $\Z^2$  by   $x\zesim y$ if and only if  $\supp{\mu_{x,y}}\cap[\zeta, \eta]=\varnothing$. It is an equivalence relation because $\mu_{x,x}$ is the identically zero measure, and  $\B{\xi\sig}_{x,z}= \B{\xi\sig}_{x,y}+\B{\xi\sig}_{y,z}$  implies that $\abs{\mu_{x,z}} \le \abs{\mu_{x,y}} + \abs{\mu_{y,z}}$.     In terms of coalescence,  $x\zesim y$ if and only if the coalescence points  $\coal{\xi\sig}(x,y)$ remain  constant in $\Z^2$ as  $\xi$ varies across $[\zeta, \eta]$ and $\sigg$ over $\{-,+\}$.  (This follows from   Propositions \ref{pr:supp1}  and  \ref{pr:supp2} proved below.)    As usual, replace $\zesim$ with $\xisim$ when $[\zeta, \eta]=[\xi, \xi]$. 

\begin{proposition}\label{lm:K9} 
The equivalence classes of the relation $\zesim$ are exactly the connected components {\rm(}subtrees{\rm)} of $\G{\cap[\zeta, \eta]}$. 
\end{proposition} 

Lemma \ref{lm:cross} proved below shows that nearest-neighbor points of $\Z^2$ are in distinct $\zesim$ equivalence classes if and only if the edge between them is bisected by an edge of the instability graph $\shockG{[\zeta,\eta]}$.   Together with Proposition \ref{lm:K9} this tells us that the paths of $\shockG{[\zeta, \eta]}$ are precisely the boundaries that separate distinct connected components of $\G{\cap[\zeta, \eta]}$ and the equivalence classes of $\zesim$.

The next two lemmas indicate how 
the structure of the   subtrees  of $\G{\cap[\zeta, \eta]}$ is constrained by the fact that they are  intersections of geodesic trees.     These properties are clearly visible in  the bottom pictures of Figure \ref{fig:shocks2b} where these  subtrees     are the blue/green trees in  the  islands separated by red paths.

\begin{lemma}\label{lm:K10}   Let  $\cK$ be a subtree of  $\G{\cap[\zeta, \eta]}$ and let  $x$ and  $y$ be two distinct vertices of $\cK$.   Assume that  neither strictly dominates the other in the coordinatewise ordering, that is, both coordinatewise strict  inequalities $x<y$ and $y<x$ fail.  Then    the  entire rectangle  $\lzb x\wedge y, x\vee y\rzb$  is a subset of the vertex set of $\cK$.  
\end{lemma} 

In particular, if for some integers $\{t,k,\ell\}$, level-$t$ lattice points $(k, t-k)$ and $(\ell, t-\ell)$ are vertices of a subtree $\cK$, the entire discrete interval $\{(i, t-i): i\in\lzb k,\ell\rzb\}$ is a subset of the vertex set of $\cK$.  Similarly, points on horizontal and vertical line segments between vertices of a subtree $\cK$ are again vertices of $\cK$.

\begin{lemma} \label{lm:K12}   Let  $\cK$ be a subtree of  $\G{\cap[\zeta, \eta]}$.  
There is at most one vertex  $x$ in $\cK$ such that $\{x- e_1, x- e_2\}\cap\cK=\varnothing$.  Such a point $x$ exists if and only if $\inf\{ t\in\Z: \cK\cap\level_t\ne\varnothing\}>-\infty$.  In that case   $\cK$ lies in $\{y: y\ge x\}$. 
\end{lemma} 

Note that Lemma \ref{lm:K12} does not say that a subtree has a single leaf.  Both $x$ and $x-e_i$ can be leaves of a subtree when the edge $(x-e_i, x)$ is not present in $\G{\cap[\zeta, \eta]}$.

\medskip

For the remainder of this section assume the jump condition \eqref{cond:jumpcond},   in order to 
give a sharper description of the subtrees of $\G{\cap[\zeta, \eta]}$.  
 Let $\cifset_{[\zeta, \eta]}=\{z\in\Z^2: \cid(T_z\w)\in{[\zeta, \eta]}\}$.  By Theorem \ref{thm:shocks}\eqref{thm:shocks.c}, $z\in\cifset_{[\zeta, \eta]}$ if and only if $z+\etstar$ is a branch point of the instability graph $\shockG{[\zeta,\eta]}$.    It follows then that both $z\pm\exstar$ are also  $[\zeta,\eta]$-instability points.  
 
   Assume for the moment that $\cifset_{[\zeta, \eta]}\ne\varnothing$.  By Theorem \ref{th:V1}, under the jump condition \eqref{cond:jumpcond} this is equivalent to $[\zeta, \eta]\cap\,\aUset\ne\varnothing$. 
 
The graph $\G{\cap[\zeta, \eta]}$ has no outgoing up or right edges from a point  $z\in\cifset_{[\zeta, \eta]}$ because geodesics split: 
 $\geo{}{z}{\cid(T_z\w)-}$ and   $\{\geo{}{z}{\xi\pm}:\zeta\preceq \xi\prec\cid(T_z\w)\}$ take the $ e_2$-step at $z$, while  $\geo{}{z}{\cid(T_z\w)+}$ and  $\{\geo{}{z}{\xi\pm}:  \cid(T_z\w)\prec \xi\preceq\eta\}$ take the $ e_1$-step at $z$.   For each $z\in\cifset_{[\zeta, \eta]}$, let  the tree $\cK(z)$ consist of all directed  paths in $\G{\cap[\zeta, \eta]}$ that terminate at $z$. $\cK(z)$ can consist of $z$ alone.  
 
These  properties come  from previously established facts: 
 \begin{itemize} 
\item   Each $x\in\Z^2\setminus\cifset_{[\zeta, \eta]}$ lies in a unique $\cK(z)$ determined by  following  the common path  of the geodesics $\{\geo{}{x}{\xi\sig}:\xi\in[\zeta,\eta], \,\sigg\in\{-,+\}\}$ until  the first point $z$ at which a split happens.  A split must happen eventually  because for any $u\in\cifset_{[\zeta,\eta]}$ the two geodesics  $\geo{}{u}{\cid(T_u\w)\pm}$ separate immediately at $u$, while by Theorem  \ref{thm:+-coal}    the geodesic $\geo{}{x}{\cid(T_u\w)\sig}$  
 coalesces with $\geo{}{u}{\cid(T_u\w)\sig}$  for both $\sigg\in\{-,+\}$. 
 
 \item If $\zeta\prec\eta$ then each tree $\cK(z)$ is finite.  Same  holds also  for the case $\zeta=\eta=\xi$ under the no bi-infinite geodesics condition \eqref{cond:biinf}.   This follows from Theorem \ref{thm:shocks}\eqref{thm:shocks.d} because the $[\zeta,\eta]$-instability points   $z\pm\exstar$  that flank $z$ have a common descendant $u^*$  in the graph $\shockG{[\zeta,\eta]}$. The two directed paths  of $\shockG{[\zeta,\eta]}$ that connect $z+\etstar$ to $u^*$ surround $\cK(z)$.  
 
 \end{itemize}  
 
 The final theorem of this section decomposes $\G{\cap[\zeta, \eta]}$ into its connected components. 

\begin{theorem}\label{thm:jumpG}  Assume the jump condition \eqref{cond:jumpcond}. 
\begin{enumerate}[label={\rm(\alph*)}, ref={\rm\alph*}] \itemsep=3pt 
\item\label{thm:jumpG.a} $\G{\cap[\zeta, \eta]}$ is a single tree if and only if $[\zeta, \eta]\cap\,\aUset=\varnothing$. 
\item\label{thm:jumpG.b} If $[\zeta, \eta]\cap\,\aUset\ne\varnothing$, 
the  connected components of $\G{\cap[\zeta, \eta]}$ are the trees  
$\{\cK(z):  z\in\cifset_{[\zeta, \eta]}\}$. 
\end{enumerate}
\end{theorem} 

We finish by reminding the reader that all the hypotheses and hence all the conclusions hold in  the case of i.i.d.\ exponential weights. 
The results of Section \ref{sec:Busflow}  are proved in Section  \ref{sec:Busflowpf}.


\section{Statistics of instability points in the exponential model}\label{sec:exp}  
Under condition \eqref{exp-assump}, i.e.~when the weights are exponentially distributed, we derive explicit statistics of the instability graphs. 
For $\xi\in\ri\Uset$, $k\in\Z$, and $\sigg\in\{-,+\}$, abbreviate $\depa^{\xi\sig}_k=\B{\xi\sig}\bigl(ke_1,(k+1)e_1\bigr)$ and write $\depa^\xi_k$ when there is no $\pm$ distinction.
For  $\zeta\preceq\eta$ in $\ri\Uset$ 
let 
\[ \dotsm<\tau^{\zeta, \eta}(-1)<0\le \tau^{\zeta, \eta}(0)<\tau^{\zeta, \eta}(1)<\dotsm\]  be the ordered  indices such that 
\be\label{78-45} 
\depa^{\zeta-}_k>\depa^{\eta+}_k\quad\text{if and only if}\quad k\in\{\tau^{\zeta, \eta}(i):i\in\Z\}.   
\ee
If $\depa^{\zeta-}_k>\depa^{\eta+}_k$ happens for only finitely many indices $k$, then some $\tau^{\zeta, \eta}(i)$ are set equal to $-\infty$ or $\infty$. 

By Theorem \ref{th:exp1}, under condition \eqref{exp-assump}, \eqref{78-45}  is equivalent to 
\[\coal{\zeta-}(k e_1, (k+1) e_1) \neq \coal{\eta+}(k e_1,(k+1)e_1).\] 
It is worth keeping this geometric implication of \eqref{78-45} in mind in this section to provide some context for the results that follow. 

It will be convenient in what follows to parametrize directions in $\ri\Uset$ through the increasing  bijection
 \be\label{u-a}    \zeta=\zeta(\alpha)=\biggl(    \frac{\alpha^2}{(1-\alpha)^2+\alpha^2}\,,  \frac{(1-\alpha)^2}{(1-\alpha)^2+\alpha^2}\biggr)
 \ \Longleftrightarrow \ 
 \alpha= \alpha(\zeta)= \frac{\sqrt{\zeta\cdot e_1}}{\sqrt{\zeta\cdot e_1}+ \sqrt{1-\zeta\cdot e_1}}
  \ee
 between $\zeta\in\ri\Uset$ and $\alpha\in(0,1)$. Recall   the {\it Catalan numbers}   $C_n=\frac{1}{n+1}{2n\choose n}$ for $n\geq 0$.  By \eqref{B4} from Appendix \ref{a:bus},  the conditioning event  in the theorem below has probability $\P(\depa^\zeta_0>\depa^\eta_0)= \frac{\alpha(\eta)-\alpha(\zeta)}{\alpha(\eta)}$. 
Since $\zeta\prec\eta$ are fixed, with probability one no $\pm$ distinction appears   in  the Busemann functions. 
 
 \begin{theorem}\label{th:palm1} 
Assume \eqref{exp-assump}.
Fix $\zeta\prec\eta$ in $\ri\Uset$. 
Conditional on $\depa^\zeta_0>\depa^\eta_0$,  \[\bigl\{\tau^{\zeta, \eta}(i+1)-\tau^{\zeta, \eta}(i), \depa^\zeta_{\tau^{\zeta, \eta}(i)}-\depa^\eta_{\tau^{\zeta, \eta}(i)}: i\in\Z\bigr\}\]
is an  i.i.d.\ sequence with marginal distribution 
\be\label{78-90} \begin{aligned} 
&\P\bigl(\tau^{\zeta, \eta}(i+1)-\tau^{\zeta, \eta}(i)=n, \, \depa^\zeta_{\tau^{\zeta, \eta}(i)}-\depa^\eta_{\tau^{\zeta, \eta}(i)}>r\,\big\vert\,\depa^\zeta_0>\depa^\eta_0\,\bigr) \\
&\qquad \qquad 
= C_{n-1}\,  \frac{\alpha(\zeta)^{n-1}\alpha(\eta)^n}{\bigl(\alpha(\zeta)+\alpha(\eta)\bigr)^{2n-1}}\,e^{-\alpha(\zeta) r}
\,,\quad \forall \, i\in\Z, \, n\in\N,\, r\in\R_+. 
\end{aligned} \ee
\end{theorem}

Abbreviate $\tau^\xi(i)=\tau^{\xi, \xi}(i)$. Our next goal is to describe the joint distribution of processes $(\{\tau^\xi(i):i\in\Z\}, \, \depa^{\xi-}_{\tau^{\xi}(i)}-\depa^{\xi+}_{\tau^{\xi}(i)})$ of locations and sizes of jumps in direction $\xi$, conditional on $\{\B{\xi-}_0 > \B{\xi+}_0\}$.   However, for a fixed $\xi$,    $\B{\xi+}=\B{\xi-}$ almost surely and so this conditioning has to be understood in the Palm sense.  This is natural for conditioning on a jump of a point process at a particular location. 

In the theorem below, Lebesgue measure on $\Uset$ refers to one-dimensional Lebesgue measure (length of a line segment). 
The Lebesgue-almost every qualifier is in the theorem because the Palm kernel is defined only up to Lebesgue-null sets of the points $\xi$.  We denote Palm conditioning with two vertical lines $||$ to distinguish it from ordinary conditioning.  The  definition of the Palm conditioning used in \eqref{rw113} below  appears in equation \eqref{eq:palmdef} at the end of Section \ref{sub:palm}. For references, see \cite{Kal-83, Kal-17}.

\begin{theorem}\label{th:B-palm1}  
Assume \eqref{exp-assump}.
For Lebesgue-almost every $\xi\in\ri\Uset$,  under the Palm kernel, conditional on $\B{\xi-}_0 > \B{\xi+}_0$,
 $\{\tau^\xi(i+1)-\tau^\xi(i), \depa^{\xi-}_{\tau^{\xi}(i)}-\depa^{\xi+}_{\tau^{\xi}(i)} : i\in\Z\}$ is an  i.i.d.\ sequence with marginal distribution 
\be\label{rw113} \begin{aligned} 
&\bbP\bigl\{\tau^\xi(i+1)-\tau^\xi(i) =n, \depa^{\xi-}_{\tau^{\xi}(i)}-\depa^{\xi+}_{\tau^{\xi}(i)}>r\,||\, \B{\xi-}_0 > \B{\xi+}_0\bigr\} \\
&\qquad\qquad\qquad
= C_{n-1}\,  \frac1{2^{2n-1}}\, e^{-\alpha(\xi)r},\qquad \forall\, i\in\Z, \, n\in\N,\, r\in\R_+.  
\end{aligned} \ee
\end{theorem}


Equation \eqref{rw113}  connects the Palm distribution of the locations of jumps of the Busemann process with the zero set of simple symmetric random  walk (SSRW). 
 Let $\RW_n$ denote a two-sided SSRW, that is, $\RW_0=0$ and $\RW_n-\RW_m=\sum_{i=m+1}^n \ZRW_i$ for all $m<n$ in $\Z$ where $\{\ZRW_i\}_{i\in\Z}$ are  i.i.d.\ with $P(\ZRW_i=\pm1) =1/2$.
Set $\rho_n = \one_{\{S_{2n} = 0\}}$ and let $\bfP$ be the  distribution of $\rho = \{\rho_n\}_{n\in\Z}$  on the sequence  space  $\{0,1\}^\Z$.  That is, $\bfP$ is the law of the zero set of simple symmetric random walk  sampled at even times.  The classical  inter-arrival distribution of this renewal process is (eqn.~III.3(3.7) on p.~78 of Feller \cite{Fel-68})
\be\label{rw117}  
\bfP( \rho_1=0, \dotsc, \rho_{n-1}=0, \rho_n=1) 
= C_{n-1}\,  \frac1{2^{2n-1}}.  \ee
Comparison of \eqref{rw113} and \eqref{rw117} reveals that for Lebesgue-almost every $\xi$, the Palm distribution of the locations of $\xi$-instability points on a line is the same as the law of the zero set of SSRW sampled at even times. 
(We record this fact precisely as Lemma \ref{lem:palmrw}.) The next result applies this to show that any translation invariant event which holds with probability one for the zero set of SSRW holds for all of the instability graphs simultaneously almost surely.

 \begin{theorem} \label{thm:ladder-xi}  
 Assume \eqref{exp-assump}.
 Suppose $A$ is a translation-invariant Borel subset of $\{0,1\}^\Z$ that  satisfies  $\bfP(A)=1$. 
Then 
\be\label{rw125}
\P\bigl\{ \forall\xi\in\aUset: \bigl( \one\{\B{\xi-}_\ell > \B{\xi+}_\ell\} : \ell \in \bbZ\bigr)\in A\bigr\}=1.
\ee
 \end{theorem} 
 
From \eqref{rw125} and known facts about random walk, we can derive corollaries. 
From \cite[equation (10.8)]{Rev-05}, we deduce that 
	\begin{align}
	\P\Bigl\{\forall\xi\in\aUset:\varlimsup_{n\to\infty}\frac{\sum_{i=0}^{n} \one\{\B{\xi-}_i > \B{\xi+}_i\}}{\sqrt{8n\log\log n}}=1\Bigr\}=1.
	\end{align}
From \cite[Theorem 11.1]{Rev-05} we also find that for a nonincreasing $\delta_n$,
	\begin{align}\label{red-LB}
	\P\Bigl\{\forall\xi\in\aUset:n^{-1/2}\sum_{i=0}^{n}\one\{\B{\xi-}_i > \B{\xi+}_i\}\ge\delta_n\quad\text{for all sufficiently large $n$}\Bigr\}=1 
	\end{align}
if $\sum_n\delta_n/n<\infty$ and 
	\begin{align}
	\P\Bigl\{\forall\xi\in\aUset:n^{-1/2}\sum_{i=0}^{n}\one\{\B{\xi-}_i > \B{\xi+}_i\}\le\delta_n\quad\text{infinitely often}\Bigr\}=1
	\end{align}
otherwise. 
Similar statements hold for the sums $\sum_{i=-n}^0$. This implies that for $\P$-almost every $\w$ and any $\xi\in\aUset$, the number of horizontal edges $(ke_1,(k+1)e_1)$ with $\xi\in\supp\mu_{ke_1,(k+1)e_1}$ and $-n\le k\le n$ is of order $n^{1/2}$. It suggests the number of such horizontal edges (and thus also vertical edges and $\xi$-instability points) in an $n\times n$ box should be of order $n^{3/2}$.  The next theorem gives an  upper bound.  The lower bound is left  for future work.

\begin{theorem}\label{thm:density-UB}
Assume \eqref{exp-assump} and fix $i\in\{1,2\}$.
Then for any $\zeta\in\ri\Uset$
		\begin{align*}
		\P\Bigl\{\exists n_0:\forall \xi\in[\zeta,e_2[\,,\forall n\ge n_0:\sum_{x\in\lzb 0,n\rzb^2}\one\bigl\{\xi \in \supp\mu_{x,x+e_i}\bigr\}\le2n^{3/2}\sqrt{\log n}\Bigr\}=1.
		\end{align*}
The same holds when $\lzb0,n\rzb^2$ is replaced by any of $\lzb-n,0\rzb^2$, $\lzb0,n\rzb\times\lzb-n,0\rzb$, or $\lzb-n,0\rzb\times\lzb0,n\rzb$.
\end{theorem}
\medskip

This completes the presentation of the main results.   After a list of open problems, the remaining sections cover   the proofs.
The results of Section \ref{sec:exp}  are proved in Section  \ref{sec:exp-pf}.


\section{Open problems}\label{sec:prob}


The list below contains some immediate open questions raised by the results of this paper. 
\begin{enumerate} 
\itemsep=3pt 

\item Find tail estimates for the coalescence points $\coal{\xi}(x,y)$. 

\item
Theorem \ref{thm:Vcid}\eqref{thm:Vcid.2} showed that the jump process condition \eqref{cond:jumpcond} implies that $\aUset=\{\cid(T_x\w): x\in\Z^2\}$.  Is this implication an equivalence?  

\item 
Prove the jump process condition \eqref{cond:jumpcond}  for any  model other than the exactly solvable exponential and geometric cases.  

\item 
Does the web of instability have a scaling limit?  

\item 
Does the web of instability, with branching and coalescing  in exceptional directions, have any analogue in stochastic equations in continuous space and/or continuous time?    

 \item 
 Extend the statistics of instability points in the exponential model beyond a single line on the lattice. 
 
 \end{enumerate}  

\section{Busemann measures: proofs} \label{sec:bus}
The rest of the paper relies  on Appendix \ref{app:busgeo} where prior results from the literature are collected. The reader may wish to look through that appendix before proceeding; 
in particular, we will work on the $T$-invariant full-measure event $\Omega_0$ constructed in  \eqref{Om0}.\smallskip

Fix a countable dense set $\Udense\subset\Diff$ of points of differentiability of the shape function $\gpp$ (recall \eqref{sh-th}).   These  play a role in the definition of the event $\Omega_0$ in \eqref{Om0}. 
Recall the definition \eqref{df:coal-z} of the 
coalescence point $\coal{\xi\sig}(x,y)$.  
When $\coal{\xi\sig}(x,y) \in \bbZ^2$, equation \eqref{eq:buspass} leads to the following identity, which is fundamental to the analysis that follows:
\begin{align} \label{eq:busdiff}
\B{\xi\sig}(x,y) &= G\bigl(x,\coal{\xi\sig}(x,y)\bigr) - G\bigl(y,\coal{\xi\sig}(x,y)\bigr) = \sum_{i=k}^{n-1} \w_{\geo{i}{x}{\xi\sig}} - \sum_{i=\ell}^{n-1} \w_{\geo{i}{y}{\xi\sig}},
\end{align}
where $k=x\cdot\et$, $\ell=y\cdot\et$, and $n=\coal{\xi\sig}(x,y)\cdot\et$. By Theorem \ref{thm1}\eqref{thm1:coal}, for all $\w \in \Omega_0$, all $\xi\in\Udense$,  and all $x,y \in \bbZ^2$,  both $\coal{\xi+}(x,y)$ and $\coal{\xi-}(x,y)$ are in $\bbZ^2$.

We begin with  results linking analytic properties of the Busemann process and coalescence points.

\begin{proposition} \label{pr:supp1}
For all $\w \in \Omega_0$, for any $\zeta \prec \eta$ in $\ri\Uset$, and any $x,y \in \bbZ^2$, the following statements are equivalent. 
\begin{enumerate}  [label={\rm(\roman*)}, ref={\rm\roman*}]   \itemsep=3pt 
\item\label{eq100} $\abs{\mu_{x,y}}(\,]\zeta,\eta[\,) = 0 $.
\item\label{eq101} $\B{\zeta+}(x,y) = \B{\eta-}(x,y)$ and $\coal{\zeta+}(x,y),\coal{\eta-}(x,y) \in \bbZ^2$.
\item\label{eq102} $\coal{\zeta+}(x,y) = \coal{\eta-}(x,y) \in \bbZ^2$.
\item\label{eq103} There exists  $z \in \bbZ^2$ such that the following holds. For any $\pi \in \{\geo{}{x}{\xi\sig} : \xi \in \, ]\zeta,\eta[\,, \sigg\in\{-,+\}\}$ and any $\geod{}' \in \{\geo{}{y}{\xi\sig} : \xi \in \, ]\zeta,\eta[\,, \sigg\in\{-,+\} \}$, $\geod{}\cap \geod{}' \neq \varnothing$ and $z$ is the first point where $\pi$ and $\pi'$ intersect: $z\cdot\et=\min\{z'\cdot\et:z'\in\geod{}\cap\geod{}'\}$.
\end{enumerate}
\end{proposition}

\begin{proof}  
\eqref{eq100}$\implies$\eqref{eq101}. 
Under \eqref{eq100}   the functions $\xi \mapsto \B{\xi\sig}_{x,y}$ match for $\sigg\in\{-,+\}$ and are constant on the open interval $]\zeta,\eta[$.   $\B{\zeta+}(x,y) = \B{\eta-}(x,y)$ follows by taking limits $\xi \searrow \zeta$ and $\xi \nearrow \eta$. 

Since on $]\zeta,\eta[\,\cap\,\Udense$, $\xi \mapsto \B{\xi}_{x,y}$ is  constant  and  $\coal{\xi}(x,y)\in\Z^2$  (Theorem \ref{thm1}\eqref{thm1:coal}),   \eqref{eq:busdiff} and condition  \eqref{paths.2}  imply 
that $\coal{\xi}(x,y)$ is constant in $\Z^2$   for all  $\xi\in\,]\zeta,\eta[\,\cap\,\Udense$. 
Since $\Udense$ is dense in $]\zeta,\eta[\,$,  limits   
 \eqref{coal-lim} as $\xi \searrow \zeta$ and $\xi \nearrow \eta$ imply
 that $\coal{\zeta+}(x,y),\coal{\eta-}(x,y) \in \bbZ^2$.


\smallskip

 \eqref{eq101}$\implies$\eqref{eq102}. 
Set $k=x \cdot \et $, $\ell=y \cdot \et $.   With  both $\coal{\zeta+}(x,y)$ and $\coal{\eta-}(x,y)$ in $\bbZ^2$, we also set  $m=\coal{\zeta+}(x,y)\cdot\et$, and $n=\coal{\eta-}(x,y)\cdot\et$. 
By \eqref{eq:busdiff}, 
\begin{align*}
\B{\zeta+}_{x,y} &= G\bigl(x,\coal{\zeta+}(x,y)\bigr) - G\bigl(y,\coal{\zeta+}(x,y)\bigr) = \sum_{i=k}^{m-1} \w_{\geo{i}{x}{\zeta+}} - \sum_{i=\ell}^{m-1} \w_{\geo{i}{y}{\zeta+}} 
\\
\text{and}\quad 
\B{\eta-}_{x,y} &=  G\bigl(x,\coal{\eta-}(x,y)\bigr) - G\bigl(y,\coal{\eta-}(x,y)\bigr) = \sum_{i=k}^{n-1} \w_{\geo{i}{x}{\eta-}} - \sum_{i=\ell}^{n-1} \w_{\geo{i}{y}{\eta-}}.
\end{align*}
By condition  \eqref{paths.2},  the vanishing of   $\B{\zeta+}_{x,y} - \B{\eta-}_{x,y}$    forces $m=n$, $\geo{k,m}{x}{\zeta+} = \geo{k,m}{x}{\eta-}$, and $\geo{\ell,m}{y}{\zeta+} = \geo{\ell,m}{y}{\eta-}$, and hence in particular $\coal{\zeta+}(x,y)=\coal{\eta-}(x,y)$.  


\smallskip 
 
\eqref{eq102}$\implies$\eqref{eq103}.
With $m=\coal{\zeta+}(x,y)\cdot\et=\coal{\eta-}(x,y)\cdot\et$, 
  uniqueness of finite geodesics implies  $\geo{k,m}{x}{\zeta+} = \geo{k,m}{x}{\eta-}$ and $\geo{\ell,m}{y}{\zeta+} = \geo{\ell,m}{y}{\eta-}$. Then monotonicity  \eqref{path-ordering} gives   $\geo{k,m}{x}{\zeta+} = \geo{k,m}{x}{\xi\sig} = \geo{k,m}{x}{\eta-}$ and  $\geo{\ell,m}{y}{\zeta+} = \geo{\ell,m}{y}{\xi\sig} = \geo{\ell,m}{y}{\eta-}$ for all $\xi \in\, ]\zeta,\eta[$ and $\sigg\in\{-,+\}$. 
The point $z$ is  $\geo{m}{x}{\zeta+}=\geo{m}{y}{\zeta+}=\geo{m}{x}{\eta-}=\geo{m}{y}{\eta-}$.   

\smallskip
 
 \eqref{eq103}$\implies$\eqref{eq100}.
Let $m=z \cdot \et$. It follows from uniqueness of finite geodesics that all of the paths $\geo{k,m}{x}{\xi\pm}$ must be the same, for all $\xi\in\,]\zeta,\eta[$, and similarly all of the paths $\geo{\ell,m}{y}{\xi\pm}$ must be the same. Sending $\xi \searrow \zeta$ and $\xi \nearrow \eta$, we obtain that for all $\xi \in \, ]\zeta,\eta[$, $\geo{k,m}{x}{\zeta+} = \geo{k,m}{x}{\xi \pm} = \geo{k,m}{x}{\eta-}$ and $\geo{\ell,m}{y}{\zeta+} = \geo{\ell,m}{y}{\xi \pm} = \geo{\ell,m}{y}{\eta-}$.  
Recall that \eqref{eq:busdiff} applies for any $\xi\in\Udense$.
Thus, the functions  $\xi \mapsto \B{\xi\pm}_{x,y}$ match and are constant, when restricted to the dense set $\Udense\, \cap\, ]\zeta,\eta[$. Combining this with the left-continuity of $\xi \mapsto \B{\xi-}_{x,y}$ and the right-continuity of $\xi \mapsto \B{\xi+}_{x,y}$, we see that the functions $\xi \mapsto \B{\xi\pm}_{x,y}$ match and are constant on $]\zeta,\eta[$. 
This implies \eqref{eq100}.
\end{proof}

Proposition \ref{pr:supp1} has a counterpart in terms of fixed directions lying in the support of $\mu_{x,y}$.
\begin{proposition} \label{pr:supp2}
For all $\w \in \Omega_0$ and all $x,y \in \bbZ^2$, the following are equivalent:
\begin{enumerate} [label={\rm(\roman*)}, ref={\rm\roman*}]   \itemsep=3pt 
\item\label{eq110}  $\xi \not\in \supp{\mu_{x,y}}$.
\item\label{eq111} $\coal{\xi-}(x,y) = \coal{\xi+}(x,y) \in \bbZ^2$.
\item\label{eq112} $\B{\xi-}(x,y) = \B{\xi+}(x,y)$ and $\coal{\xi-}(x,y), \coal{\xi+}(x,y) \in \bbZ^2$. 
\end{enumerate}
\end{proposition}

\begin{proof}
Let $x \cdot \et = k$ and $y \cdot \et = \ell$. Take sequences $ \zeta_n, \eta_n \in \Udense$ with $\zeta_n \nearrow \xi$ and $\eta_n \searrow \xi$. Since $\zeta_n, \eta_n \in \Udense$ we have $\coal{\zeta_n}(x,y),\coal{\eta_n}(x,y) \in \bbZ^2$ for all $n$.   Furthermore,   $\B{\zeta_n}(x,y) \to  \B{\xi-}(x,y)$ and $\B{\eta_n}(x,y) \to  \B{\xi+}(x,y)$ as $n\to\infty$. 
 
 \smallskip 
 
 \eqref{eq110}$\implies$\eqref{eq111}.
  If $\xi \notin \supp{\mu_{x,y}}$, then  $\zeta \mapsto \B{\zeta\pm}(x,y)$ is constant on some neighborhood of $\xi$.   Then   \eqref{eq100}$\implies$\eqref{eq102} from  Proposition \ref{pr:supp1} gives \eqref{eq111}.  

\smallskip
   
\eqref{eq111}$\implies$\eqref{eq112}.
 Let $m=\coal{\xi-}(x,y)\cdot\et=\coal{\xi+}(x,y)\cdot\et$. Then  uniqueness of finite geodesics implies that  $\geo{k,m}{x}{\xi-} = \geo{k,m}{x}{\xi+}$ and $\geo{\ell,m}{y}{\xi-} = \geo{\ell,m}{y}{\xi+}$.
  \eqref{pmgeolim} implies that for sufficiently large $n$,  $\geo{k,m}{x}{\xi-} = \geo{k,m}{x}{\zeta_n}$, $\geo{k,m}{x}{\xi+} = \geo{k,m}{x}{\eta_n}$, $\geo{\ell,m}{y}{\xi-} = \geo{\ell,m}{y}{\zeta_n}$, and $\geo{\ell,m}{y}{\xi+} = \geo{\ell,m}{y}{\eta_n}$.  For these large $n$, 
\begin{align*}
\B{\zeta_n}(x,y) &= G\bigl(x,\coal{\zeta_n}(x,y)\bigr) - G\bigl(y,\coal{\zeta_n}(x,y)\bigr) \\
&= \sum_{i=k}^{m-1} \w_{\geo{i}{x}{\zeta_n}} - \sum_{i=\ell}^{m-1} \w_{\geo{i}{y}{\zeta_n}} 
=\sum_{i=k}^{m-1} \w_{\geo{i}{x}{\xi-}} - \sum_{i=\ell}^{m-1} \w_{\geo{i}{y}{\xi-}}\\
&= \sum_{i=k}^{m-1} \w_{\geo{i}{x}{\xi+}} - \sum_{i=\ell}^{m-1} \w_{\geo{i}{y}{\xi+}}
= \sum_{i=k}^{m-1} \w_{\geo{i}{x}{\eta_n}} - \sum_{i=\ell}^{m-1} \w_{x_i^{y,\eta_n}}\\
&=  G\bigl(x,\coal{\eta_n}(x,y)\bigr) - G\bigl(y,\coal{\eta_n}(x,y)\bigr) =\B{\eta_n}(x,y).
\end{align*}
Taking $n\to\infty$ gives
$\B{\xi+}(x,y)=\B{\xi-}(x,y)$. Claim \eqref{eq112} is proved.
 
 \smallskip 
 
 \eqref{eq112}$\implies$\eqref{eq110}.
 The assumption $\coal{\xi-}(x,y), \coal{\xi+}(x,y) \in \bbZ^2$ allows us to use \eqref{eq:busdiff}.  This  
  and convergence of geodesics \eqref{pmgeolim} implies that  $\B{\zeta_n}(x,y) = \B{\xi-}(x,y)= \B{\xi+}(x,y)=\B{\eta_n}(x,y)$ for sufficiently large $n$.  The equivalence between \eqref{eq101} and \eqref{eq100} in Proposition \ref{pr:supp1} implies that for such $n$, both processes are constant on the interval $]\zeta_n,\eta_n[\,$. Therefore $\xi \notin \supp{\mu_{x,y}}$.
\end{proof}

With these results in hand, we next turn to the proofs of our main results.

\begin{proof}[Proof of Theorem \ref{thm:nonint}]
Fix $\w\in\Omega_0$, $x,y\in\Z^2$, and $\xi\in\ri\Uset$.
Suppose that \eqref{thm:nonint.a} does not hold, i.e.\ $\xi \notin \supp{\mu_{x,y}}$. 
By Proposition \ref{pr:supp2}, we have $\coal{\xi-}(x,y) = \coal{\xi+}(x,y)\in\Z^2$, in which case both $\geo{}{x}{\xi-} \cap \geo{}{y}{\xi+}$ and $\geo{}{x}{\xi+} \cap \geo{}{y}{\xi-}$ include this common point and thus
\eqref{thm:nonint.b} is false.
This proves that \eqref{thm:nonint.b} implies \eqref{thm:nonint.a}.

Now, suppose that $\xi \in \supp{\mu_{x,y}}$ and that $\geo{}{x}{\xi-} \cap \geo{}{y}{\xi+} \neq \varnothing$ and $\geo{}{x}{\xi+} \cap \geo{}{y}{\xi-} \neq \varnothing$.   Without loss of generality assume that $x \cdot \et = k \leq m=y \cdot \et $. Let $z_1$ denote the first point at which $\geo{}{x}{\xi-}$ and $\geo{}{y}{\xi+}$ meet and call $z_2$ the first point at which $\geo{}{x}{\xi+}$ and $\geo{}{y}{\xi-}$ meet. Let $\ell_1 = z_1 \cdot \et$ and $\ell_2 = z_2 \cdot \et$. We denote by $u$ the leftmost (i.e.\ smallest $e_1$ coordinates) of the three points $\geo{m}{x}{\xi+},y,\geo{m}{x}{\xi-}$ and by $v$ the rightmost of these three points. Note that if $u=v$, then $\coal{\xi+}(x,y) = \coal{\xi-}(x,y)=y$, which would imply that  $\xi \notin \supp{\mu_{x,y}}$.  Thus   $u\ne v$ and there are  two cases: either $y \in \{u,v\}$ or not. We show a contradiction in both cases.

First, we work out the case $y=v$, with the case of $y=u$ being similar. See the left picture in Figure \ref{fig:nonint} for an illustration. 
In this case we have, for all $n \geq m$, $\geo{n}{x}{\xi-} \preceq \geo{n}{x}{\xi+}\preceq \geo{n}{y}{\xi+}$ and $\geo{n}{x}{\xi-} \preceq \geo{n}{y}{\xi-}\preceq \geo{n}{y}{\xi+}$. In words, $\geo{}{y}{\xi+}$ is the rightmost geodesic and $\geo{}{x}{\xi-}$ is the leftmost geodesic among the four geodesics $\geo{}{x}{\xi\pm}$, $\geo{}{y}{\xi\pm}$. By the path ordering \eqref{path-ordering} and planarity,  $z_1$ must lie on all four geodesics. Then by the uniqueness of finite geodesics  $\geo{k,\ell_1}{x}{\xi+} = \geo{k,\ell_1}{x}{\xi-}$ and   $\geo{m,\ell_1}{y}{\xi+} = \geo{m,\ell_1}{y}{\xi-}$. It follows that $z_1 =\coal{\xi+}(x,y) = \coal{\xi-}(x,y)$, contradicting $\xi \in \supp{\mu_{x,y}}$.

 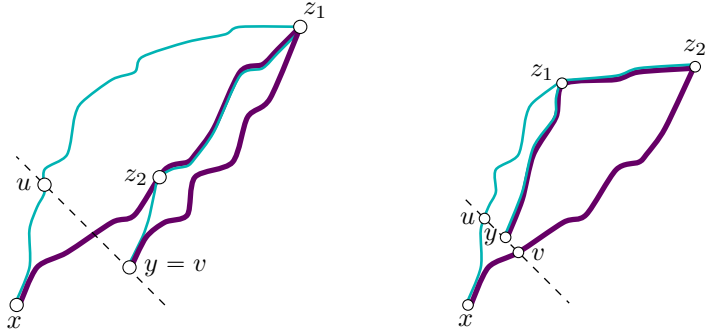
\begin{figure}[h]
 	\begin{center}
 		 \begin{tikzpicture}[scale=1,shorten >=1pt,>={Latex[length=3.5mm,width=2mm]}]
	 		\begin{scope}
			\draw[color=sussexg, line width=1pt]plot[smooth] coordinates{(0,0)(0.2,0.5)(0.2,1)(.3,1.3)(0.4,1.5)(0.4,1.8)(0.7,2)(0.8,2.5)(0.9,2.7)(1.3,3)(1.6,3.1)(1.7,3.3)(2.3,3.5)(2.8,3.6)(3,3.7)(3.8,3.7)};
			\draw[color=sussexp, line width=2pt]plot[smooth] coordinates{(0.07,0)(0.3,0.5)(0.7,.7)(1.3,1.1)(1.6,1.2)(2,1.8)(2.3,1.9)(2.4,2)(2.65,2.35)(3,3.1)(3.3,3.2)(3.5,3.4)(3.8,3.7)};
			\draw[color=sussexg, line width=1pt]plot[smooth] coordinates{(1.5,0.5)(1.7,0.9)(1.8,1.2)(1.9,1.7)};
			\draw[color=sussexg, line width=1pt]plot[smooth] coordinates{(1.9,1.7)(2.1,1.808)(2.3,1.9-0.05)(2.4,2-0.05)(2.65,2.35-0.05)(3,3.1-0.05)(3.3,3.2-0.05)(3.5,3.4-0.05)(3.8,3.7-0.05)};
			\draw[color=sussexp, line width=2pt]plot[smooth] coordinates{(1.55,0.5)(1.75,0.9)(2,1.1)(2.3,1.2)(2.4,1.7)
			(2.9,1.9)(3.1,2.5)(3.4,2.75)(3.8,3.7)};

			\draw[black,dashed,line width=0.5pt]plot(2,0)--(0,2);
			\draw[black,fill=white](0.03,0) circle(2.5pt);
			\draw[black,fill=white](1.53,0.5) circle(2.5pt);
			\draw[black,fill=white](3.8,3.7) circle(2.5pt);
			\draw[black,fill=white](1.93,1.7) circle(2.5pt);
			\draw[black,fill=white](0.4,1.6) circle(2.5pt);

			\draw(0,-.03)node[below]{$x$};
			\draw(1.6,0.5)node[right]{$y=v$};
			\draw(0.35,1.6)node[left]{$u$};
			\draw(4,3.7)node[above]{$z_1$};
			\draw(1.9,1.7)node[left]{$z_2$};
			\end{scope}
			
	 		\begin{scope}[shift={(6,0)}]
			\draw[color=sussexg, line width=1pt]plot[smooth] coordinates{(0,0)(0.2,0.5)(0.2,1)(.3,1.3)(0.4,1.5)(0.4,1.8)(0.7,2)(0.8,2.5)(0.9,2.7)(1.3,3)}; 
			\draw[color=sussexp, line width=2pt]plot[smooth] coordinates{(0.07,0)(0.3,0.5)(0.7,.7)(1.3,1.1)(1.6,1.2)(2,1.8)(2.3,1.9)(2.4,2)(2.65,2.35)(3.05,3.2)};
			\draw[color=sussexp, line width=2pt]plot[smooth] coordinates{(0.55,0.9)(0.75,1.4)(.9,2)(1.2,2.5)(1.3,2.94)(2,3)(2.3,3.1)(3.05,3.15)};
			\draw[color=sussexg, line width=1pt]plot[smooth] coordinates{(0.53,0.9+.05)(0.75,1.4+.1)(.9,2+.1)(1.18,2.5)(1.2,2.75)(1.25,2.94+0.01)};
			\draw[color=sussexg, line width=1pt]plot[smooth] coordinates{(1.3,2.94+.05)(2,3+.05)(2.3,3.1+.05)(3.05,3.2)};

			\draw[black,dashed,line width=0.5pt]plot(0,1.4)--(1.4,0);
			\draw[black,fill=white](0.03,0) circle(2pt);
			\draw[black,fill=white](0.53,0.9) circle(2pt);
			\draw[black,fill=white](1.28,2.95) circle(2pt);
			\draw[black,fill=white](3.05,3.17) circle(2pt);
			\draw[black,fill=white](0.25,1.15) circle(2pt);
			\draw[black,fill=white](0.7,0.7) circle(2pt);

			\draw(0,0)node[below]{$x$};
			\draw(0.5+0.05,0.9)node[left]{$y$};
			\draw(0.25,1.15)node[left]{$u$};
			\draw(0.75,0.68)node[right]{$v$};
			\draw(3.05,3.2)node[above]{$z_2$};
			\draw(1.3,3.1)node[left]{$z_1$};
			\end{scope}
		\end{tikzpicture}
	\end{center}
 	\caption{\small Proof of Theorem \ref{thm:nonint}. $\xi+$ geodesics are in purple with medium thickness. $\xi-$ geodesics are in green and thin.}
 \label{fig:nonint}
 \end{figure}

If $y \notin \{u,v\}$, then  $u = \geo{m}{x}{\xi-}\prec y\prec  v = \geo{m}{x}{\xi+}$ (right picture in Figure \ref{fig:nonint}). The geodesics  $\geo{}{x}{\xi+}$ and $\geo{}{x}{\xi-}$ have already split and so cannot meet again by the uniqueness of finite geodesics.   For all $n \geq m$, $\geo{n}{x}{\xi-}\preceq \geo{n}{y}{\xi-}\preceq \geo{n}{y}{\xi+} \preceq \geo{n}{x}{\xi+}$.   Due to this ordering, the meeting of  $\geo{n}{x}{\xi-}$ and  $\geo{n}{y}{\xi+}$   at $z_1$ implies that  $\geo{n}{x}{\xi-}$ and  $\geo{n}{y}{\xi-}$ coalesce at or before $z_1$.   By the uniqueness of finite geodesics again,  $\geo{}{y}{\xi-}$ and $ \geo{}{y}{\xi+}$ agree from $y$ to $z_1$.    The same reasoning applies to $z_2$ and gives that   $\geo{}{y}{\xi-}$ and $ \geo{}{y}{\xi+}$ agree from $y$ to $z_2$ and that  $\geo{}{y}{\xi+}$ and $\geo{}{x}{\xi+}$  coalesce at $z_2$.  Thus now    $\geo{}{y}{\xi-}$ and $ \geo{}{y}{\xi+}$ agree from $y$ through both $z_1$ and $z_2$.     The coalescence of $\geo{}{x}{\xi-}$ with $\geo{}{y}{\xi-}$ and the coalescence of $\geo{}{y}{\xi+}$ with $\geo{}{x}{\xi+}$ then  force $\geo{}{x}{\xi-}$ and $\geo{}{x}{\xi+}$ to  meet again, contradicting what was said above.
 We have now shown that \eqref{thm:nonint.a} implies \eqref{thm:nonint.b}.\smallskip

\eqref{thm:nonint.b} implies \eqref{thm:nonint.c} by the directedness in Theorem \ref{thm1}\eqref{thm1:exist}.
It remains to prove the reverse implication under the regularity condition \eqref{g-reg}.  
Without loss of generality we can assume that $x\cdot\et\le y\cdot\et=k$. If $\geod{k}^x\prec y$, then the extremality of the geodesics $\geo{}{\abullet}{\xi\pm}$ in Theorem \ref{thm:extreme}  and the fact that
$\geod{}^x\cap\geod{}^y=\varnothing$ imply that $\geo{}{x}{\xi-} \cap \geo{}{y}{\xi+} = \varnothing$. Similarly, if $\geod{k}^x\succ y$, then we get $\geo{}{x}{\xi+} \cap \geo{}{y}{\xi-} = \varnothing$. 
\end{proof}

\begin{proof}[Proof of Theorem \ref{thm:1path}] 
The equivalence  \eqref{eq100}$\iff$\eqref{eq103} of Proposition \ref{pr:supp1}, together with the uniqueness of finite geodesics, gives Theorem \ref{thm:1path}. 
\end{proof} 


\begin{proof}[Proof of Lemma \ref{lm:isolated}]
For $\xi \in \Udense$, almost surely 
$\coal{\xi+}(x,y)=\coal{\xi-}(x,y)=\coal{\xi}(x,y)\in\Z^2$.  Proposition \ref{pr:supp2} implies that $\xi$   lies in the complement of the closed set $\supp{\mu_{x,y}}$.  
%
%
\end{proof}

Next, we prove Theorem \ref{thm:supp-tri} about the relation between the coalescence points and properties of the support of Busemann measures.

%

\begin{proof}[Proof of Theorem \ref{thm:supp-tri}]
Take $\w\in\Omega_0$.
Equivalence \eqref{thm:supp-tri:1} follows from Proposition \ref{pr:supp2}. 
Equivalence \eqref{thm:supp-tri:2} follows from the  equivalences in \eqref{thm:supp-tri:1} and \eqref{thm:supp-tri:3}.

The two equivalences of   \eqref{thm:supp-tri:3}  are proved the same way. 
We prove the first equivalence  in this form:  $\exists \eta\succ\xi$ such that $\abs{\mu_{x,y}}(\,]\xi, \eta[\,)=0$    $\iff$ $\coal{\xi+}(x,y)\in \bbZ^2$. 

 The implication $\implies$  is contained  in \eqref{eq100}$\implies$\eqref{eq101} of Proposition \ref{pr:supp1}.    

To prove $\rimplies$, let $k=x \cdot \et $ and $\ell=y \cdot \et$,  suppose $\coal{\xi+}(x,y)\in \bbZ^2$, and let $m = \coal{\xi+}(x,y) \cdot \et$.   Take a sequence $\eta_n \in \Udense$ with $\eta_n \searrow \xi$ as $n\to\infty$.   For sufficiently large $n$,   $\geo{k,m}{x}{\xi+} = \geo{k,m}{x}{\eta_n}$ and  $\geo{\ell,m}{y}{\xi+} = \geo{\ell,m}{y}{\eta_n}$, and hence $\coal{\eta_n}(x,y)=\coal{\xi+}(x,y)$.
Implication  \eqref{eq102}$\implies$\eqref{eq100} of Proposition \ref{pr:supp1} gives 
$\abs{\mu_{x,y}}(\,]\xi, \eta_n[\,)=0$.  
\end{proof}

When the jump process condition \eqref{cond:jumpcond} holds,  call the event in the statement of that condition $\Omega_0^5$. 
As noted when it was introduced, Theorem \ref{thm:+-coal}, which gives the equivalence between  \eqref{cond:jumpcond} and coalescence of $\xi\sigg$ geodesics, is essentially an immediate consequence of Theorem \ref{thm:supp-tri}.

\begin{proof}[Proof of Theorem \ref{thm:+-coal}]
Assume the jump process condition \eqref{cond:jumpcond}. Fix $\w\in\Omega_0\cap\Omega_0^5$, $x,y\in\Z^2$, and $\xi\in\ri\Uset$. 
If $\xi\not\in\supp{\mu_{x,y}}$, then Proposition \ref{pr:supp2} says that $\coal{\xi-}(x,y)=\coal{\xi+}(x,y)\in\Z^2$. In particular, $\geo{}{x}{\xi+}$ coalesces with $\geo{}{y}{\xi+}$ and $\geo{}{x}{\xi-}$ coalesces with $\geo{}{y}{\xi-}$. If, on the other hand, $\xi\in\supp{\mu_{x,y}}$, then it is an isolated point and now Theorem \ref{thm:supp-tri} says that $\coal{\xi\pm}(x,y)\in\Z^2$ (although now the two points are not equal). 
Again, $\geo{}{x}{\xi\pm}$ coalesces with $\geo{}{y}{\xi\pm}$, respectively. 
Statement  \eqref{thm:+-coal.b} is proved.

Now, assume \eqref{thm:+-coal.b} holds and let $\Omega_0^6$ be a full measure event on which the statement \eqref{thm:+-coal.b} holds. Let $\w\in\Omega_0\cap\Omega_0^6$, 
$x,y\in\Z^2$, and $\xi\in\supp{\mu_{x,y}}$. The fact that $\geo{}{x}{\xi\pm}$ and $\geo{}{y}{\xi\pm}$ coalesce, respectively, says that $\coal{\xi\pm}(x,y)\in\Z^2$. Since we assumed  $\xi\in\supp{\mu_{x,y}}$, Proposition \ref{pr:supp2} implies that the two coalescence  points $\coal{\xi\pm}(x,y)$  are not equal.
   Theorem \ref{thm:supp-tri} implies that $\xi$ is isolated.
\end{proof}

The proof of Lemma \ref{lm:downright} is delayed to the end of Section \ref{shockgraphs:pfs}.
When  the jump process condition  \eqref{cond:jumpcond} holds, define 
	\begin{align}\label{Omjump}
	\Omjump=\Omega_0\cap\Omega_0^5.  
	\end{align}

\begin{proof}[Proof of Theorem \ref{thm:Vcid}]
Part \eqref{thm:Vcid.1}.  Take $\w\in\Omega_0$.
Let $x \cdot \et = k$ and   $\xi = \cid(T_x\w)$. Take  $T_x\w$ in place of $\w$ in \eqref{cid},  let  $\zeta\to\cid(T_x\w)$, and use \eqref{cov-prop}, \eqref{coc-prop}, and \eqref{Busemann-limits},   to get $\B{\xi-}(x,x+e_2)\le\B{\xi-}(x,x+e_1)$ and $\B{\xi+}(x,x+e_1)\le\B{\xi+}(x,x+e_2)$. 
Then by definition $\geo{k}{x}{\xi+}  = \geo{k}{x}{\xi-} =x$,  $\geo{k+1}{x}{\xi+} = x+  e_1$, and $\geo{k+1}{x}{\xi-} =  x + e_2$.  Therefore we cannot have $\coal{\xi+}(x,x+e_i) = \coal{\xi-}(x,x+e_i) \in \bbZ^2$ for either $i \in \{1,2\}$ by uniqueness of finite geodesics. By Proposition \ref{pr:supp2},   $\xi \in \supp \mu_{x,x+e_1}\cap \supp \mu_{x,x+e_2}$.  

For the converse, for $\zeta\prec\cid(T_x\w)\prec\eta $ we have  $\B{\zeta\pm}(x,x+e_2)=\w_x=\B{\eta\pm}(x,x+e_1) $.  Thus 
\[  
  \supp\mu_{x,x+e_1} \subset \,] e_2\,, \cid(T_x\w)] \quad\text{and}\quad
\supp \mu_{x,x+e_2} \subset [\cid(T_x\w), e_1[ .  \] 
Consequently,  $\supp \mu_{x,x+e_1} \cap \supp\mu_{x,x+e_2} \subset \{\cid(T_x\w)\}$.
 
\medskip 

Part \eqref{thm:Vcid.2}.   
 It remains to show  $\aUset\subset\{\cid(T_x\w):x\in\Z^2\}$. 
Assume the jump process condition \eqref{cond:jumpcond}
and that $\w\in\Omjump$.
Suppose $\zeta\in\supp\mu_{x,y}$. 
By Theorem \ref{thm:supp-tri}\eqref{thm:supp-tri:2} the coalescence points  $\coal{\zeta\pm}(x,y)$ are distinct lattice points.    Hence the  geodesics $\geo{}{x}{\zeta+}$ and $\geo{}{x}{\zeta-}$  separate at some  point $z$ where then  $\cid(T_z\w)=\zeta$.  
\end{proof} 	

The next results relate   $\aUset$ to regularity properties of the shape function $\gpp$.  



\begin{lemma}\label{lm:V1}
The following holds for   all $\w\in\Omega_0$:  
for all  $\zeta\prec\eta$ in $\ri\Uset$,    $\,]\zeta,\eta[\,\cap\,\aUset\ne\varnothing$ if and only if $\nabla\gpp(\zeta+)\ne\nabla\gpp(\eta-)$. 
\end{lemma} 


\begin{proof} 

If $\nabla\gpp(\zeta+)\ne\nabla\gpp(\eta-)$,  Theorem \ref{thm:cif}\eqref{thm:cif.c} says   that $]\zeta,\eta[$ contains some $\cid(T_x\w)$, which  by Theorem \ref{thm:Vcid}\eqref{thm:Vcid.1}  is a member of $\aUset$. 


If $\nabla\gpp(\zeta+)=\nabla\gpp(\eta-)$, then by concavity,   $\nabla\gpp(\zeta+)=\nabla\gpp(\eta-)=\nabla\gpp(\xi\sig)$ for all $\xi\in\,]\zeta,\eta[$ and $\sigg\in\{-,+\}$.  By  Theorem \ref{thm:Bus}\eqref{thm:cocyexist:foo},    $\B{\xi\sig}(x,y,\w)$ is constant over $\xi\in\,]\zeta,\eta[$ and  $\sigg\in\{-,+\}$, for any $x,y\in\Z^2$ and $\w\in\Omega$.  Consequently,  for  any given $x$,  the geodesics  $\geo{}{x}{\xi\sig}$ match.    By Theorem \ref{thm1}\eqref{thm1:coal}, all these geodesics  coalesce on the event $\Omega_0$.   Hence  the coalescence points $\coal{\xi\sig}(x,y)$ also match.  By Theorem \ref{thm:supp-tri}\eqref{thm:supp-tri:1}, no point $\xi\in\,]\zeta,\eta[$ is a member of $\aUset$. 
\end{proof}

\begin{proof}[Proof of Theorem \ref{th:V1}]
Part \eqref{th:V1.a}. Let $\xi \in \Diff$.
Theorem \ref{thm1}\eqref{thm1:coal} says that almost surely $\coal{\xi\pm}(x,x+e_i)\in\Z^2$   for    $x\in\Z^2$ and $i\in\{1,2\}$. 
Theorem \ref{thm:Bus}\eqref{thm:cocyexist:g} says that there is  no $\pm$ distinction.
Hence $\mathbb{P}(\coal{\xi-}(x,x+e_i)=\coal{\xi+}(x,x+e_i)\in\Z^2)=1$ and therefore $\bbP(\xi \in \supp \mu_{x,x+e_i})=0$ by Proposition \ref{pr:supp2}.  A union bound implies that $\bbP(\exists x\in\Z^2,i\in\{1,2\}:\xi \in \supp \mu_{x,x+e_i})=0$. 
The cocycle property \eqref{coc-prop} implies then that $\bbP(\xi \in \aUset)=0$.

Let $\xi \in\ri\Uset\setminus\Diff$.  Theorem \ref{thm:cif}\eqref{thm:cif.c1} implies that $\xi$ is among $\{\cid(T_x\w): x\in\Z^2\}$, and these lie in $\aUset$ by   Theorem \ref{thm:Vcid}\eqref{thm:Vcid.1}. 


\medskip 

Part \eqref{th:V1.b}. 
 By definition,  the complement of $\fUset$ is the   union of the (at most countably many)  maximal open intervals $\,]\zeta,\eta[\,$ such that $\nabla\gpp(\zeta+)=\nabla\gpp(\eta-)$ and their differentiable endpoints.  Lemma \ref{lm:V1} together with  part \eqref{th:V1.a} proved above imply that 
$\P(\aUset\cap\fUset^c\ne\varnothing)=0$.  Hence almost surely, $\{\cid(T_x\w): x\in\Z^2\}\subset\aUset\subset\fUset$. 

 
For the density claim it is enough to prove that  $\{\cid(T_x\w): x\in\Z^2\}$ is dense in $\fUset$. 
Suppose first that   $\xi\in\fUset\cap\Diff$.  Then $\xi$ is not on a closed linear segment of $\gpp$, and hence for any $\zeta\prec\xi\prec\eta$ we have $\nabla\gpp(\zeta+)\ne\nabla\gpp(\eta-)$.  By Theorem \ref{thm:cif}\eqref{thm:cif.c} the open interval $\,]\zeta,\eta[\,$ contains a value $\cid(T_x\w)$.  
 The other case is $\xi\in\fUset\setminus\Diff$. Then  $\xi\in\{\cid(T_x\w): x\in\Z^2\}$  by Theorem \ref{thm:cif}\eqref{thm:cif.c1}. 
\end{proof}

The next proof, of Theorem \ref{th:V2},  identifies $\Uset \backslash \aUset$ in terms of directions in which (Busemann) geodesic uniqueness holds.

\begin{proof}[Proof of Theorem \ref{th:V2}]
Part \eqref{th:V2.a}. 
Fix $\w\in\Omega_0$ and $\xi\in\ri\Uset$.
Suppose first that there exists an $x \in \bbZ^2$ with the property that $\geo{}{x}{\xi+} \neq \geo{}{x}{\xi-}$.  These geodesics separate at some point $z$  where then $\xi=\cid(T_z\w)\in\aUset$. 
If, on the other hand, $\geo{}{x}{\xi+} = \geo{}{x}{\xi-}$ for all $x \in \bbZ^2$, but there exist $x$ and $y$ for which $\geo{}{x}{\xi}$ and $\geo{}{y}{\xi}$ do not coalesce, then   Proposition \ref{pr:supp2}  implies that $\xi \in \supp{\mu_{x,y}} \subset \aUset$.  

Conversely, suppose $\xi \in \aUset$ and let $x,y$ be such that $\xi \in \supp{\mu_{x,y}}$. Then by Theorem \ref{thm:nonint}, possibly after interchanging the roles of $x$ and $y$, we have $\geo{}{x}{\xi+} \cap \geo{}{y}{\xi-} = \varnothing$. In particular, these two geodesics do not coalesce. Part \eqref{th:V2.a} is proved.


Part \eqref{th:V2.c}.  Assume the jump process condition \eqref{cond:jumpcond} and let $\w\in\Omjump$. Suppose that 
$\geo{}{x}{\xi+}=\geo{}{x}{\xi-}=\geo{}{x}{\xi}$.   By Theorem \ref{thm:+-coal},  
  $\geo{}{y}{\xi\sig}$ coalesces with $\geo{}{x}{\xi}$ for all   $y\in\Z^2$ and both signs $\sigg\in\{-,+\}$.    By the uniqueness of finite geodesics $\geo{}{y}{\xi+}=\geo{}{y}{\xi-}$.  Now  all these geodesics coalesce.  Part \eqref{th:V2.a} implies  $\xi\notin\aUset$.  

Parts \eqref{th:V2.b} and \eqref{th:V2.d}  follow  from  \eqref{th:V2.a} and \eqref{th:V2.c}, respectively, because  under the regularity condition \eqref{g-reg}, Theorem \ref{thm:extreme} implies that the uniqueness of a $\Uset_\xi$-directed geodesic out of $x$ is equivalent to  $\geo{}{x}{\xi+}=\geo{}{x}{\xi-}$. 
\end{proof}

The next lemma completes the proof of Theorem \ref{th:capst}.  Recall the event $\Omega_0$ defined in \eqref{Om0}.

\begin{lemma}\label{lm:capst-aux}
Assume  the regularity condition \eqref{g-reg}. If $\w\in\Omega_0$ and $\xi \in \aUset$, then $\Uset_\xi = \{\xi\}$.
\end{lemma}

\begin{proof}
Take $\w\in\Omega_0$ and
suppose  $\Uset_\xi \neq \{\xi\}$. Recall the dense set of differentiability directions $\Udense$ introduced just before \eqref{Om0}. Because $\Uset_\xi$ is a line segment in $\Uset$, there exists a $\zeta\in\Udense \cap \Uset_\xi$.
By its definition, $\Omega_0\subset\Omega_\zeta^3$, where $\Omega_\zeta^3$ was introduced in Theorem \ref{thm1}. 
Theorem \ref{thm1}\eqref{thm1:cont} then implies that $\coal{\zeta-}(x,y) = \coal{\zeta+}(x,y)\in\Z^2$ for each pair $x,y$. 
Since $\zeta,\xi\in\Uset_\xi$ and we assumed \eqref{g-reg}, Theorem \ref{thm:Bus}\eqref{thm:cocyexist:foo} implies that for all $x,y \in \bbZ^2$, $\B{\zeta-}(x,y) = \B{\zeta+}(x,y)=\B{\xi-}(x,y)=\B{\xi+}(x,y)$. 
Consequently, $\coal{\xi-}(x,y) = \coal{\xi+}(x,y)\in\Z^2$ for all $x,y\in\Z^2$ and Theorem \ref{thm:supp-tri}\eqref{thm:supp-tri:1} shows that $\xi \notin \aUset$.
\end{proof}

\begin{proof}[Proof of Theorem \ref{th:exp1}]
Part  \eqref{th:exp1.a}. Theorem \ref{thm-Coupier} implies  that for $\xi\in\aUset$, $\geo{}{x}{\xi-}$ and $\geo{}{x}{\xi+}$  are the only $\xi$-directed geodesics out of $x$.  

Part  \eqref{th:exp1.b}.  Consider any  three  geodesics with the same asymptotic direction $\xi\in\ri\Uset$.  If $\xi\in(\ri\Uset)\setminus\aUset$ then by Theorem \ref{th:capst}\eqref{th:capst.c} all three coalesce.   If $\xi\in\aUset$ then by part  \eqref{th:exp1.a} 
  at least  two of these three geodesics must have the same sign $+$ or $-$.  By Theorem \ref{th:capst}\eqref{th:capst.d}  these two coalesce.   
  
Part \eqref{th:exp1.c}. Consider a sequence $v_n$ as in the first part of the statement and call $k = x \cdot \ehat$. From an arbitrary subsequence, extract a further subsequence $n_\ell$ so that $\geo{k,\infty}{x}{v_{n_\ell}}$ converges to a semi-infinite geodesic $\pi_{k,\infty}$  vertex-by-vertex. Let $\zeta \prec \xi \prec \eta$. Using the fact that $v_n/n \to \xi$ and directedness of $\geo{}{x}{\zeta+}$ and $\geo{}{x}{\eta-}$, for all sufficiently large $n$, we must have $\geo{n}{x}{\zeta+} \prec v_n \prec \geo{n}{x}{\eta-}$.  By uniqueness of finite geodesics, we must then have  $\geo{m}{x}{\zeta+} \prec \geo{m}{x}{v_n} \prec \geo{m}{x}{\eta-}$ for all $m \geq x \cdot \ehat$ and all such $n$. It then follows by sending $\zeta,\eta \to \xi$ that $\pi$ must be $\xi$-directed. Therefore, by part \eqref{th:exp1.a}, $\pi \in \{\geo{}{x}{\xi+}, \geo{}{x}{\xi-}\}$. Let $r = \spt{\xi}{x} \cdot \ehat$ and let $n_\ell$ be sufficiently large that $\geo{k,r+1}{x}{v_{n_\ell}} = \pi_{k,r+1}$. By definition of the competition interface, since $v_{n_\ell} \prec \varphi_{n_\ell}^{\spt{\xi}{x}}$ we must have $\pi_{r+1} =\spt{\xi}{x} + e_2$, which identifies $\pi$ as $\geo{}{x}{\xi-}$.  As the subsequence was arbitrary,  the result follows. The second claim is similar.

Part \eqref{th:exp1:d}. By Theorem \ref{thm1}\eqref{thm1:direct}, any semi-infinite geodesic emanating from $x$ is $\xi$-directed for some $\xi \in \Uset$. Combining part \eqref{th:exp1.a} and Theorem \ref{th:capst}\eqref{th:capst.c}, the only claim which remains to be shown is that $\geo{}{x}{e_i}$ is the only $e_i$-directed geodesic. This comes from Lemma \ref{lem:notriv}.   
\end{proof}

\section{Webs of instability: proofs} \label{sec:webpf}

Recall again event $\Omega_0$ constructed in  \eqref{Om0} and fix $\w\in\Omega_0$ throughout this section.  Unless otherwise indicated, an assumption of the form  $\zeta\preceq\eta$ includes the case $\zeta=\eta=\xi$. 

\subsection{Instability points and graphs}\label{shockgraphs:pfs}

\begin{proof}[Proof of Lemma \ref{lem:shockequiv}]
Suppose  $x^*$ is a  $[\zeta,\eta]$-instability point. Then there exists a direction $\xi\in[\zeta,\eta]\cap\supp\mu_{x+e_1,x+e_2}$, which by Theorem \ref{thm:nonint} implies $\geo{}{x+e_1}{\xi-}\cap\geo{}{x+e_2}{\xi+}=\varnothing$.  Then the ordering of geodesics implies $\geo{}{x+e_1}{\zeta-}\cap\geo{}{x+e_2}{\eta+}=\varnothing$.  


If $x^*$ is not a $[\zeta,\eta]$-instability point, then combining Propositions \ref{pr:supp1} and \ref{pr:supp2} 
we have that 
	\[z=\coal{\zeta-}(x+e_1,x+e_2)=\coal{\zeta+}(x+e_1,x+e_2)=\coal{\eta-}(x+e_1,x+e_2)=\coal{\eta+}(x+e_1,x+e_2)\in\Z^2,\]
$\geo{}{x+e_1}{\zeta\pm}$ and $\geo{}{x+e_1}{\eta\pm}$ all match until $z$ is reached, and $\geo{}{x+e_2}{\zeta\pm}$ $\geo{}{x+e_2}{\eta\pm}$ also all match until $z$ is reached. 
In particular, $z\in\geo{}{x+e_1}{\zeta-}\cap\geo{}{x+e_2}{\eta+}$.
\end{proof}

The following is immediate from the definitions and monotonicity.

\begin{lemma}\label{no-cross}
Let $\zeta\preceq\eta$ in $\ri\Uset$. A directed path  in $\dS{[\zeta,\eta]}$ can never cross a directed path in $\G{\eta+}$ from right to left {\rm (}i.e.~along a dual edge in the $-e_1$ direction{\rm )} nor a directed path in $\G{\zeta-}$ from above to below {\rm (}i.e.~along a dual edge in the $-e_2$ direction{\rm )}. 
\end{lemma}

The next lemma  characterizes the ancestors of an instability point in the graph $\dG{\cup[\zeta,\eta]}$.  

\begin{lemma}\label{lm:an1}
Let $\zeta\preceq\eta$ in $\ri\Uset$ and $x^*\in\shock{[\zeta,\eta]}$.  The following statements \eqref{an110} and \eqref{an111} are equivalent for any point $y^*\in\Z^{2*}$. 
\begin{enumerate} [label={\rm(\roman*)}, ref={\rm\roman*}]   \itemsep=3pt 
\item\label{an110}    There is a directed path from $y^*$ to $x^*$ in the graph $\dG{\cup[\zeta,\eta]}$. 
\item\label{an111}    $y^*\ge x^*$ and $y^*$ is between the two geodesics $\geo{}{x^*+\exstar}{\zeta-}$ and $\geo{}{x^*-\exstar}{\eta+}$, embedded as paths on $\R^2$.
\end{enumerate}
 \end{lemma}

\begin{proof} \eqref{an110}$\implies$\eqref{an111}.   By  Lemma \ref{no-cross} no directed path in $\dS{[\zeta,\eta]}$ can go from $y^*$ to $x^*$ unless
$y^*$ lies between $\geo{}{x^*+\exstar}{\zeta-}$ and $\geo{}{x^*-\exstar}{\eta+}$.

\medskip 

\eqref{an111}$\implies$\eqref{an110}.   We prove this by induction on $\abs{y^*-x^*}_1$.  The claim is trivial if $y^*=x^*$. Suppose $y^*\ge x^*$ is such that $y^*\neq x^*$  and $y^*$  is between $\geo{}{x^*+\exstar}{\zeta-}$ and $\geo{}{x^*-\exstar}{\eta+}$. If $y^*$ points to  both $y^*-e_1$ and $y^*-e_2$ in $\dS{[\zeta,\eta]}$, then since $y^*-e_i$ is between the two geodesics for at least one $i\in\{1,2\}$, the induction hypothesis implies that there is a directed path from $y^*$ to $x^*$ through this $y^*-e_i$. 

   Suppose next that  $y^*$ points to $y^*-e_1$ in $\dS{[\zeta,\eta]}$ but $y^*-e_1$  is not between the two geodesics. Then, on the one hand, $y^*-e_2$ must be between the geodesics and the induction hypothesis implies that there is a path from $y^*-e_2$ to $x^*$.
 On the other hand, $y^*-\etstar$ must point to $y^*+\exstar$ in $\G{\zeta-}$ to prevent $y^*-e_1$ from falling between the two geodesics.   
This implies that $y^*$ points to $y^*-e_2$ in $\dS{[\zeta,\eta]}$. Now we have a path from $y^*$ to $x^*$ through $y^*-e_2$.
  See the left plot in Figure \ref{fig:lineage}. The case when $y^*$ points to $y^*-e_2$ and the latter is not between the two geodesics is similar.
\end{proof}

The next lemma characterizes  $[\zeta,\eta]$-instability points as the {\it endpoints} of semi-infinite directed paths in $\dG{\cup[\zeta,\eta]}$.  Furthermore, such paths consist entirely of instability points.  

\begin{lemma} \label{lm:ans}  Let $\zeta\preceq\eta$ in $\ri\Uset$.
\begin{enumerate}   [label={\rm(\alph*)}, ref={\rm\alph*}]   \itemsep=3pt 
\item\label{ans110}   Let $\{x^*_k\}_{k\ge m}$ be any semi-infinite path on $\Z^{2*}$ such that  $x^*_{k+1}$ points to $x^*_k$ in $\dG{\cup[\zeta,\eta]}$  for each $k\ge m$.  
Then 
 $\{x^*_k\}_{k\ge m}\subset\shock{[\zeta,\eta]}$ and  as $k\to\infty$,   the limit points of $x^*_k/k$    lie in the interval $[\zetamin,\etamax]$. 
\item\label{ans111}  
Let  $x^*\in\Z^{2*}$ and $m=x^*\cdot\et$.  Then $x^*\in\shock{[\zeta,\eta]}$ if and only if there is a path $\{x^*_k\}_{k\ge m}$ on $\Z^{2*}$ such that $x^*_m=x^*$ and  for each $k\ge m$, $x^*_k\cdot\et=k$ and   $x^*_{k+1}$ points to $x^*_k$ in $\dG{\cup[\zeta,\eta]}$.   When this happens, the path $\{x^*_k\}_{k\ge m}$ satisfies part \eqref{ans110} above.  
 \end{enumerate} 
\end{lemma}

\begin{proof} Part \eqref{ans110}.  For each $k$, Lemma \ref{no-cross} implies that the geodesics 
$\geo{}{x^*_k+\exstar}{\zeta-}$ and $\geo{}{x^*_k-\exstar}{\eta+}$ are disjoint because they remain forever separated by the path $\{x^*_k\}_{k\ge m}$. 
Since the backward path $\{x^*_k\}_{k\ge m}$   is sandwiched between the geodesics $\geo{}{x^*_m+\exstar}{\zeta-}$ and $\geo{}{x^*_m-\exstar}{\eta+}$, Theorem \ref{thm1}\eqref{thm1:exist} implies that as $k\to\infty$ the limit points of $x^*_k/k$   
lie in the interval $[\zetamin,\etamax]$. 
\medskip 

Part \eqref{ans111}.   The {\it if} claim follows from part \eqref{ans110} of the lemma.  To prove the {\it only if} claim, suppose  $x^*\in\shock{[\zeta,\eta]}$.  Then the geodesics 
$\geo{}{x^*+\exstar}{\zeta-}$ and $\geo{}{x^*-\exstar}{\eta+}$ are disjoint. At every level $k>x^*\cdot\et$ we can choose a  point $y^*_k$ between the geodesics 
$\geo{}{x^*+\exstar}{\zeta-}$ and $\geo{}{x^*-\exstar}{\eta+}$, that is, a point $y^*_k\in\Z^{2*}$ such that $y^*\cdot\et=k$ and $\geo{k}{x^*+\exstar}{\zeta-}\prec y^*_k\prec\geo{k}{x^*-\exstar}{\eta+}$.  By Lemma \ref{lm:an1} there is a directed path in $\dG{\cup[\zeta,\eta]}$ from  each  $y^*_k$ to $x^*$.  Along some subsequence these directed paths converge to a semi-infinite directed path to $x^*$. 
\end{proof}

\begin{proof}[Proof of Theorem \ref{thm:shock1}] 

{\bf Step 1.}  
We show that $\shock{\xi}\neq\varnothing$ for 
any  $\xi\in\aUset$.  Since   $\xi\in\supp{\mu_{z,y}}$  for some $z,y\in\Z^2$,  the cocycle property \eqref{coc-prop} implies that  
  $\xi\in\supp{\mu_{x,x+e_i}}$ for some nearest-neighbor edge $(x,x+e_i)$.   Since $\mu_{x+e_1, x+e_2}=\mu_{x+e_1, x}+\mu_{x,x+e_2}$ is a sum of two positive measures there can be no cancellation, and hence    $\xi\in\supp{\mu_{x+e_1,x+e_2}}$ and thereby $x+\etstar\in\shock{\xi}$.  
 

\medskip
	
{\bf Step 2.}   We show that every  edge of $\dS{[\zeta,\eta]}$ that emanates from a point of $\shock{[\zeta,\eta]}$ is an edge of  $\shockG{[\zeta,\eta]}$.
Take $x^*\in\shock{[\zeta,\eta]}$.
Then $\geo{}{x^*+\exstar}{\zeta-}$ and $\geo{}{x^*-\exstar}{\eta+}$ are disjoint.
Suppose $x^*$ points to  $x^*-e_1$ in $\dS{[\zeta,\eta]}$. Then $x^*-\etstar$ points to  $x^*-\exstar$ in $\G{\eta+}$.  The geodesic 
$\geo{}{x^*-\etstar}{\eta+}$ takes first  an $e_1$ step and then follows $\geo{}{x^*-\exstar}{\eta+}$. Since $\geo{}{x^*-e_1+\exstar}{\zeta-}$ must 
always stay to the left of $\geo{}{x^*+\exstar}{\zeta-}$, it is prevented from touching $\geo{}{x^*-\etstar}{\eta+}=\geo{}{x^*-e_1-\exstar}{\eta+}$ and we see that $x^*-e_1\in\shock{[\zeta,\eta]}$. See the right plot in Figure \ref{fig:lineage}. The case when $x^*$ points to  $x^*-e_2$ in $\dS{[\zeta,\eta]}$ is similar.

\medskip  

{\bf Step 3.}  We conclude the proof.  Combining Lemma \ref{lm:ans}\eqref{ans110} with Step 2 implies that every bi-infinite directed path of the graph $\dG{\cup[\zeta,\eta]}$ is in fact a directed path of the graph $\shockG{[\zeta,\eta]}$.   

Conversely, let $x^*\in\shock{[\zeta,\eta]}$.   Lemma \ref{lm:ans} together with Step 2  implies that $x^*$ is the endpoint of a semi-infinite directed path in  $\shockG{[\zeta,\eta]}$ which is inherited from $\dG{\cup[\zeta,\eta]}$.    Step 2 implies that by following the edges of $\dG{\cup[\zeta,\eta]}$  from $x^*$ creates  an infinite down-left directed path   in the graph $\dG{\cup[\zeta,\eta]}$, and this path is a directed path also in $\shockG{[\zeta,\eta]}$.   In other words, every instability point $x^*\in\shock{[\zeta,\eta]}$ lies on a bi-infinite directed path of the graph $\shockG{[\zeta,\eta]}$ that was inherited from $\dG{\cup[\zeta,\eta]}$. 

The  $[\zetamin,\etamax]$-directedness of these paths comes from Lemma \ref{lm:ans}\eqref{ans110}. 
\end{proof} 

\begin{figure}
 	\begin{center}
 		 \begin{tikzpicture}[scale=1,shorten >=1pt,>={Latex[length=3.5mm,width=2mm]}]
		\begin{scope}
			\draw[->,>={Latex[length=3.5mm,width=3mm]},line width=3pt, color=nicosred](4.5,2.5)--(3.4,2.5);
			\draw[line width=3pt, color=nicosred](4,2.5)--(3.5,2.5);
			\draw[->,>={Latex[length=3.5mm,width=3mm]},line width=3pt, color=nicosred](4.5,2.5)--(4.5,1.6);
			\draw[line width=3pt, color=nicosred](4.5,2)--(4.5,1.5);
			\draw[nicosred,fill=nicosred](4.5,2.5) circle(3pt);
			\draw[nicosred,fill=nicosred](3.5,2.5) circle(3pt);
			\draw[nicosred,fill=nicosred](4.5,1.5) circle(3pt);
			\draw[nicosred,fill=nicosred](1/2,1/2) circle(3pt);
		 	\draw[->,line width=1pt, color=sussexg](4.1,2)--(4.1,3.1);
		 	\draw[line width=1pt, color=sussexg](4.1,2.5)--(4.1,3.2);
			\draw[color=sussexg, line width=1pt,->]plot[smooth] coordinates{(0,1)(0.5,1.3)(1.1,1.55)};
			\draw[color=sussexg, line width=1pt]plot[smooth] coordinates{(0.95,1.5)(1.5,1.6)(1.8,1.9)(4.1,2)};
			\draw[color=sussexp, line width=2pt,->]plot[smooth] coordinates{(1,0)(1.5,0)(2,0.3)(2.5,0.5)(4.5,0.6)(5.1,0.8)(5.7,1.2)};			
			
			\draw[black,fill=white](4.1,2) circle(3pt);
			\draw[black,fill=white](4.1,3.1) circle(3pt);
			\draw[black,fill=white](1,0) circle(3pt);
			\draw[black,fill=white](0,1) circle(3pt);

			\draw(4.5+.1,2.5)node[right]{$y^*$};
			\draw(0.5,0.5)node[left]{$x^*$};
			\draw(3.5,2.5)node[left]{$y^*-e_1$};
			\draw(4.6,1.5)node[right]{$y^*-e_2$};
			\draw(4,3.1)node[above]{$y^*+\exstar$};
			\draw(3.3,2)node[below]{$y^*-\etstar$};
		\end{scope}

		\begin{scope}[scale=0.8,shift={(12,0)}]
			\draw[nicosred,fill=nicosred](1/2,1/2) circle(3pt);
			\draw[nicosred,fill=nicosred](-1/2,1/2) circle(3pt);
			\draw[->,>={Latex[length=3.5mm,width=3mm]},line width=3pt, color=nicosred](1/2,1/2)--(-.3,1/2);
			\draw[line width=3pt, color=nicosred](0,1/2)--(-1/2,1/2);
			\draw(1/2+0.05,1/2+0.1)node[right]{$x^*$};
		 	\draw[->,line width=2pt, color=sussexp](0,0)--(0.85,0);
		 	\draw[line width=2pt, color=sussexp](0.5,0)--(1,0);
			\draw[color=sussexp, line width=2pt,->]plot[smooth] coordinates{(1,0)(1.5,0)(2,0.5)(2.5,0.75)};			
			\draw[color=sussexg, line width=1pt,->]plot[smooth] coordinates{(0,1)(0.5,1.4)(1.1,1.5)(1.5,1.9)(2,2.1)};
			\draw[color=sussexg, line width=1pt,->]plot[smooth] coordinates{(-1,1)(-0.5,1.1)(0.1,1.7)(0.7,1.8)(1,2)(1.5,2.3)};			
			\draw[black,fill=white](0,0) circle(3pt);
			\draw[black,fill=white](1,0) circle(3pt);
			\draw[black,fill=white](0,1) circle(3pt);
			\draw[black,fill=white](-1,1) circle(3pt);
			\draw(-1-0.05,1+0.1)node[left]{$x^*-e_1+\exstar$};
			\draw(-1/2-0.05,1/2)node[left]{$x^*-e_1$};
			\draw(1.5+0.05,0)node[below]{$x^*-\exstar$};
			\draw(-1/2,0)node[below]{$x^*-\etstar$};
			\draw(0.1,1)node[right]{$x^*+\exstar$};
		\end{scope}

		\end{tikzpicture}
 	 \end{center}
 	\caption{\small  The proofs of Lemma \ref{lm:an1} (left) and  Theorem \ref{thm:shock1}  (right). $\zeta-$ geodesics are in green and thin. $\eta+$ geodesics are in purple with medium thickness. Directed edges in $\dS{[\zeta,\eta]}$ are in red/thick.  White circles are points in $\Z^2$ while points in $\Z^{2*}$ are filled in (red).}
 \label{fig:lineage}
 \end{figure}
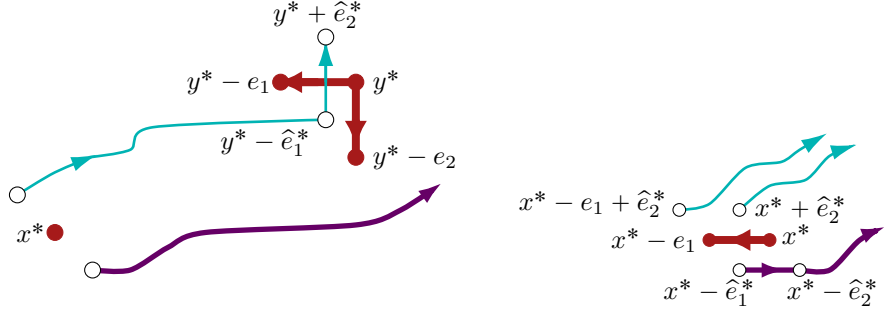
 
\begin{proof}[Proof of Theorem \ref{thm:shocks}] 

{\it Part \eqref{thm:shocks.c}.}  Let $x=x^*-\etstar$. If $x^*$ is a branch point in $\shockG{[\zeta,\eta]}$, then $\geo{}{x}{\zeta-}$ goes from $x$ to $x+e_2$ and $\geo{}{x}{\eta+}$ goes from $x$ to $x+e_1$, which is equivalent to 
$\B{\zeta-}(x+e_1,x+e_2)\le0\le\B{\eta+}(x+e_1,x+e_2)$, which in turn is equivalent to $\cid(T_x\w)\in[\zeta,\eta]$. 

Conversely, suppose $\cid(T_x\w)\in[\zeta,\eta]$. Reversing the above equivalences we see that 
$x^*\in\shock{[\zeta,\eta]}$ and points to both $x^*-e_1$ and $x^*-e_2$ in $\dS{[\zeta,\eta]}$. 
By Step 2 of the proof of Theorem  \ref{thm:shock1} above these edges are in  $\shockG{[\zeta,\eta]}$,  and hence $x^*$ is a branch point. 
\medskip

{\it Part \eqref{thm:shocks.d}}. Start with the case $\zeta\prec\eta$.
Let $\xi\in[\zeta,\eta]\cap\,\Udense$. 
 Then $\Omega_{\xi}^3\subset\Omega_0$ and parts \eqref{thm1:coal} and \eqref{thm1:biinf} of Theorem \ref{thm1} imply that $\G{\xi}$ is a tree that does not contain any bi-infinite up-right paths. 
 (Recall that for $\xi\in\Udense$ there is no $\pm$ distinction.)
This implies that $\dG{\xi}$ is a tree as well, i.e.\ all down-left paths of $\dG{\xi}$ coalesce.
Since $\dG{\xi}\subset\dG{\cup[\zeta,\eta]}$, one can follow the edges e.g.\ in $\dG{\xi}$ starting from $x^*$ and from $y^*$ to get to a coalescence point $z^*$ that will then be a descendant of both points in $\shockG{[\zeta,\eta]}$. 
The same argument can be repeated if $\zeta=\eta=\xi\in\aUset$ when condition \eqref{cond:biinf} holds, since then both $\dG{\xi\pm}$ are trees.
Claim \eqref{thm:shocks.d} is proved.\medskip

{\it Part \eqref{thm:shocks.e}}. Observe that for any $x^*,y^*\in\shock{[\zeta,\eta]}$, Theorem \ref{thm:+-coal} says that 
under the jump process condition \eqref{cond:jumpcond}, if $\w\in\Omjump$ (defined in \eqref{Omjump}), 
then the geodesics $\geo{}{x^*+\exstar}{\zeta-}$ and $\geo{}{y^*+\exstar}{\zeta-}$ coalesce, as do  $\geo{}{x^*-\exstar}{\eta+}$ and $\geo{}{y^*-\exstar}{\eta+}$.  
By Lemma \ref{lm:an1}, any point in $\shock{[\zeta,\eta]}$ that is between the two $+$ and $-$ coalesced geodesics is an ancestor to both $x^*$ and $y^*$. Such a point exists. For example, 
take a point $z$ on $\geo{}{x^*+\exstar}{\zeta-}$ above the coalescence levels, in other words, such that  $z\cdot\et\ge(\coal{\zeta-}(x^*+\exstar,y^*+\exstar)\cdot\et)\vee(\coal{\eta+}(x^*-\exstar,y^*-\exstar)\cdot\et)$.  
Since $\geo{}{z}{\eta+}$ coalesces with 
$\geo{}{x^*-\exstar}{\eta+}$, which does not touch $\geo{}{z}{\zeta-}$ (because this latter  is part of $\geo{}{x^*+\exstar}{\zeta-}$), $\geo{}{z}{\eta+}$ must separate from $\geo{}{z}{\zeta-}$ at some point $z'$.
The dual point $z'+\etstar$ is then in $\shock{[\zeta,\eta]}$ and is an ancestor to both $x^*$ and $y^*$. 
Part \eqref{thm:shocks.e} is proved.\medskip

{\it Part \eqref{thm:shocks.f}}.   The assumption is that $\zeta\prec\eta$ and $]\zeta,\eta[\,\cap\aUset\ne\varnothing$.   By Lemma \ref{lm:V1},  $\nabla\gpp(\zeta+)\ne\nabla\gpp(\eta-)$.   For any $z\in\Z^2$,  Theorem \ref{thm:cif}\eqref{thm:cif.c} gives a strictly increasing sequence $z<z_1<z_2<\dotsm$     such that 
  $\cid(T_{z_k}\w)\in\,]\zeta,\eta[$ for each $k$.   Then by \eqref{cid}, $\B{\zeta+}(z_k+e_1,z_k+e_2)<0<\B{\eta-}(z_k+e_1,z_k+e_2)$, which implies that $z^*_k=z_k+\etstar$ is a $[\zeta,\eta]$-instability point.
 Each such point is a branch point in $\shockG{[\zeta,\eta]}$ because $z_k$ points to $z_k+e_2$ in $\G{\zeta+}$, and hence also in $\G{\zeta-}$, and to $z_k+e_1$ in $\G{\eta-}$, and hence also in $\G{\eta+}$. 
 
%
%

The proof of the existence of infinitely many coalescence points in $\shockG{[\zeta,\eta]}$ follows from this and the first claim in part \eqref{thm:shocks.d} in a way similar to the proof below for the case of $\shockG{\xi}$ (but without the need for any extra conditions) and is therefore omitted.\medskip

{\it Part \eqref{thm:shocks.g}.}    Fix $\xi\in\aUset$ for the duration of the proof. 
Assume the jump process condition \eqref{cond:jumpcond}.
By Theorem \ref{thm:nonint} there exist $x,y\in\Z^2$ such that $\geo{}{x}{\xi-}\cap\geo{}{y}{\xi+}=\varnothing$ and then Theorem \ref{thm:+-coal} says that for any $z\in\Z^2$, the two geodesics $\geo{}{z}{\xi\pm}$ must separate at some point $z_1$ (in order to coalesce with $\geo{}{x}{\xi-}$ and $\geo{}{y}{\xi+}$, respectively). Uniqueness of finite geodesics implies that $\geo{}{z_1+e_i}{\xi\pm}$, $i\in\{1,2\}$, cannot touch. Thus, $z_1+\etstar\in\shock{\xi}$. 
Now define inductively $z_{n+1}$ to be the point where the geodesics $\geo{}{z_n+\et}{\xi\pm}$ separate.    Then for each $n$,  $z_{n+1}>z_n$ coordinatewise and  $z^*_n=z_n+\etstar$ is a point in $\shock{\xi}$.

Next, assume both  the jump process condition \eqref{cond:jumpcond} and the no bi-infinite geodesic condition \eqref{cond:biinf}. 
We prove the second claim of  part \eqref{thm:shocks.g} about infinitely many  coalescence points by mapping branch points injectively  into coalescence points as follows.

Given a branch point $x^*$,  let $\pi^*$ and $\bar\pi^*$ be the two innermost down-left paths out of $x^*$ along the directed graph $\shockG{\xi}$, defined by these rules: 
\be\label{aux783} \begin{aligned} 
 &\text{(i) $\pi^*$ starts with edge $(x^*, x^*-e_1)$,  follows the arrows of $\shockG{\xi}$, 
 and at vertices}\\ 
 &\text{where both $-e_1$ and $-e_2$ steps are allowed, it takes the  $-e_2$ step; }\\
   &\text{(ii) $\bar\pi^*$ starts with edge $(x^*, x^*-e_2)$,  follows the arrows of $\shockG{\xi}$,  and whenever}\\  
   &\text{both steps are available takes the  $-e_1$ step. }
  \end{aligned}\ee 

By part \eqref{thm:shocks.d}, $x^*-e_1$ and $x^*-e_2$ have a common descendant (this is where assumption \eqref{cond:biinf} is used).  By planarity,  the   paths $\pi^*$ and $\bar\pi^*$ must then meet at some point after $x^*$.    Let $z^*$  be their first common point after $x^*$, that is, the  point   $z^*\in(\pi^*\cap\bar\pi^*)\setminus\{x^*\}$  that maximizes $z^*\cdot\et$. This $z^*$ is the coalescence point that branch point $x^*$ is mapped to.  

We argue that the map $x^*\mapsto z^*$ thus defined  is one-to-one.   Two observations that help:  
\begin{itemize} 
\item   There cannot be any $\shock{\xi}$-points strictly inside the region bounded by $\pi^*$ and $\bar\pi^*$ between $x^*$ and  $z^*$.  By Theorem \ref{thm:shock1} such a point would lie on an  $\shockG{\xi}$ path that contradicts the choice of $\pi^*$ and $\bar\pi^*$ as the innermost paths between $x^*$ to $z^*$. 
\item  The last step that $\pi^*$ takes to reach $z^*$ is $-e_2$ and the last step of $\bar\pi^*$   is $-e_1$.   Otherwise $\pi^*$ and $\bar\pi^*$ would have met before $z^*$.  
\end{itemize} 

 Suppose another branch point $y^*\in\shock{\xi}$ distinct from $x^*$  maps to the same coalescence point $z^*$. Let the innermost paths  from $y^*$ to $z^*$ be $\gamma^*$ and $\bar\gamma^*$, defined by the same rules \eqref{aux783} but with $x^*$ replaced by $y^*$.      As observed, $\gamma^*$ and $\bar\gamma^*$  cannot  enter the region strictly between $\pi^*$ and $\bar\pi^*$.

\begin{figure}[h]
 	\begin{center}
 		 \begin{tikzpicture}[scale=0.7]
		 
		 \begin{scope}
		 	\draw[line width=1pt](0,0)--(1,0)--(1,1)--(3,1)--(3,2)--(4,2);
			\draw[middlearrow={stealth},line width=1pt](4,4)--(4,2);
			\draw[middlearrow={stealth},line width=1pt] (0,2) -- (0,0);			
		 	\draw[line width=1pt](0,2)--(1,2)--(1,3);
			\draw[middlearrow={stealth},line width=1pt] (3,3) -- (1,3);			
		 	\draw[line width=1pt](3,3)--(3,4);
			\draw[middlearrow={stealth},line width=1pt](4,4)--(3,4);
			\draw[fill=black](0,0) circle(4pt);
			\draw[fill=black](3,3) circle(4pt);
			\draw[fill=black](4,4) circle(4pt);
			\draw(0.3,2)node[above]{$\pi^*$};
			\draw(1,.3)node[right]{$\bar\pi^*$};
			\draw(2.7,3)node[above]{$x^*$};
			\draw(4.1,4)node[right]{$y^*$};
			\draw(-0.2,-0.1)node[below]{$z^*$};
			\draw(3.5,4)node[above]{$\gamma^*$};
			\draw(2.95,2.5)node[right]{$\gamma^*$};
			\draw(4,3)node[right]{$\bar\gamma^*$};
			\draw[middlearrow={stealth},line width=1pt] (3,3) -- (3,2);	
		\end{scope}
		
		\begin{scope}[shift={(7,0)}]
		 	\draw[line width=1pt](0,0)--(1,0)--(1,1)--(3,1)--(3,2)--(4,2);
			\draw[middlearrow={stealth},line width=1pt](4,4)--(4,2);
			\draw[middlearrow={stealth},line width=1pt] (0,2) -- (0,0);			
		 	\draw[line width=1pt](0,2)--(1,2)--(1,3)--(3,3);
			\draw[middlearrow={stealth},line width=1pt](4,4)--(2,4);
			\draw[middlearrow={stealth},line width=1pt](2,4)--(2,3);
			\draw[fill=black](0,0) circle(4pt);
			\draw[fill=black](3,3) circle(4pt);
			\draw[fill=black](4,4) circle(4pt);
			\draw(0.3,2)node[above]{$\pi^*$};
			\draw(1,.3)node[right]{$\bar\pi^*$};
			\draw(2.9,3)node[above]{$x^*$};
			\draw(4.1,4)node[right]{$y^*$};
			\draw(-0.2,-0.1)node[below]{$z^*$};
			\draw(2.8,4)node[above]{$\gamma^*$};
			\draw(4,3)node[right]{$\bar\gamma^*$};
			\draw[line width=1pt](3,3) -- (3,2);	
		\end{scope}

		\end{tikzpicture}
 	 \end{center}
 	\caption{\small  Illustration of the proof that the map $x^*\mapsto z^*$ is one-to-one.}
 \label{fig:thm:shocks.g}
 \end{figure}
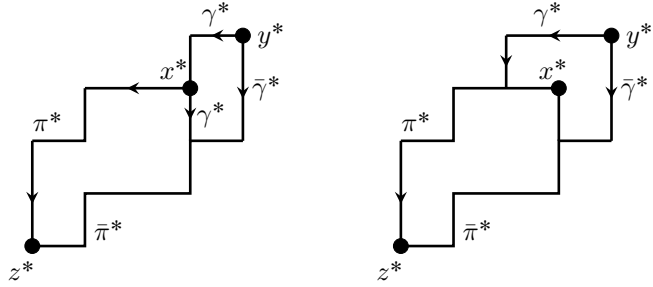

  Since   $\gamma^*$ uses the edge $(z^*+e_2, z^*)$, it must coalesce at some point with $\pi^*$.
 Point  $x^*$ itself cannot  lie on $\gamma^*$ because otherwise   \eqref{aux783} forces   $\gamma^*$ to take the edge $(x^*, x^*-e_2)$ and $\gamma^*$ cannot   follow $\pi^*$ to $z^*$.  This scenario is depicted by the left drawing in Figure \ref{fig:thm:shocks.g}.
 Thus $\gamma^*$ meets $\pi^*$ after $x^*$, at which point rule \eqref{aux783} forces them to  coalesce (right drawing in Figure \ref{fig:thm:shocks.g}).  
 
  Similarly,  $x^*$ cannot lie on $\bar\gamma^*$, and  $\bar\gamma^*$ meets $\bar\pi^*$ after $x^*$ at which point these coalesce (right drawing in Figure \ref{fig:thm:shocks.g}).
  
 Paths from $y^*$ cannot meet both $\pi^*$ and $\bar\pi^*$ while avoiding $x^*$ unless   $y^*>x^*$   holds coordinatewise.    
  It follows now  that $x^*$ must lie strictly inside the region bounded by $\gamma^*$ and $\bar\gamma^*$ between $y^*$ and  $z^*$, as illustrated by 
  the right drawing in Figure \ref{fig:thm:shocks.g}.
  But we already ruled out such a possibility.  These contradictions show that the map is one-to-one.


Since we already proved that under the jump process condition \eqref{cond:jumpcond} there are infinitely many branch points
in $\shockG{\xi}$, it now follows that there are also infinitely many coalescence points and part \eqref{thm:shocks.g} is proved.
%
\end{proof}

\begin{proof}[Proof of the claim in Remark \ref{rk:erg2}] 
%
It suffices to consider the case where $]\zeta,\eta[\,\cap\aUset=\varnothing$ but $\{\zeta,\eta\}\cap\aUset\ne\varnothing$. 
By Theorem \ref{th:V1}\eqref{th:V1.a}, the differentiable endpoints of the (countably many) linear segments of $\gpp$ are all outside $\aUset$.
By Theorem \ref{th:V1}\eqref{th:V1.b} we know $]\zeta,\eta[$ must be inside a linear segment. 
Thus, it must be the case that $\{\zeta,\eta\}\cap\aUset\setminus\Diff\ne\varnothing$. 
Suppose, without loss of generality, that $\zeta$ is in this intersection. 
Then Theorem \ref{thm:cif}\eqref{thm:cif.c1} 
implies the existence of infinitely many $x\in\Z^2$ with $\cid(T_x\w)=\zeta\in\aUset$ and part Theorem \ref{thm:shocks}\eqref{thm:shocks.c} says that the corresponding dual points $x^*$ are all branch points in $\shockG{\zeta}\subset\shockG{[\zeta,\eta]}$. The claim about coalescence points follows from the just proved infinite number of branch points, combined with the first claim in part \eqref{thm:shocks.d}, similarly to the way the claim is proved in Theorem \ref{thm:shocks}\eqref{thm:shocks.g}.
\end{proof}

In words, the next result says that there are no semi-infinite horizontal or vertical paths in any of the instability graphs $\shock{[\zeta,\eta]}$. The idea behind the proof is that the existence of such a path would force the existence of a semi-infinite horizontal or vertical path in one of the geodesic graphs $\G{\xi \sig}$ for some $\sig \in  \{+,-\}$ and $\xi \in \ri \Uset$. This is ruled out by the law of large numbers behavior of the Busemann functions.

\begin{lemma}\label{hor-red}
For any $\w\in\Omega_0$, $\zeta\preceq\eta$, and $i\in\{1,2\}$, there does not exist an $x^*\in\shock{[\zeta,\eta]}$ such that $x^*-ne_i\in\ans{[\zeta,\eta]}{x^*-(n+1)e_i}$ for all $n\in\Z_+$ and nor does there exist an $x^*\in\shock{[\zeta,\eta]}$ such that $x^*+(n+1)e_i\in \ans{[\zeta,\eta]}{x^*+ne_i}$ for all $n\in\Z_+$.
\end{lemma}

\begin{proof}
We prove the result for $i=1$, $i=2$ being similar. We also only work with paths of the first type. The other type can be treated similarly.

The existence of a path of the first type, with $i=1$, implies that $x^*-ne_1-\etstar$ points to  $x^*-(n-1)e_1-\etstar$ in $\G{\eta+}$ 
for all $n\in\Z_+$. But this implies that $\B{\eta+}(x^*-ne_1-\etstar,x^*-(n-1)e_1-\etstar)=\w_{x^*-ne_1-\etstar}$, for all $n\in\Z_+$. Take any sequence $\eta_m\in\Udense$ such that $\eta_m\searrow\eta$.
Then \eqref{mono} and \eqref{coc-prop} imply that
	\[\sum_{k=1}^n\w_{x^*-ke_1-\etstar}=\B{\eta+}(x^*-ne_1-\etstar,x^*-\etstar)\ge\B{\eta_m}(x^*-ne_1-\etstar,x^*-\etstar).\]
Divide by $n$ and apply the ergodic theorem on the left-hand side and \eqref{eq:erg-coc} on the right-hand side to get $\E[\w_0]\ge e_1\cdot\nabla\gpp(\eta_m)$ for all $m$.
Take $m\to\infty$ to get $\E[\w_0]\ge e_1\cdot\nabla\gpp(\eta+)$. It follows from Martin's estimate of the asymptotic behavior of the shape function near the boundary of $\Uset$, \cite[Theorem 2.4]{Mar-04}, along with concavity that this cannot happen.
\end{proof}

\begin{proof}[Proof of Lemma \ref{lm:downright}]
A general step of the path can be decomposed as 
$x_{i+1}-x_i=\sum_k (y_{k+1}-y_k)$ where each $y_{k+1}-y_k\in\{e_1, -e_2\}$.  Then each $\mu_{y_k, y_{k+1}}$ is a negative measure, and consequently $\supp\mu_{x_i, x_{i+1}}=\bigcup_k \supp\mu_{y_k, y_{k+1}}$.  Thus we may assume   that the path satisfies $x_{i+1}-x_i\in\{e_1, -e_2\}$ for all $i$.  

One direction is clear:
	$\bigcup_{i\in\Z} \supp\mu_{x_i,x_{i+1}}\subset\aUset$. 
	
For the other direction, take $\xi\in\aUset$.  By  Theorem \ref{thm:shock1},   there is a bi-infinite up-right reverse-directed  path $x^*_{-\infty,\infty}$   through $x^*$ in $\shockG{\xi}$ with  increments in $\{e_1,e_2\}$. By Lemma \ref{hor-red} this path must cross any down-right lattice path $x_{-\infty,\infty}$. This means that there exists an $i\in\Z$ such that either $x_{i+1}-x_i=e_1$ and $x_i+\etstar$ points to  $x_i-\exstar$ in $\shockG{\xi}$, i.e.\ $x_i$ points to $x_i+e_2$ in $\G{\xi-}$, or $x_{i+1}-x_i=-e_2$ and $x_i-\exstar$ points to  $x_i-\etstar$ in $\shockG{\xi}$, i.e.\ $x_i$ points to $x_i+e_1$ in $\G{\xi+}$. 
In the former case, $\geo{}{x_i}{\xi-}$ goes from $x_i$ to $x_i+e_2$ and from there 
it never touches $\geo{}{x_{i+1}}{\xi+}=\geo{}{x_i+e_1}{\xi+}$, since $x_i+\etstar\in\shock{xi}$. Consequently, in this case Theorem \ref{thm:nonint} says that $\xi\in\supp\mu_{x_i,x_{i+1}}$. The other case is similar and again gives $\xi\in\supp\mu_{x_i,x_{i+1}}$.  This proves Lemma \ref{lm:downright}.
\end{proof}

\subsection{Density of instability points on the lattice}
For $\zeta\preceq\eta$ in $\ri\Uset$, $x\in\Z^2$, and $i\in\{1,2\}$ 
let 
\be\label{df:rho5} \rho_x^i(\zeta,\eta)=\one\bigl\{[\zeta,\eta]\cap\supp{\mu_{x,x+e_i}}\neq\varnothing\bigr\}.\ee 
We write $\rho_x^i(\xi)$ for $\rho_x^i(\xi,\xi)$. 
By definition, $x+\etstar\in\shock{[\zeta,\eta]}$ if and only if $\rho_x^1(\zeta,\eta)$ and $\rho_x^2(\zeta,\eta)$ are not both $0$.
By Lemma \ref{lm:cross} below,  $\rho_x^i(\zeta,\eta)=1$ is equivalent to $x+\etstar$ pointing to $x+\etstar-e_{3-i}$ in $\shockG{[\zeta,\eta]}$. 
Also, $\rho_x^1(\xi)=\rho_x^2(\xi)=1$ if and only if $\cid(T_x\w)=\xi$.
Let
	\[\kappa_i(\zeta,\eta)=\P\{\rho_0^i(\zeta,\eta)=1\}.\]
Since $\supp{\mu_{x,x+e_i}}$ is by definition closed,  $\kappa_i$ is left-continuous in $\zeta$ and right-continuous in $\eta$. 
Furthermore, by Theorem \ref{th:V1} $\kappa_i$ is continuous in each argument at points of differentiability of $\gpp$. 
Again, we write $\kappa_i(\xi)$ for $\kappa_i(\xi,\xi)$.
We thus have
	\[\kappa_i(\xi)=\lim_{\zeta\nearrow\xi,\eta\searrow\xi}\kappa_i(\zeta,\eta).\]
By Theorem \ref{th:V1}, $\xi \in \Diff$ if and only if $\kappa_i(\xi)=0$ for any (and hence both) $i\in\{1,2\}$.
Let 
	\[\kappa_{12}(\zeta,\eta)=\P\{\rho_0^1(\zeta,\eta)=\rho_0^2(\zeta,\eta)=1\}=\P\{\cid\in[\zeta,\eta]\}\]
and write $\kappa_{12}(\xi)$ for $\kappa_{12}(\xi,\xi)$. The last equality above follows because if $\cid\notin[\zeta,\eta]$, then by recovery \eqref{rec-prop2} and by the Busemann characterization \eqref{cid} of $\cid$, one of the processes $\xi\mapsto\B{\xi\pm}(0,e_i)$ for $i\in\{1,2\}$ is constant for $\xi\in[\zeta,\eta]$. 


The next result essentially follows from the ergodic theorem and 
gives the density of horizontal and vertical edges, instability points, branch points, and coalescence points. 


\begin{lemma}\label{lm:den}
Assume the regularity condition \eqref{g-reg}.
There exists a $T$-invariant event $\Omega_0'\subset\Omega_0$ with $\P(\Omega_0')=1$ and 
such that for all $\w\in\Omega+0'$, $i,j\in\{1,2\}$, $a,b,a',b'\in\{0,1\}$ with $(b-a)(b'-a')\ne0$, and for all $\zeta\preceq\eta$ in $\ri\Uset$, we have
	\begin{align}
		\begin{split}
		\label{den:edge}
		&\lim_{n\to\infty}\frac1{\abs{b-a}n}\sum_{k=-an}^{bn}\rho_{ke_i}^j(\zeta,\eta)\\
		&\qquad\qquad=\lim_{n\to\infty}\frac1{\abs{(b-a)(b'-a')}n^2}\sum_{x\in[-an,bn]\times[-a'n,b'n]}\!\!\!\!\!\!\!\!\!\!\!\!\!\rho_x^j(\zeta,\eta)=\kappa_j(\zeta,\eta),
		\end{split}\\
		\begin{split}
		\label{den:red}
		&\lim_{n\to\infty}\frac1{\abs{b-a}n}\sum_{k=-an}^{bn}\one\{ke_i+\etstar\in\shock{[\zeta,\eta]}\}\\
		&\qquad\qquad=\lim_{n\to\infty}\frac1{(b-a)(b'-a')n^2}\sum_{x\in[-an,bn]\times[-a'n,b'n]}\!\!\!\!\!\!\!\!\!\!\!\!\!\one\{x+\etstar\in\shock{[\zeta,\eta]}\}\\
		&\qquad\qquad=\kappa_1(\zeta,\eta)+\kappa_2(\zeta,\eta)-\kappa_{12}(\zeta,\eta),
		\end{split}\\
		\begin{split}
		\label{den:cid}
		&\lim_{n\to\infty}\frac1{\abs{b-a}n}\sum_{k=-an}^{bn}\rho_{ke_i}^1(\zeta,\eta)\rho_{ke_i}^2(\zeta,\eta)\\
		&\qquad\qquad=\lim_{n\to\infty}\frac1{\abs{(b-a)(b'-a')}n^2}\sum_{x\in[-an,bn]\times[-a'n,b'n]}\!\!\!\!\!\!\!\!\!\!\!\!\!\rho_x^1(\zeta,\eta)\rho_x^2(\zeta,\eta)=\kappa_{12}(\zeta,\eta),	
		\end{split}\\
		&\text{and}\notag\\
		\begin{split}
		\label{den:coal}
		&\lim_{n\to\infty}\frac1{\abs{b-a}n}\sum_{k=-an}^{bn}\rho_{ke_i}^1(\zeta,\eta)\rho_{(k+1)e_i-e_{3-i}}^2(\zeta,\eta)\\
		&\qquad\qquad=\lim_{n\to\infty}\frac1{\abs{(b-a)(b'-a')}n^2}\sum_{x\in[-an,bn]\times[-a'n,b'n]}\!\!\!\!\!\!\!\!\!\!\!\!\!\rho_{x-e_1}^1(\zeta,\eta)\rho_{x-e_2}^2(\zeta,\eta)=\kappa_{12}(\zeta,\eta).
		\end{split}
               	\end{align}
All of the above limits are positive if and only if $\nabla\gpp(\zeta+)\ne\nabla\gpp(\eta-)$.
\end{lemma}


 \begin{proof}
As explained in Remark \ref{rk:ergodicity}, under the regularity condition \eqref{g-reg}, the Busemann process is a measurable function of $\{\w_x:x\in\Z^2\}$. Thus, by the ergodic theorem,  there exists a $T$-invariant event 
$\Omega_0'\subset\Omega_0$ with $\P(\Omega_0')=1$ and such that for $\w\in\Omega_0'$ the limits (\ref{den:edge}-\ref{den:coal}) hold for
 all $\zeta,\eta\in\Udense\cup\bigl((\ri\Uset)\setminus\Diff\bigr)$.
 
 To justify the equality of the limit in \eqref{den:coal} with the one in \eqref{den:cid} observe that since every instability point must have at least one descendant and at least one ancestor, we have
	\[\P\{-\etstar\in\shock{[\zeta,\eta]}\}=\P\{\rho_{-e_1}^1(\zeta,\eta)=1\}+\P\{\rho_{-e_2}^2(\zeta,\eta)=1\}-\P\{\rho_{-e_1}^1(\zeta,\eta)=\rho_{-e_2}^2(\zeta,\eta)=1\}\]
and
	\[\P\{\etstar\in\shock{[\zeta,\eta]}\}=\P\{\rho_0^1(\zeta,\eta)=1\}+\P\{\rho_0^2(\zeta,\eta)=1\}-\P\{\rho_0^1(\zeta,\eta)=\rho_0^2(\zeta,\eta)=1\}.\]
By shift invariance, the first three probabilities in the first display match the corresponding three probabilities in the second display.  Thus, 
	\[\P\{\rho_{-e_1}^1(\zeta,\eta)=\rho_{-e_2}^2(\zeta,\eta)=1\}=\kappa_{12}(\zeta,\eta).\]

We now prove the first limit in \eqref{den:edge}, the rest of the limits in the statement of the lemma being similar.
Take $\w\in\Omega_0'$ and any $\zeta\prec\eta$ in $\ri\Uset$. Suppose first $\gpp$ is differentiable at both $\zeta$ and $\eta$.
Take sequences $\zeta_m'\prec\zeta\prec\zeta_m\prec\eta_m\prec\eta\prec\eta_m'$ with $\zeta_m', \zeta_m, \eta_m, \eta_m' \in \Udense$ and use monotonicity and the continuity of $\kappa_j$ to get
	\begin{align*}
	\kappa_j(\zeta_m,\eta_m)
	&=\lim_{n\to\infty}\frac1{(b-a)n}\sum_{k=-an}^{bn}\rho_{ke_i}^j(\zeta_m,\eta_m)\\
	&\le\varliminf_{n\to\infty}\frac1{(b-a)n}\sum_{k=-an}^{bn}\rho_{ke_i}^j(\zeta,\eta)
	\le \varlimsup_{n\to\infty}\frac1{(b-a)n}\sum_{k=-an}^{bn}\rho_{ke_i}^j(\zeta,\eta)\\
	&\le \lim_{n\to\infty}\frac1{(b-a)n}\sum_{k=-an}^{bn}\rho_{ke_i}^j(\zeta_m',\eta_m')
	=\kappa_j(\zeta_m',\eta_m').
	\end{align*}
Taking $m\to\infty$ and using continuity of $\kappa_j$ at $\zeta$ and $\eta$ 
gives that the above liminf and limsup are equal to $\kappa_j(\zeta,\eta)$.
The same proof works if $\zeta=\eta$ is a point of differentiability of $\gpp$. In this case, we can use $0$ as a lower bound and for the upper bound we have $\kappa_j(\zeta)=\kappa_j(\eta)=0$. 	

Next, suppose $\zeta$ is a point of non-differentiability of $\gpp$, but $\eta$ is still a point of differentiability.
We can repeat the same argument as above, but this time only using the sequences $\eta_m$ and $\eta'_m$ and the intervals $[\zeta,\eta_m]$ and $[\zeta,\eta_m']$ for the upper and lower bounds, 
because $\zeta$ has been included in the set $\Udense\cup\bigl((\ri\Uset)\setminus\Diff\bigr)$. 
A similar argument works if $\zeta$ is a point of differentiability but $\eta$ is not. When $\gpp$ is not differentiable at both $\zeta$ and $\eta$, the claimed limits follow from the choice of $\Omega_0'$.
\end{proof}

\begin{proof}[Proof of Proposition \ref{prop:trichotomy}]
The claim follows from Lemma \ref{lm:den} 
\end{proof}

\subsection{Flow of Busemann measures}\label{sec:Busflowpf}

\begin{proof}[Proof of Theorem \ref{th:flowS}]
The vertex set of $\cB^*_{[\zeta, \eta]}$ is by definition the same as that of $\shockG{[\zeta,\eta]}$. That the edges also agree follows from Lemma \ref{lm:cross} below. 
\end{proof}

\begin{lemma}\label{lm:cross}
For $i\in\{1,2\}$, $[\zeta,\eta]\cap\supp{\mu_{x,x+e_i}}\neq\varnothing$  if and only if 
$(x+\etstar, x+\etstar-e_{3-i})$ is a directed  edge in the graph $\shockG{[\zeta,\eta]}$.  
\end{lemma}

\begin{proof}
We argue the case of $i=1$.
Assume first that $[\zeta,\eta]\cap\supp{\mu_{x,x+e_1}}\neq\varnothing$. 
From $\mu_{x+e_1, x+e_2}=\mu_{x+e_1, x}+\mu_{x,x+e_2}$ and $\mu_{x-\ex, x}=\mu_{x-\ex, x+e_1}+\mu_{x+e_1, x}$ (sums of positive measures) we see that both $x+\etstar, x+\etstar-e_{2}\in\shock{[\zeta, \eta]}$. 

Suppose $\xi\in[\zeta,\eta]\cap\supp{\mu_{x,x+e_1}}$.  By Theorem \ref{thm:nonint}, $x$ must point to $x+e_2$ in $\G{\xi-}$,   which forces the same   in $\G{\zeta-}$.  
Thus    $x+\etstar$ points to $x+\etstar-e_2$ in $\dG{\zeta-}$ and hence also in $\dS{[\zeta,\eta]}$.

Conversely, if $x+\etstar\in\shock{[\zeta,\eta]}$ then $\geo{}{x+e_2}{\zeta-}$ and $\geo{}{x+e_1}{\eta+}$ do not intersect. If furthermore $x+\etstar$ points to $x+\etstar-e_2$ in $\dS{[\zeta,\eta]}$, then $x$ points to $x+e_2$ in $\G{\zeta-}$ and hence $\geo{}{x}{\zeta-}$ joins $\geo{}{x+e_2}{\zeta-}$ and does not intersect $\geo{}{x+e_1}{\eta+}$.   
 
 Let $\zeta'\prec\zeta$ and $\eta'\succ\eta$.  By geodesic ordering \eqref{path-ordering}, $\geo{}{x}{\zeta'+}$ and $\geo{}{x+e_1}{\eta'-}$ are disjoint.  In particular, the coalescence points $\coal{\zeta'+}(x,x+e_1)$ and $\coal{\eta'-}(x,x+e_1)$ cannot coincide on $\Z^2$.  By Proposition \ref{pr:supp1}, $]\zeta', \eta'[$  intersects  $\supp{\mu_{x,x+e_1}}$. Since this holds  for every choice of $]\zeta', \eta'[\,\supset [\zeta,\eta]$, it follows that also 
  $[\zeta,\eta]$ intersects $\supp{\mu_{x,x+e_1}}$.
\end{proof}

\begin{proof}[Proof of Proposition \ref{lm:K9}]
Suppose $x\zesim y$.    Since  $\supp\mu_{x,y}$ is a closed subset of $\ri\Uset$  and $[\zeta, \eta]$  a compact set,   we can find $\zeta'\prec\zeta$ and $\eta'\succ\eta$ such that  $\abs{\mu_{x,y}}(\,]\zeta', \eta'[\,)=0$.   Then by Proposition \ref{pr:supp1}, there exists $z\in\Z^2$ such that all geodesics $\geo{}{x}{\xi\sig}$ and $\geo{}{y}{\xi\sig}$ for $\xi\in[\zeta, \eta]$ and $\sigg\in\{-,+\}$ meet at $z$.  Thus  $x$ and $y$ are in the same subtree of the graph $\G{\cap[\zeta, \eta]}$. 

Conversely, suppose $x$ and $y$ are two distinct points  in the same subtree $\cK$ of the graph $\G{\cap[\zeta, \eta]}$.   In this tree the following holds. 
\be\label{aux775} \begin{aligned} 
&\text{In $\cK$ there is a point $z$ and a path $\pi$ from $x$ to $z$  and a path $\pi'$ from $y$ to $z$ }\\[-2pt] 
&\text{such that $z$ is the first common point of $\pi$ and $\pi'$. For each $\xi\in[\zeta, \eta]$ and }\\[-2pt]
&\text{both signs  $\sigg\in\{-,+\}$, all the geodesics $\geo{}{x}{\xi\sig}$ follow $\pi$  from $x$ to $z$,  }\\[-2pt]
&\text{and all the geodesics $\geo{}{y}{\xi\sig}$  follow $\pi'$  from $y$ to $z$.} 
\end{aligned}\ee
  Consequently each  $\xi\in[\zeta, \eta]$ satisfies $\coal{\xi-}(x,y)=\coal{\xi+}(x,y)=z$.  By Proposition \ref{pr:supp2} each  $\xi\in[\zeta, \eta]$ lies outside $\supp\mu_{x,y}$. 
\end{proof}

\begin{figure}[h]
 	\begin{center}
 		 \begin{tikzpicture}[scale=1]
		 
		 	\draw(0,0)--(1.5,0)--(1.5,1.5)--(0,1.5)--(0,0);
		 	\draw[line width=1.5pt](0,1.5)--(0.5,1.5)--(0.5,1.7)--(0.9,1.7)--(0.9,1.9)--(1.4,1.9)--(1.4,2.1)--(1.8,2.1)--(1.8,2.3)--(2.2,2.3)--(2.2,2.5)--(2.7,2.5);
			\draw[fill=black](0,1.5) circle(2pt);
			\draw(0,1.5)node[left]{$x$};
			\draw(0.9,1.9)node[above]{$\pi$};
			\draw[line width=1.5pt](1.5,0)--(1.5,0.6)--(1.8,0.6)--(1.8,1)--(2.1,1)--(2.1,1.6)--(2.4,1.6)--(2.4,2.1)--(2.7,2.1)--(2.7,2.5);
			\draw[fill=black](1.5,0) circle(2pt);
			\draw(1.5,0)node[right]{$y$};
			\draw(2.1,1.3)node[right]{$\pi'$};
			\draw[line width=1.5pt](0.4,0.8)--(0.8,0.8)--(0.8,1.2)--(1.2,1.2)--(1.2,1.7)--(1.6,1.7)--(1.6,2.1);
			\draw[fill=black](0.4,0.8) circle(2pt);
			\draw(0.38,0.8)node[below]{$u$};
			\draw(1.2,1.25)node[below]{$\pi''$};
			\draw[fill=black](2.7,2.5) circle(2pt);
			\draw(2.85,2.5)node[above]{$z$};
		
		\end{tikzpicture}
 	 \end{center}
 	\caption{\small  Proof of Lemma \ref{lm:K10}.}
 \label{fig:K10}
 \end{figure}

\begin{proof}[Proof of Lemma \ref{lm:K10}]
The hypotheses imply that, by switching $x$ and $y$ around if necessary, $x\cdot e_1\le y\cdot e_1$ and $x\cdot e_2\ge y\cdot e_2$.    Let $z, \pi, \pi'$ be as in \eqref{aux775}.  
Let $u$ be any point of  $\lzb x\wedge y, x\vee y\rzb$.    By planarity, 
 each geodesic $\geo{}{u}{\xi\sig}$ for $ \xi\in[\zeta, \eta]$ and  $\sigg\in\{-,+\}$  must eventually intersect $\pi$ or $\pi'$ and then follow this to $z$.   See Figure \ref{fig:K10}.
 By  uniqueness of finite geodesics,  all  these geodesics  $\geo{}{u}{\xi\sig}$  follow the same path $\pi''$ from $u$ to $z$.     Thus $\pi''$ is part of the graph $\G{\cap[\zeta, \eta]}$, and since it comes together  with $\pi$ and $\pi'$ at $z$,  $\pi''$ is part of the same  subtree $\cK$.  
\end{proof}

\begin{proof} [Proof of Lemma \ref{lm:K12}] 
Suppose $x$ is such a vertex but  $\cK\subset\{y: y\ge x\}$ fails. 
We claim that then there necessarily exists a vertex $y\in\cK$ such that $x$ and $y$ satisfy the hypotheses of Lemma \ref{lm:K10} and one of $\{x- e_1, x- e_2\}$ lies in  $\lzb x\wedge y, x\vee y\rzb$. This leads to a contradiction. 

To verify the claim,  
pick $y\in\cK$ such that $y\ge x$ fails.  If $y<x$ also fails,  there are two possible cases: 
\begin{enumerate}     [label={\rm(\roman*)}, ref={\rm(\roman*)}]   \itemsep=3pt  
\item    $y\cdot e_1< x\cdot e_1$ and $y\cdot e_2\ge  x\cdot e_2$, in which case  $x- e_1\in\lzb x\wedge y, x\vee y\rzb\subset\cK$;
\item  $y\cdot e_1\ge x\cdot e_1$ and $y\cdot e_2< x\cdot e_2$, in which case  $x- e_2\in\lzb x\wedge y, x\vee y\rzb\subset\cK$.  
\end{enumerate} 

If $y<x$ does not fail, follow the geodesics $\{\geo{}{y}{\xi\pm}: \xi\in[\zeta, \eta]\}$ until they hit the level  $\level_{x\cdot\et}$ at some point $y'$.    The assumption that neither $x- e_1$ nor $x- e_2$ lies in $\cK$ implies that $y'\ne x$.  Thus $y'$ is a point of  $\cK$ that fails both   $y'\ge x$ and  $y'<x$.  Replace $y$ with $y'$ and apply the previous argument. 

We have shown that the existence of $x\in\cK$ such that $\{x- e_1, x- e_2\}\cap\cK=\varnothing$ implies that $\cK\subset\{y: y\ge x\}$.     That such $x$ must be unique follows since $x$ lies outside $\{y: y\ge x'\}$ for any $x'\ne x$ that satisfies $x'\ge x$. 
 
Assuming that $\inf\{ t\in\Z: \cK\cap\level_t\ne\varnothing\}>-\infty$,  pick $x \in\cK$ to minimize the level $x\cdot\et$.  
\end{proof}

\begin{proof}[Proof of Theorem \ref{thm:jumpG}] 
{\it Part \eqref{thm:jumpG.a}.}  If $[\zeta, \eta]\cap\,\aUset=\varnothing$ then the interval  $[\zeta, \eta]$  is  strictly on one    side of $\cid(T_x\w)$ at every $x$.   Hence the graphs  $\{ \G{\xi\sig}: \xi\in[\zeta, \eta], \sigg\in\{-,+\}\}$ are all identical.  This common graph is a tree by Theorem \ref{thm:+-coal}. 

Conversely, if   $\xi\in[\zeta, \eta]\cap\,\aUset$,  then there exist $x,y$ such that $\xi\in\supp{\mu_{x,y}}$ and by Theorem \ref{thm:nonint} there are disjoint geodesics in $\G{\cap[\zeta, \eta]}$. 
\medskip

{\it Part \eqref{thm:jumpG.b}.}
It follows from what was already said that $\{\cK(z):  z\in\cifset^{[\zeta, \eta]}\}$ are disjoint subtrees  of $\G{\cap[\zeta, \eta]}$ and their vertex sets cover $\Z^2$.   Suppose  $(x,x+ e_i)$ is an edge in $\G{\cap[\zeta, \eta]}$.  Then all geodesics $\{\geo{}{x}{\xi\sig}: \xi\in[\zeta, \eta], \sigg\in\{-,+\}\}$ go through this edge.  Thus this edge must be an edge of the tree $\cK(z)$ that contains both $x$ and $x+ e_i$.    Hence each edge of $\G{\cap[\zeta, \eta]}$ is an edge of one of the trees $\cK(z)$, and no such edge can connect two  trees $\cK(z)$ and $\cK(z')$ for distinct $z$ and $z'$. 
\end{proof} 


\section{Instability points in the exponential model: proofs}\label{sec:exp-pf}

We turn  to the proof of the results in Section \ref{sec:exp}, beginning with a discussion of Palm kernels, which are needed in order to prove Theorems \ref{th:B-palm1} and \ref{thm:ladder-xi}.
\subsection{Palm kernels} \label{sub:palm}

Let $\cM_{\Z\times\ri\Uset}$ denote the space of locally bounded positive  Borel measures  on the locally compact space  $\Z\times\ri\Uset$.   Consider  $\Z\times\ri\Uset$  as the disjoint union of copies of $\ri\Uset$, one copy for each horizontal edge $(ke_1, (k+1)e_1)$ on the $x$-axis. Recall that $\B{\xi\sig}_k=\B{\xi\sig}(ke_1,(k+1)e_1)$.
We define  two random measures $\Bppa$ and $\Bpp$  on  $\Z\times\ri\Uset$ in terms of the Busemann functions $\xi\mapsto\B{\xi\pm}_k$ attached to   these  edges.

On each subset $\{k\}\times\ri\Uset$ of $\Z\times\ri\Uset$ we (slightly abuse notation and) define the measure $\Bppa_k$ by 
 \[  \Bppa_k\bigl(\{k\}\times ]\zeta, \eta]\,\bigr) = \Bppa_k\bigl(\,]\zeta,\eta]\, \bigr)=\B{\zeta+}_k-\B{\eta+}_k\] 
 for $\zeta\prec\eta$ in $\ri\Uset$. 
In terms of definition \eqref{Bmeas},   $\Bppa_k=\mu_{(k+1)e_1, ke_1}$ is  a positive measure due to monotonicity \eqref{mono}. 
 On $\Z\times\ri\Uset$, define the measure   $\Bppa=\sum_k\Bppa_k$.  In other words, for Borel sets $A_k\subset\ri\Uset$,  
$\Bppa\bigl( \,\bigcup_k \{k\}\times A_k\bigr)=\sum_k \Bppa_k(A_k)$.

Let $\Bpp_k$ denote the simple point process on  $\{k\}\times\ri\Uset$ that records   the locations of the jumps of the Busemann function $\xi\mapsto\B{\xi\pm}_k$:  for Borel $A\subset\ri\Uset$,
 \[  \Bpp_k(\{k\}\times A) = \Bpp_k(A) = \sum_{\xi\in A} \ind\{\B{\xi-}_k>  \B{\xi+}_k\}.  \]

We describe the probability distributions of the component measures $\Bppa_k$ and $\Bpp_k$, given in Theorem 3.4 of \cite{Fan-Sep-20}. 
Marginally, for each $k$, $\Bpp_k$ is a Poisson point process on $\ri\Uset$  with intensity measure
\be\label{Bmm}
  \Bmm\bigl(\,]\zeta, \eta]\,\bigr) = \Bmm_k\bigl(\,]\zeta, \eta]\,\bigr)= \E\bigl[ \Bpp_k(\,]\zeta, \eta]\,)\bigr]= \int_{\alpha(\zeta)}^{\alpha(\eta)}  \frac{ds}s = \log \frac{\alpha(\eta)}{\alpha(\zeta)}. 
  \ee
  In particular,  almost every realization of $\Bpp_k$ satisfies  $\Bpp_k[\zeta, \eta]<\infty$
for all $\zeta\prec\eta$ in $\ri\Uset$. 

 Create a marked Poisson process by attaching an independent Exp$(\alpha(\xi))$-distributed weight $Y_{\xi}$ to each point $\xi$ in the support of $\Bpp_k$.  Then   the distribution of $\Bppa_k$ is that of the purely atomic measure defined by 
\be\label{Bppa4}   \Bppa_k\bigl(\,]\zeta, \eta]\,\bigr) = \sum_{\xi\in\ri\Uset:\,\Bpp_k(\xi)=1}  Y_{\xi} \,\ind_{]\zeta, \eta]}(\xi) \quad \text{for }  \zeta\prec\eta\text{ in } \ri\Uset. \ee
The random variable $\Bppa_k(\,]\zeta, \eta]\,)$ has distribution $\text{Ber}(1-\frac{\alpha(\zeta)}{\alpha(\eta)})\otimes\text{Exp}(\alpha(\zeta))$ (product of a Bernoulli and an independent exponential) and 
expectation 
\be\label{Bmma}
  \E\bigl[ \Bppa_k\bigl(\,]\zeta, \eta]\,\bigr)\bigr]=  \frac1{\alpha(\zeta)} - \frac1{\alpha(\eta)}\,. 
  \ee

Note the following technical point.  The jumps of $\B{\xi\pm}_k$ concentrate at $e_2$ and $\B{e_2-}_k=\infty$.  To define $\Bppa$ and $\Bpp$ as locally finite measures, the standard Euclidean topology of  $\ri\Uset$ has to be metrized so that $]e_2, \eta]$ is an unbounded set for any $\eta\succ e_2$.   This point makes no difference to our calculations and we already encountered this same issue around  definition \eqref{Bmeas} of the Busemann measures.  
With this convention we can regard   $\Bpp=\sum_k\Bpp_k$ as a simple point process on $\Z\times\ri\Uset$ with mean measure $\widetilde{\Bmm}= (\text{counting  measure on $\bbZ$})\otimes \lambda$.  

For $(k,\xi)\in \Z \times \ri \Uset$, let $Q_{(k,\xi)}$ be the Palm kernel of $\Bppa$ with respect to $\Bpp$.  That is, $Q_{(k,\xi)}$ is the stochastic kernel from $ \Z \times \ri \Uset$ into $\cM_{\Z \times \ri \Uset}$  that gives the  distribution of $\Bppa$,  conditional on $\Bpp$ having a point at $(k,\xi)$, understood in the Palm sense.   Rigorously, the kernel is defined by disintegrating the Campbell measure of the pair $(\Bpp,\Bppa)$ with respect to the mean measure $\widetilde{\Bmm}$ of $\Bpp$ (this is developed in Section 6.1 in \cite{Kal-17}):
for any nonnegative  Borel  function $f:(\Z \times \ri \Uset) \times\cM_{\Z \times \ri \Uset}\to\R_+$,
\begin{align}\label{79-50-50} 
\E\Bigl[ \int_{\Z \times \ri \Uset} \!\!\!f(k,\xi,\Bppa)\, \Bpp(dk \otimes d\xi)\Bigr] &= \int_{\Z\times\ri\Uset}  \int_{\cM_{\Z\times\ri\Uset}} f(k,\xi,\Bppavar) \,Q_{(k,\xi)}(d\Bppavar)\, \widetilde{\Bmm}(dk \otimes d\xi).
\end{align}

Now we consider the indices $\tau^\xi(i)=\tau^{\xi,\xi}(i)$ of jumps at $\xi$,  defined in \eqref{78-45}.  
In terms of the random measures introduced above, for $(k,\xi)\in \Z \times \ri \Uset$, 
\[   
\Bppa\{(k,\xi)\}>0 \; \iff \; \Bpp\{(k,\xi)\}=1 \; \iff \; 
\B{\xi-}_k> \B{\xi+}_k   
\;\iff \;k\in\{\tau^\xi(i):i\in\Z\}. \]

 We condition on the event $\{\Bpp(0,\xi)=1\}$, in other words, consider the distribution of 
  $\{\tau^\xi(i)\}$ under $Q_{(0, \xi)}$. For this  to be well-defined, we  define these functions also on the space $\cM_{\Z \times \ri \Uset}$ in the obvious way:  for  $\nu\in\cM_{\Z \times \ri \Uset}$, the $\Z\cup\{\pm\infty\}$-valued functions $\tau^\xi(i)=\tau^\xi(i,\nu)$ are defined by the order requirement  
\[ \dotsm<\tau^\xi(-1,\nu)<0\le \tau^\xi(0,\nu)<\tau^\xi(1,\nu)<\dotsm\]  and the condition 
\[ 
\text{for $k\in\Z$,} \quad  \nu\{(k,\xi)\}>0\quad\text{if and only if}\quad k\in\{\tau^\xi(i,\nu):i\in\Z\}.   
\] 
Since $\Bppa$ is $\P$-almost surely a purely atomic measure, it follows from general theory that  $Q_{(0, \xi)}$ is also supported on such measures. Furthermore, the conditioning itself forces   $Q_{(0, \xi)}\{\nu: \tau^\xi(0,\nu)=0\}=1$.   Thus   the random integer points $\tau^\xi(i,\nu)$ are not all trivially $\pm\infty$ under  $Q_{(0, \xi)}$. Connecting back to the notation of Section \ref{sec:exp}, for each $k \in \bbZ$, $\xi \in \ri \Uset$, each finite $A \subset \bbZ$  and $n_i \in \bbZ_+, r_i \in \bbR_+$ with $i \in A$, the Palm kernel introduced in that section is defined  by
\begin{align}
&\bbP\bigl\{ \tau^\xi(i+1)-\tau^\xi(i) =n_i, \depa^{\xi-}_{\tau^{\xi}(i)}-\depa^{\xi+}_{\tau^{\xi}(i)}>r_i  \; \forall   i \in A \,||\, \B{\xi-}_k > \B{\xi+}_k\bigr\} \label{eq:palmdef} \\
&\qquad=Q_{(k, \xi)}\bigl\{ \nu: \tau^\xi(i+1,\nu)-\tau^\xi(i,\nu)=n_i, \,\nu\{(\tau^\xi(i,\nu),\xi)\}> r_i   \; \forall   i \in A\bigr\}. \notag
\end{align}

\subsection{Statistics of instability points}
We   turn to the proofs of the theorems of Section \ref{sec:exp}. These proofs make use of results from  Appendices \ref{a:bus} and \ref{a:RW}.

\begin{proof}[Proof of Theorem \ref{th:palm1}]
 By Corollary \ref{cor:B-q2},   the process $\{ \depa^\zeta_k-\depa^\eta_k\}_{k\in\Z}$ has the same distribution as $\{W^+_k\}_{k\in\Z}$ defined in \eqref{Wk}.   An application of the appropriate  mapping to these sequences produces the sequence 
 $\bigl\{\depa^\zeta_0-\depa^\eta_0, \, \tau^{\zeta, \eta}(i+1)-\tau^{\zeta, \eta}(i), \depa^\zeta_{\tau^{\zeta, \eta}(i)}-\depa^\eta_{\tau^{\zeta, \eta}(i)}: i\in\Z\bigr\}$ 
 that appears in  Theorem \ref{th:palm1} 
 and the sequence 
 $\{W^+_0, \, \sigma_{i+1}-\sigma_i, \, W^+_{\sigma_i}: i\in\Z\}$ that appears in  Theorem \ref{th:78-70}.  Hence these sequences also have identical distribution.   ($W^+_{\sigma_i}=W_{\sigma_i}$ by \eqref{78-65}.) 
The distributions remain equal when these sequences are conditioned on the positive probability events $ \depa^\zeta_0-\depa^\eta_0>0$ and $W^+_0>0$.  
 \end{proof}

It will be convenient to have notation for the conditional joint distribution that appears in \eqref{78-90} in Theorem \ref{th:palm1}. 
For $0<\alpha\le\beta\le 1$ define   probability distributions  $q^{\alpha, \beta}$   on the product space   $\Z^\Z\times[0,\infty)^\Z$ as follows.  Denote the  generic variables on this product space  by $(\{\tau_i\}_{i\in\Z},  \{\Delta_k\}_{k\in\Z})$ with $\tau_i\in\Z$ and $0\le\Delta_k<\infty$.     Given an integer $L>0$, integers 
$n_{-L}<\dotsm < n_{-2}< n_{-1}<n_0=0<n_1<n_2< \dotsm<n_L$,  
and positive reals $r_{-L}, \dotsc, r_L$,  abbreviate $b_i=n_{i+1}-n_i$.  The measure  $q^{\alpha, \beta}$ is defined by 
\be\label{qab4}\begin{aligned}  
&q^{\alpha, \beta}\bigl\{\tau_{i}=n_i \text{ and }  \Delta_{n_i}> r_i \text{ for } i\in\lzb-L,L\rzb, \,  \Delta_k=0 
\text{ for } k\in\lzb n_{-L},n_L\rzb\setminus\{ n_j\}_{j\in\lzb-L,L\rzb}  \bigr\}\\
&=\biggl( \; \prod_{i=-L}^{L-1} C_{b_i-1}\,  \frac{\alpha^{b_i-1}\beta^{b_i}}{(\alpha+\beta)^{2b_i-1}}\biggr) \cdot 
 \biggl( \; \prod_{i=-L}^{L} e^{-\alpha r_i}\biggr).  
\end{aligned} \ee
To paraphrase the definition,  the following holds under  $q^{\alpha, \beta}$:  $\tau_0=0$,   $\Delta_k=0$ for $k\notin\{\tau_i\}_{i\in\Z}$, and  the variables   $\{\tau_{i+1}-\tau_{i}, \Delta_{\tau_i}\}_{i\in\Z}$   are mutually independent with marginal distribution 
\be\label{qab6} 
q^{\alpha, \beta}\{\tau_{i+1}-\tau_{i}=n, \Delta_{\tau_i}>r\}=C_{n-1}\,  \frac{\alpha^{n-1}\beta^n}{(\alpha+\beta)^{2n-1}} e^{-\alpha r}\quad \text{ for } i\in\Z, \,n\ge 1,\, r\ge0. 
\ee
 Abbreviate   $q^\alpha=q^{\alpha,\alpha}$   which has marginal  $q^\alpha\{\tau_{i+1}-\tau_{i}=n, \Delta_{\tau_i}>r\}=C_{n-1}(\tfrac12)^{2n-1} e^{-\alpha r}$.   As 
$\beta\to\alpha$,  $q^{\alpha, \beta}$ converges weakly to $q^\alpha$. 


Theorem \ref{th:palm1} can now be restated by saying that,  conditional on
$\depa^\zeta_0>\depa^\eta_0$,    the variables   \[\bigl(\{\tau^{\zeta, \eta}(i)\}_{i\in\Z}, \,  \{\depa^\zeta_k-\depa^\eta_{k}\}_{k\in\Z}\bigr)\]
  have joint   distribution $q^{\alpha(\zeta), \alpha(\eta)}$. 
Consequently, for a measurable set $A\subset\Z^\Z\times[0,\infty)^\Z$, 
\be\label{78-92} \begin{aligned}
&\P\bigl[  \depa^\zeta_0>\depa^\eta_0, \, \bigl(\{\tau^{\zeta, \eta}(i)\}_{i\in\Z}, \,  \{\depa^\zeta_k-\depa^\eta_{k}\}_{k\in\Z}\bigr) \in A  \bigr]\\[4pt] 
&=\P\bigl(  \depa^\zeta_0>\depa^\eta_0\bigr) \, \P\bigl[ \bigl(\{\tau^{\zeta, \eta}(i)\}_{i\in\Z}, \,  \{\depa^\zeta_k-\depa^\eta_{k}\}_{k\in\Z}\bigr) \in A  \,\big\vert\, \depa^\zeta_0>\depa^\eta_0\,\bigr] 
 \\
&=  \frac{\alpha(\eta)-\alpha(\zeta)}{\alpha(\eta)} \cdot   q^{\alpha(\zeta), \alpha(\eta)}(A). 
%
\end{aligned} \ee
The first probability on the last line came from \eqref{B4} and the second from Theorem \ref{th:palm1}.

\begin{proof}[Proof of Theorem \ref{th:B-palm1}]
 
Define $\Z\cup\{\pm\infty\}$-valued ordered indices 
   $\dotsm<\tau^{\zeta, \eta}_{-1}<0\le \tau^{\zeta, \eta}_0<\tau^{\zeta, \eta}_1<\dotsm$   as measurable  functions of a locally finite measure  $\nu\in\cM_{\Z \times \ri \Uset}$ by  the rule  
\be\label{78-105}   \nu(\{k\}\times[\zeta,\eta])>0
\ \iff  \ k\in\{\tau^{\zeta, \eta}_i:i\in\Z\}. \ee
If  $\nu(\{k\}\times[\zeta,\eta])>0$ does not hold for infinitely many $k>0$ then  $\tau^{\zeta, \eta}_i=\infty$ for large enough $i$, and analogously  for 
$k<0$.    Definition \eqref{78-105}  applied  to the random measure   $\Bppa=\sum_k\Bppa_k$  reproduces  \eqref{78-45}.  

Fix integers   $K, N\in\N$ and  $ \ell_{-N}\le \dotsm \le  \ell_{-1}\le \ell_0=0 \le \ell_1\le \dotsm\le \ell_N$   and strictly positive reals $r_{-K}, \dotsc, r_K$. 
  Define the event 
\be\label{Bze8} \begin{aligned} 
  H^{\zeta, \eta}=H(\zeta, \eta) = 
   &\bigcap_{1\le i\le N}\bigl\{ \nu:  \tau^{\zeta, \eta}_{-i}\le \ell_{-i} \text{ and }  \tau^{\zeta, \eta}_{i}\ge\ell_i \bigr\}
   \\
   &\qquad 
   \cap
   \bigcap_{-K\le k\le K}\bigl\{ \nu:  \nu(\{k\}\times]\zeta, \eta]) < r_k\bigr\} 
\end{aligned}  \ee 
 on the space $\cM_{\Z \times \ri \Uset}$. 
  Note the monotonicity 
 \be\label{Bze-mon}
 H^{\zeta, \eta}\subset H^{\zeta', \eta'}\quad\text{ for } \ [\zeta', \eta']\subset [\zeta, \eta].
 \ee 
 Abbreviate $H^\xi=H^{\xi, \xi}$.     Recall the measures $q^{\alpha, \beta}$ defined in \eqref{qab4}. 
The analogous event under the measures $q^{\alpha, \beta}$ on the space $\Z^\Z\times[0,\infty)^\Z$  is denoted  by 
\be\label{Bze18} \begin{aligned} 
 H_q =\bigl\{ (\{\tau_i\}_{i\in\Z}, \{\Delta_k\}_{k\in\Z})\in \Z^\Z\times[0,\infty)^\Z :   \ &\tau_{-i}\le \ell_{-i} \text{ and } \tau_{i}\ge\ell_i \text{ for } i\in\lzb 1,N\rzb, \,  \\
    &\quad   \Delta_k< r_k\text{ for } k\in\lzb-K,K\rzb\bigr\}. 
 \end{aligned}\ee
 
 Fix $\zeta\prec\eta$ in $\ri\Uset$. 
  We prove the theorem by showing that  
 \be\label{Bze4} \begin{aligned} 
 Q_{(0,\xi)}(H^\xi)  =q^{\alpha(\xi)}(H_q) 
 \qquad\text{for Lebesgue-almost every $\xi\in\,]\zeta, \eta]$.} 
\end{aligned}  \ee
This equality comes from separate arguments for upper and lower bounds.  

\bigskip 

{\it Upper bound proof.}  
Define a sequence of nested partitions $\zeta=\zeta^n_0\prec\zeta^n_1\prec\dotsm\prec \zeta^n_n=\eta$.  For each $n$ and $\xi\in\,]\zeta, \eta]$,  let $]\zeta^n(\xi), \eta^n(\xi)]$ denote the unique interval $]\zeta^n_i, \zeta^n_{i+1}]$ that contains $\xi$.  Assume that, as $n\nearrow\infty$,   the mesh size $\max_i\abs{\zeta^n_{i+1}-\zeta^n_i} \to 0$. Consequently, for each $\xi\in\,]\zeta, \eta]$,  the intervals $]\zeta^n(\xi), \eta^n(\xi)]$ decrease to the singleton $\{\xi\}$.

The key step of this upper bound proof is that  for all $m$ and $i$ and Lebesgue-a.e.\ $\xi\in\,]\zeta, \eta]$, 
\be\label{79-43}\begin{aligned}  &Q_{(0, \xi)}(H^{\zeta^m_i, \zeta^m_{i+1}}) 
=  \lim_{n\to\infty}     \P\bigl\{ \Bppa \in H^{\zeta^m_i, \zeta^m_{i+1}} \,\big\vert\, \Bpp_{0}(\,]\zeta^n(\xi), \eta^n(\xi)]\,)\ge 1 \bigr\} .  
\end{aligned} \ee 
This limit is a special case of Theorem 6.32(iii) in Kallenberg \cite{Kal-17}, for the simple point process $\Bpp$ and  the sets $B_n=\{0\}\times(\zeta^n(\xi), \eta^n(\xi)]\searrow\{(0,\xi)\}$. 
  The proof given for  Theorem 12.8  in \cite{Kal-83} can also be used to establish this limit.    Theorem 12.8  of  \cite{Kal-83}  by itself is not quite adequate because we use the Palm kernel for the measure $\Bppa$ which is not the same as $\Bpp$.

 If we take $\xi\in \,]\zeta^m_i, \zeta^m_{i+1}]$, then for   $n\ge m$,   
 $]\zeta^n(\xi), \eta^n(\xi)]\subset \,]\zeta^m(\xi), \eta^m(\xi)]=\,]\zeta^m_i, \zeta^m_{i+1}]$.  Considering all $\xi$ in the union $]\zeta, \eta]=\bigcup_i \,]\zeta^m_i, \zeta^m_{i+1}]$, we have that  for any fixed $m$ and Lebesgue-a.e.\ $\xi\in\,]\zeta, \eta]$, 
 \begin{align*}
 Q_{(0, \xi)}(H^{\zeta^m(\xi), \eta^m(\xi)}) &=
  \lim_{n\to\infty} \P\bigl\{  \Bppa \in H^{\zeta^m(\xi), \eta^m(\xi)} \,\big\vert\, \Bpp_{0}(\,]\zeta^n(\xi), \eta^n(\xi)]\,)\ge 1 \bigr\}  \\[3pt] 
&  \le  
   \varliminf_{n\to\infty} \P\bigl\{  \Bppa \in H^{\zeta^n(\xi), \eta^n(\xi)} \,\big\vert\, \Bpp_{0}(\,]\zeta^n(\xi), \eta^n(\xi)]\,)\ge 1 \bigr\} .
   \end{align*} 
The inequality is due to \eqref{Bze-mon}.    
   
 Interpreting    \eqref{78-92}  in terms of the random measures $\Bppa$ and $\Bpp$ and referring to \eqref{Bze8} and  \eqref{Bze18} gives the identity 
\[ \P\bigl\{  \Bppa \in H^{\zeta^n(\xi), \eta^n(\xi)} \,\big\vert\, \Bpp_{0}(\,]\zeta^n(\xi), \eta^n(\xi)]\,)\ge 1 \bigr\}
=  q^{\alpha(\zeta^n(\xi)), \alpha(\eta^n(\xi))}(H_q).  \]
 As $(\zeta^n(\xi), \eta^n(\xi)]\searrow\{\xi\}$, the   parameters converge:  $\alpha(\zeta^n(\xi)), \alpha(\eta^n(\xi))\to\alpha(\xi)$. 
Consequently the distribution $q^{\alpha(\zeta^n(\xi)), \,\alpha(\eta^n(\xi))}$
 converges to   $q^{\alpha(\xi)}$. 
Hence
 \begin{align*}
  \lim_{n\to\infty} \P\bigl\{  \Bppa \in H^{\zeta^n(\xi), \eta^n(\xi)} \,\big\vert\, \Bpp_{0}(\,]\zeta^n(\xi), \eta^n(\xi)]\,)\ge 1 \bigr\} 
  =q^{\alpha(\xi)}(H_q). 
\end{align*}  
 In summary, we have for all $m$   and Lebesgue-a.e.\ $\xi\in \,]\zeta, \eta]$, 
\be\label{79-46} \nn \begin{aligned} 
&Q_{(0, \xi)}(H^{\zeta^m(\xi), \eta^m(\xi)})  \le q^{\alpha(\xi)}(H_q). 
\end{aligned} \ee
Let $m\nearrow\infty$ so that $H^{\zeta^m(\xi), \eta^m(\xi)}\nearrow H^\xi$, to obtain the upper  bound  
\be\label{79-47} \begin{aligned} 
&Q_{(0, \xi)}(H^\xi)   \le q^{\alpha(\xi)}(H_q) 
\end{aligned} \ee
for Lebesgue-a.e.\ $\xi\in \,]\zeta, \eta]$.  
 
\bigskip 

{\it Lower bound proof.} 
Let  $\zeta=\zeta_0\prec\zeta_1\prec\dotsm\prec \zeta_\ell=\eta$ be a  partition of the interval $[\zeta, \eta]$ and set  $\alpha_j=\alpha(\zeta_j)$. 

 In order to get an estimate below,  let $\mvec=(m_i)_{1\le\abs i\le N}$ be a $2N$-vector of 
  integers such that  $m_i<\ell_i$ for $-N\le i\le -1$ and $m_i>\ell_i$ for $1\le i\le N$. Define the  subset  $H^{\mvec}_q$ of $H_q$ from \eqref{Bze18} by truncating the coordinates $\tau_i$: 
\be\label{Bze20} \begin{aligned} 
 H^{\mvec}_q =\bigl\{ (\{\tau_i\}_{i\in\Z}, \{\Delta_k\}_{k\in\Z})\in \Z^\Z\times[0,\infty)^\Z :   \ &m_{-i}\le \tau_{-i}\le \ell_{-i} \text{ and }  \ell_i \le \tau_{i}\le m_i   \\
    &  \text{ for } i\in\lzb 1,N\rzb, \, \Delta_k< r_k\text{ for } k\in\lzb-K,K\rzb\bigr\}. 
 \end{aligned}\ee

On the last line in the following computation, $c_1$ is a constant that depends on the parameters $\alpha(\zeta)$ and $\alpha(\eta)$ and on the quantities in \eqref{Bze20}:
\begin{align*}
&\int_{]\zeta,\eta]}  Q_{(0,\xi)}(H^\xi)\,\Bmm_0(d\xi)  = \sum_{j=0}^{\ell-1} \int_{]\zeta_j,\zeta_{j+1}]}  Q_{(0,\xi)}(H^\xi)\,\Bmm_0(d\xi) \\[4pt] 
&\ge \sum_{j=0}^{\ell-1} \int_{]\zeta_j,\zeta_{j+1}]}  Q_{(0,\xi)}(H^{\zeta_j,\zeta_{j+1}})\,\Bmm_0(d\xi)
=  \sum_{j=0}^{\ell-1} \E\bigl[  \Bpp_0(\,]\zeta_j,\zeta_{j+1}]\,)  \cdot \ind_{H^{\zeta_j,\zeta_{j+1}}}(\Bppa)  \bigr]  
 \\[4pt] 
 &\ge \sum_{j=0}^{\ell-1} 
 \P\bigl\{   \Bpp_0(\,]\zeta_j,\zeta_{j+1}]\,) \ge 1,\,   \Bppa\in H^{\zeta_j,\zeta_{j+1}}  \bigr\} \\
 &= \sum_{j=0}^{\ell-1}  \frac{\alpha_{j+1}-\alpha_j}{\alpha_{j+1}} \,q^{\alpha_j, \alpha_{j+1}}(H_q) 
\; \ge \; \sum_{j=0}^{\ell-1}  \frac{\alpha_{j+1}-\alpha_j}{\alpha_{j+1}} \,q^{\alpha_j, \alpha_{j+1}}(H^{\mvec}_q) 
 \\ &
 \ge \sum_{j=0}^{\ell-1}  \frac{\alpha_{j+1}-\alpha_j}{\alpha_{j+1}} \, q^{\alpha_{j+1}} (H^{\mvec}_q) \cdot \bigl(1-c_1(\alpha_{j+1}-\alpha_j)\bigr) .  
\end{align*} 
The steps above come as follows.  The second equality uses the characterization \eqref{79-50-50}  of  the kernel $Q_{(0,\xi)}$. 
The third equality is from \eqref{78-92}.    The second last inequality  is from $H^{\mvec}_q\subset H_q$.    The last inequality is from Lemma \ref{lm:palm10} below, which is valid once the mesh size $\max(\alpha_{j+1}-\alpha_j)$ is small enough relative to the numbers $\{m_i, \ell_i\}$.  

The function $\alpha\mapsto q^\alpha (H^{\mvec}_q)$ is continuous in the Riemann sum approximation on the last line of the calculation above. 
Let $\max(\alpha_{j+1}-\alpha_j)\to 0$ to obtain the inequality
 \begin{align*}
&\int_{]\zeta,\eta]}  Q_{(0,\xi)}(H^\xi) \,\Bmm_0(d\xi)
\ge   \int_{\alpha(\zeta)}^{\alpha(\eta)}  q^\alpha (H^{\mvec}_q) \,\frac{d\alpha}\alpha 
= \int_{]\zeta,\eta]}   q^{\alpha(\xi)}  (H^{\mvec}_q)\,\Bmm_0(d\xi). 
\end{align*} 
Let $m_i\searrow-\infty$ for $-N\le i\le -1$ and $m_i\nearrow\infty$ for $1\le i\le N$.  
The above turns into 
\be\label{79-104} \begin{aligned}
\int_{]\zeta,\eta]}  Q_{(0,\xi)}(H^\xi)\,\Bmm_0(d\xi) 
\ge 
 \int_{]\zeta,\eta]}   q^{\alpha(\xi)}  (H_q)\,\Bmm_0(d\xi). 
\end{aligned}  \ee

\medskip

The upper bound \eqref{79-47} and the lower bound \eqref{79-104} together imply the conclusion \eqref{Bze4}. 
 \end{proof} 
 
 The proof of Theorem \ref{th:B-palm1} is complete once we verify the auxiliary lemma used in the calculation above. 
 
\begin{lemma} \label{lm:palm10} 
Let the event $H^{\mvec}_q$ be as defined in  \eqref{Bze20}. 
 Fix $0<\underline\alpha<\wb\alpha<1$.  Then there exist constants  $\e, c_1\in(0,\infty)$    such that 
\begin{align*}
q^{\alpha, \beta}(H^{\mvec}_q) \ge q^{\beta}(H^{\mvec}_q) \cdot \bigl(1-c_1(\beta-\alpha)\bigr) 
\end{align*}
for  all $\alpha,\beta\in[\underline\alpha,\wb\alpha]$ such that $\alpha\le\beta\le\alpha+\e$.    The constants  $\e, c_1\in(0,\infty)$   depend on $\underline\alpha$,  $\wb\alpha$, and  the parameters $\ell_i, m_i$ and $r_k$  in \eqref{Bze20}. 
\end{lemma} 

\begin{proof} Let 
\begin{align*} 
\cA=  \bigl\{ &\pvec=(p_i)_{-N\le i\le N} \in\Z^{2N+1} :   p_0=0, \, p_i<p_j \text{ for } i<j, \\
&\qquad \qquad\qquad
m_{-i}\le p_{-i} \le \ell_{-i} \text{ and } \ell_i\le p_{i}\le m_i \; \forall i\in\lzb1,N\rzb\bigr\}
\end{align*}
be the relevant  finite set of    integer-valued $(2N+1)$-vectors for  the decomposition below.   For each  $\pvec\in\cA$ let  $\cK(\pvec)=\{p_i: i\in\lzb-N,N\rzb,  p_i\in \lzb-K,K\rzb\}$ be the set of coordinates of $\pvec$  in $\lzb-K,K\rzb$.  Abbreviate  $b_i=p_{i+1}-p_i$. 
Recall that, under $q^{\alpha, \beta}$, $\tau_0=0$, that $\Delta_k<r_k$ holds with probability one if $k\notin\{\tau_i\}$,  and the independence in \eqref{qab6}.    The factors $d_k>0$ below that satisfy  $ 1-e^{-\alpha r_k} \ge (1-e^{-\beta r_k}) (1- d_k(\beta-\alpha))$ can be chosen uniformly for $\alpha\le\beta$ in $[\underline\alpha,\wb\alpha]$, as functions of $\underline\alpha$, $\wb\alpha$, and $\{r_k\}$. Now compute: 
\begin{align*}
&q^{\alpha, \beta}(H^{\mvec}_q)
=q^{\alpha, \beta}\bigl\{ m_i\le \tau_{-i}\le \ell_{-i} \text{ and } \ell_i\le \tau_{i}\le m_i \; \forall i\in\lzb1,N\rzb, \,   \Delta_k< r_k\text{ for } k\in\lzb-K,K\rzb\bigr\} \\[4pt] 
&\qquad=\sum_{\pvec\,\in\,\cA}  q^{\alpha, \beta}\bigl\{   \tau_{i}=p_i \text{ for }  i\in\lzb-N,N\rzb, \,   \Delta_k< r_k\text{ for } k\in\lzb-K,K\rzb\bigr\}
\\[4pt] 
&\qquad=\sum_{\pvec\,\in\,\cA}  q^{\alpha, \beta}\bigl\{   \tau_{i+1}-\tau_{i}=b_i \ \forall i\in\lzb-N,N-1\rzb\bigr\}  \cdot \prod_{k\,\in\,\cK(\pvec)}  (1-e^{-\alpha r_k})  \\
&\qquad\ge\sum_{\pvec\,\in\,\cA} \; \biggl( \; \prod_{i=-N}^{N-1}  C_{b_i-1}\,  \frac{\alpha^{b_i-1}\beta^{b_i}}{(\alpha+\beta)^{2b_i-1}}  \biggr)  
 \cdot \prod_{k\,\in\,\cK(\pvec)}  (1-e^{-\beta r_k}) \bigl(1- d_k(\beta-\alpha)\bigr)  \\
&\qquad\ge \sum_{\pvec\,\in\,\cA} \;  \biggl( \; \prod_{i=-N}^{N-1}  C_{b_i-1}\,  \biggl(\frac12\biggr)^{2b_i-1} \biggr) \biggl(\; \prod_{k\,\in\,\cK(\pvec)}  (1-e^{-\beta r_k}) \biggr) \cdot \bigl(1-c_1(\beta-\alpha)\bigr)   \\
&\qquad=\sum_{\pvec\,\in\,\cA}  q^\beta\bigl\{   \tau_{i+1}-\tau_{i}=b_i \ \forall i\in\lzb-N,N-1\rzb, \,   \Delta_k< r_k\text{ for } k\in\lzb-K,K\rzb\bigr\} 
 \cdot \bigl(1-c_1(\beta-\alpha)\bigr) \\
&\qquad= q^\beta(H^{\mvec}_q) \cdot \bigl(1-c_1(\beta-\alpha)\bigr). 
\end{align*}
To get  the  inequality above, (i)  apply   Lemma \ref{a:lm:t500} to the first factor in parentheses with  $\e$ chosen    so that $0<\e<\underline\alpha/b_i$ for all $\pvec\in\cA$, and (ii)   set $c_1= \sum_{k=-K}^K  d_k$. 
\end{proof} 

%
%
%

In the proofs that follow, we denote the indicators of the locations of the positive atoms of a measure $\nu\in\cM_{\Z\times\ri\Uset}$ by 
$u_k(\nu,\xi) = u_k(\xi)=\ind[\nu\{(k,\xi)\}>0]$ for $(k,\xi)\in\Z\times\ri\Uset$.  Applied to the random measure $\Bppa$ this gives $u_k(\Bppa, \xi)=\Bpp_k(\xi)$.

\begin{lemma} \label{lem:palmrw}
For Lebesgue-almost every $\xi\in\ri\Uset$   and all $m\in\Z$, 
 \be\label{rw123}  Q_{(m,\xi)}\bigl[ \nu :  \{\umeas_{m+k}(\nu,\xi)\}_{k\in\Z} \in A\bigr]  
 = \bfP(A)
 \ee
 for all Borel sets $A\subset\{0,1\}^\Z$. \end{lemma}
\begin{proof}
For $m=0$, \eqref{rw123} comes from a comparison of \eqref{rw113} and \eqref{rw117}.   For general $m$ it then  follows from the shift-invariance of the weights $\w$. 
\end{proof}

\begin{proof}[Proof of Theorem \ref{thm:ladder-xi}]


 
Take $A\subset\{0,1\}^\Z$ as in the statement of Theorem \ref{thm:ladder-xi}.  Fix   $\zeta \prec \eta $ in  $\ri \Uset$ and let $N \in \bbN$.   
We restrict the integrals below to the compact set $[-N,N] \times [\zeta,\eta]$ with the indicator 
\[  g(k,\xi,\Bppavar) = \one_{[-N,N]\times [\zeta,\eta]}(k,\xi) \] 
  and then define on  $\Z \times \ri \Uset \times \sM_{\Z \times \ri \Uset} $
\begin{align*}
f(k,\xi,\Bppavar) &= g(k,\xi,\Bppavar) \cdot \one_{\{(u_\ell\{\xi\} : \ell \in \Z) \,\in\, A\}}(\xi,\Bppavar).
\end{align*}
  By the definition  \eqref{79-50-50} of the Palm kernel,  
\begin{align*}
&\E\left[\;\int_{\Z \times \ri \Uset} f(k,\xi,\Bppa) \,\Bpp(dk \otimes d\xi )\right] 
= \int_{[-N,N] \times [\zeta,\eta]} Q_{(k,\xi)}\{(u_\ell\{\xi\} : \ell \in \Z) \in A\} \,\widetilde{\Bmm}( dk \otimes d\xi) \\
&\qquad\qquad\qquad=\int_{[-N,N] \times [\zeta,\eta]} Q_{(k,\xi)}\{(u_{k+\ell}\{\xi\} : \ell \in \Z) \in A\} \,\widetilde{\Bmm}(dk \otimes d\xi) \\
&\qquad\qquad\qquad
= \int_{[-N,N] \times [\zeta,\eta]} \widetilde{\Bmm}(dk \otimes d\xi) = \E\left[\;\int_{\Z \times \ri \Uset} g(k,\xi,\Bppa) \Bpp(dk \otimes d\xi )\right].
\end{align*}
The second equality used shift-invariance of $A$ and  the third equality used \eqref{rw123} and $\bfP(A)=1$. 
The left-hand side and the right-hand side are both finite because the integrals are restricted to the compact set $[-N,N] \times [\zeta,\eta]$. Since $\Bpp$ is a positive random measure, it follows that
\begin{align*}
\bbP\left(\;\int_{\Z \times \ri \Uset} f(k,\xi,\Bppa) \,\Bpp(dk \otimes d\xi ) = \int_{\Z \times \ri \Uset} g(k,\xi,\Bppa) \,\Bpp(dk \otimes d\xi )\right)=1.
\end{align*}
As $\zeta,\eta,$ and $N$ were arbitrary, we conclude that $\P$-almost surely   $(\Bpp_\ell\{\xi\} : \ell \in \Z) \in A$ for all $(k,\xi)\in\Z\times\ri\Uset$ such that $\Bpp\{(k,\xi)\}=1$.  Lemma \ref{lm:downright} applied to the $x$-axis ($x_i=ie_1$) then shows that $\xi\in\aUset$ if and only if   $\Bpp\{(k,\xi)\}=1$ for some $k$. 
\end{proof}

\begin{lemma}\label{lm:UB-aux}
Assume \eqref{exp-assump}.
Then for any $\delta\in(0,1)$, $n\in\N$, and $\zeta\in\ri\Uset$ we have
	\[\P\Bigl\{\exists\xi\in[\zeta,e_1[\,: \Bpp( \lzb0,n\rzb\times\{\xi\}) 
	>2\delta n+1\Bigr\}\le 2(n+1)\Bigl(\frac{(1-\delta/2)^{2-\delta}}{(1-\delta)^{1-\delta}}\Bigr)^{n}\log\alpha(\zeta)^{-1}.\]
\end{lemma}

\begin{proof}
Let $\{\Delta_j\}_{j\in\N}$ be i.i.d.\ random variables with probability mass function 
$p(n)=C_{n-1}{2^{1-2n}}$  for $n\in\N$.  
For $k\in\lzb 0,n\rzb$ and $\xi\in\ri\Uset$ use a union bound, translation, and  \eqref{rw113} 
to write
	\begin{align*}
	&Q_{(k,\xi)}\Bigl\{\,\sum_{i=0}^n u_i(\xi)>2\delta n+1\Bigr\}
	\le Q_{(k,\xi)}\Bigl\{\,\sum_{i=k-n}^{k+n} u_i(\xi)>2\delta n+1\Bigr\}\\
	&\quad= Q_{(0,\xi)}\Bigl\{\,\sum_{i=-n}^n u_i(\xi)>2\delta n+1\Bigr\}\\
	&\quad\le Q_{(0,\xi)}\Bigl\{\,\sum_{i=1}^n u_i(\xi)>\delta n\Bigr\}
	+ Q_{(0,\xi)}\Bigl\{\,\sum_{i=-n}^{-1} u_i(\xi)>\delta n\Bigr\}
	\le 2P\Bigl\{\,\sum_{j=1}^{\ce{\delta n}}\Delta_j\le n\Bigr\}.
	\end{align*}
Using 
the generating function $f(s)=\sum_{n\ge 0} C_ns^n=\tfrac12(1-\sqrt{1-4s}\,)$   of Catalan numbers we obtain for $0<s<1$, 
	\begin{align*}
	P\Bigl\{\,\sum_{j=1}^{\ce{\delta n}}\Delta_j\le n\Bigr\}
	&\le s^{-n}E[s^{\Delta}]^{\delta n}
	= s^{-n} \Bigl(2\sum_{n=1}^\infty C_{n-1} \,(s/4)^n\Bigr)^{\delta n}\\ 
	&= s^{-n}\Bigl(\,\frac{s}{2} \sum_{k=0}^\infty C_k \,(s/4)^k\Bigr)^{\delta n} 
	=  s^{-n}\bigl(1-\sqrt{1-s}\bigr)^{\delta n}.
	\end{align*}
Take  $s=\frac{4(1-\delta)}{(2-\delta)^2}<1$  in the upper bound above to get
	\[Q_{(k,\xi)}\Bigl\{\,\sum_{i=0}^n u_i(\xi)>2\delta n+1\Bigr\}\le 2\Bigl(\frac{(1-\delta/2)^{2-\delta}}{(1-\delta)^{1-\delta}}\Bigr)^{n}.\]
Apply \eqref{79-50-50} to write
	\begin{align*}
	&\E\Bigl[\,\int_{\ri\Uset}\one\{\xi\in[\zeta,e_1[\,\}\cdot\one\bigl\{
	\Bpp( \lzb0,n\rzb\times\{\xi\})
	>2\delta n+1\bigr\}\,\Bpp_k(d\xi)\Bigr]\\
	&\qquad=\int_{\ri\Uset} \one\{\xi\in[\zeta,e_1[\,\}\,Q_{(k,\xi)}\Bigl\{\,\sum_{i=0}^n u_i(\xi) >2\delta n+1\Bigr\}\,\Bmm_k(d\xi)\\
	&\qquad\le 2\Bigl(\frac{(1-\delta/2)^{2-\delta}}{(1-\delta)^{1-\delta}}\Bigr)^{n}\int_{\ri\Uset}\one\{\xi\in[\zeta,e_1[\,\}\Bmm_k(d\xi)
	\overset{\eqref{Bmm}}= 2\Bigl(\frac{(1-\delta/2)^{2-\delta}}{(1-\delta)^{1-\delta}}\Bigr)^{n}\log\alpha(\zeta)^{-1}.
	\end{align*}
To complete the proof, add over $k\in\lzb0,n\rzb$ and observe that
	\begin{align*}
	&\int_{\ri\Uset}\one\{\xi\in[\zeta,e_1[\,\}\one\bigl\{
	\Bpp( \lzb0,n\rzb\times\{\xi\})
	>2\delta n+1\bigr\}\sum_{k=0}^n\Bpp_k(d\xi)\\
	&\qquad\qquad\ge\one\Bigl\{\exists\xi\in[\zeta,e_1[\,:\Bpp( \lzb0,n\rzb\times\{\xi\})
	>2\delta n+1\Bigr\}.\qedhere
	\end{align*}
\end{proof}

\begin{proof}[Proof of Theorem \ref{thm:density-UB}]
The result follows from Theorem \ref{thm:density-UB2} below and the observation that 
for any $\e>0$, $\delta_n=2\sqrt{n^{-1}\log n}$
satisfies the summability condition in that theorem. 
\end{proof}

\begin{theorem}\label{thm:density-UB2}
Assume \eqref{exp-assump} and fix $i\in\{1,2\}$. 
Consider a sequence $\delta_n\in(0,1)$ with $\sum n^2 e^{-n\delta_n^2}<\infty$. Then for any $\zeta\in\ri\Uset$
		\begin{align}\label{claim2-UB}
		\P\Bigl\{\exists n_0:\forall \xi\in[\zeta,e_2[,\forall n\ge n_0:\sum_{x\in\lzb0,n\rzb^2}\one\bigl\{\xi \in \supp\mu_{x,x+e_i}\bigr\}\le n^2\delta_n\Bigr\}=1.
		\end{align}
The same result holds when $\lzb0,n\rzb^2$ is replaced by any one of $\lzb-n,0\rzb^2$, $\lzb0,n\rzb\times\lzb-n,0\rzb$, or $\lzb-n,0\rzb\times\lzb0,n\rzb$.
\end{theorem}

\begin{proof}
Apply Lemma \ref{lm:UB-aux} and a union bound to get that for any $j\in\{1,2\}$, $\delta\in(0,1)$, $n\in\N$, and $\zeta\in\ri\Uset$,
	\[\P\Bigl\{\exists\xi\in[\zeta,e_1[\,:\sum_{x\in\lzb0,n\rzb^2}\rho^j_x(\xi) \ge(2\delta n+1)(n+1)\Bigr\}\le 2(n+1)^2\Bigl(\frac{(1-\delta/2)^{2-\delta}}{(1-\delta)^{1-\delta}}\Bigr)^{n}\log\alpha(\zeta)^{-1}.\]
A Taylor expansion gives 
	\[\log\Bigl(\frac{(1-\delta/2)^{2-\delta}}{(1-\delta)^{1-\delta}}\Bigr)
	=-\delta^2/4+\cO(\delta^3).\]
Thus, we see that for any $\zeta\in\ri\Uset$, for any sequence $\delta_n\in(0,1)$ such that $\sum n^2 e^{-n\delta_n^2}<\infty$, we have
	\[\P\Bigl\{\exists n_0\in\N:\forall \xi\in[\zeta,e_1[\,,\forall n\ge n_0:\sum_{x\in\lzb0,n\rzb^2}\rho^j_x(\xi)\le n^2\delta_n\Bigr\}=1.\]
The result for the other three sums comes similarly.
\end{proof}

\appendix


\section{The geometry of geodesics: previously known results}\label{app:busgeo}

This appendix states   the properties of Busemann functions, geodesics, and competition interfaces which were discussed informally in Section \ref{sec:bus-brief}.   Theorem \ref{thm:Bus}  introduces the {\it Busemann process} with its main properties. It combines results that follow from 
Theorems 4.4 and 4.7, Lemmas 4.5(c) and 4.6(c), and Remark 4.11 in \cite{Jan-Ras-20-aop} and Lemmas 4.7 and 5.1 in \cite{Geo-Ras-Sep-17-ptrf-2}.

\begin{theorem}\label{thm:Bus}{\rm\cite{Jan-Ras-20-aop,Geo-Ras-Sep-17-ptrf-2}}
Let $\P_0$ be a probability measure on $\R^{\Z^2}$ under which the coordinate projections are i.i.d., have positive variance, and have $p>2$ finite moments.
There exists a Polish probability space $(\Omega,\sF,\bbP)$ with 
\begin{enumerate}[label={\rm(\arabic*)}, ref={\rm\arabic*}]  \itemsep=3pt 
\item   a group $T=\{T_x\}_{x\in\bbZ^2}$  of $\sF$-measurable $\bbP$-preserving bijections $T_x:\Omega\to\Omega$,  
\item   a family $\{\w_x(\w) : x \in \bbZ^2\}$ of real-valued random variables $\w_x:\Omega\to\R$
such that  $\w_y(T_x \w) = \w_{x+y}(\w)$ for all $x,y \in \bbZ^2$, 
\item real-valued measurable functions $\B{\xi+}(x,y,\w)=\B{\xi+}_{x,y}(\w)$ and $\B{\xi-}(x,y,\w)=\B{\xi-}_{x,y}(\w)$ of $(x,y,\w,\xi)\in\Z^2\times\Z^2\times\Omega\times\ri\Uset$, 
\item  and  $T$-invariant events 
$\Omega_0^1 \subset \Omega$ and  $\Omega_\xi^1\subset\Omega_0^1$ for each  $\xi\in\ri\Uset$, with $\bbP(\Omega_0^1) = \P(\Omega_\xi^1)=1$, 
\end{enumerate} 
 such that properties \eqref{thm:cocyexist:a}--\eqref{thm:cocyexist:g} listed below hold. 

\begin{enumerate}[label={\rm(\alph*)}, ref={\rm\alph*}]  \itemsep=3pt 
\item\label{thm:cocyexist:a} $\{\w_x:x\in\Z^2\}$ has distribution $\P_0$ under $\P$.
\item\label{thm:cocyexist:b} For any $I \subset \bbZ^2,$ the variables
\begin{align*}
\bigl\{(\w_x, \B{\xi\sig}(x,y,\w)): x \in I,y\ge x,\sigg\in\{-,+\},\xi\in \ri\Uset\bigr\}	
\end{align*}
are independent of $\{\w_x : x \in I^{<}\}$ where $I^<=\{x\in\Z^2:x\not\ge z\ \forall z\in I\}$.
\item\label{thm:cocyexist:c} For each $\xi\in\ri\Uset$,  $x,y\in\Z^2$, and $\sigg\in\{-,+\}$,
$\B{\xi\sig}(x,y)$ are  integrable and \eqref{E[h(B)]} holds.
\item\label{thm:cocyexist:foo} For each $\w\in\Omega$, $x,y\in\Z^2$, and $\sigg\in\{-,+\}$, 
if $\zeta,\eta\in\ri\Uset$ are such that $\nabla\gpp(\zeta\sigg)=\nabla\gpp(\eta\sigg)$, then $\B{\zeta\sig}(x,y,\w)=\B{\eta\sig}(x,y,\w)$.
\item\label{thm:cocyexist:d} For each $\w\in\Omega_0^1$,  $x,y,z\in\Z^2$, 
$\xi\in \ri\Uset$, and $\sigg\in\{-,+\}$ properties \eqref{cov-prop}, \eqref{coc-prop}, and \eqref{rec-prop2} hold. 
\item\label{thm:cocyexist:e} For each $\w\in\Omega_0^1$, monotonicity \eqref{mono} holds.
\item\label{thm:cocyexist:f} For each $\w\in\Omega_0^1$,  one-sided limits \eqref{Busemann-limits} hold.
\item\label{thm:cocyexist:h} For each $\w \in \Omega_0^1$ and each $x \in \bbZ^2$,
	\be\label{B:inf-lim} 
	\B{\xi\sig}(x,x+e_i)\to\infty \quad\text{as  }\ \xi\to e_{3-i}, \  \text{ for } \ i\in\{1,2\}. 
	\ee 
\item\label{thm:ties} If $\P(\w_0\le r)$ is continuous in $r$, then for all $\xi\in\ri\Uset$, $\w\in\Omega_\xi^1$, $x\in\Z^2$, and $\sigg\in\{-,+\}$,
	\begin{align}\label{eq:ties}
	\B{\xi\sig}(x,x+e_1)\ne\B{\xi\sig}(x,x+e_2).
	\end{align}
\item\label{thm:cocyexist:erg} For all $\xi\in\ri\Uset$, $\w\in\Omega_\xi^1$, and $\sigg\in\{-,+\}$,
	\begin{align}\label{eq:erg-coc}
	\lim_{n\to\infty}\max_{x\in \,n\Uset\cap\Z^2_+}n^{-1}\abs{\B{\xi\sig}(0,x)-x\cdot\nabla\gpp(\xi\sigg)}=0.
	\end{align}
\item\label{thm:cocyexist:g} For all $\xi \in \Diff$, $\w\in\Omega_\xi^1$, and $x,y \in \bbZ^2$
	\begin{align}\label{B+=B-:temp}
	\B{\xi+}(x,y,\w)=\B{\xi-}(x,y,\w)=\B{\xi}(x,y,\w).
	\end{align} 
\item\label{thm:cocyexist:Busfn} If $\xi, \ximin, \ximax \in \sD$ then for all $\w \in \Omega_\xi^1$, the Busemann limit \eqref{eq:Busfn} holds.
\end{enumerate}	
\end{theorem}


\begin{remark}[Weak limit construction of the Busemann process] 
\label{rmk:weak-lim} 
 Both articles \cite{Jan-Ras-20-aop,Geo-Ras-Sep-17-ptrf-2} on which we rely for Theorem \ref{thm:Bus} construct the process $\B{}$ as a weak limit point of  Ces\`aro averages of probability distributions of pre-limit objects.  This gives existence of the process on a probability space $\Omega$ that is larger than the product space $\R^{\Z^2}$ of the i.i.d.\ weights $\{\w_x\}$.  Article  \cite{Geo-Ras-Sep-17-ptrf-2} takes the outcome of the weak limit from existing literature in the form of a queueing fixed point, while \cite{Jan-Ras-20-aop}  builds the weak limit from scratch by considering the distribution of increments of point-to-line passage times, following the approach introduced in \cite{Dam-Han-14}.  

 To appeal to queueing literature,  \cite{Geo-Ras-Sep-17-ptrf-2}  assumes that   $\P(\w_0\ge c)=1$ for some real $c$.   A payoff is that each process $\{\B{\xi\sig}(x,y):x,y\in\Z^2\}$ is ergodic under either shift  $T_{e_i}$ \cite[Theorem 5.2(i)]{Geo-Ras-Sep-17-ptrf-2}.  The construction in \cite{Jan-Ras-20-aop} does not need the lower bound assumption but  gives only  the  $T$-invariance stated above in Theorem \ref{thm:Bus}\eqref{thm:cocyexist:a}.    
 
Theorem \ref{thm1} below quotes results from  \cite{Geo-Ras-Sep-17-ptrf-2} that were proved with the help of ergodicity. Remark \ref{rmk:erg3} explains how the required properties can be obtained without ergodicity. 
 \end{remark}

\begin{remark}[Strong existence and ergodicity of the Busemann process] 
\label{rk:ergodicity}
The regularity condition \eqref{g-reg} is equivalent to the  existence of a countable dense set $\Ddense\subset\Diff$ such that   $\zetamin,\zetamax\in\Diff$  for each $\zeta\in\Ddense$.
When \eqref{g-reg}  holds, \cite[Theorem 3.1]{Geo-Ras-Sep-17-ptrf-1} shows that for $\zeta$ in $\Ddense$, 
$\B{\zeta}(x,y)=\B{\zeta\pm}(x,y)$ 
can be realized as an almost sure limit of $G_{x,v_n}-G_{y,v_n}$ when  $v_n/n\to\xi$. The remaining values  $\B{\xi\sig}(x,y)$ 
can be obtained as  left and right limits from  $\{\B{\zeta}(x,y)\}_{\zeta\in\Ddense}$ as $\zeta\to\xi$. This way the  entire process $\{\B{\xi\sig}(x,y):x,y\in\Z^2,\xi\in\ri\Uset, \sigg\in\{-,+\}\}$ becomes    a measurable function of the i.i.d.\ weights $\{\w_x:x\in\Z^2\}$.  We can take $\Omega=\R^{\Z^2}$ and the Busemann process is  ergodic under any shift $T_x$ for $x\ne 0$. 
\end{remark}

We record a simple observation here, 
 valid  under the continuous i.i.d.\ weights assumption \eqref{main-assump}:   there exists an event $\Omega_0^2$ with $\P(\Omega_0^2)=1$ such that $\forall \w\in\Omega_0^2$, 
\begin{align} 
\begin{split}
&\text{for every nonempty  finite subset $I\subset\Z^2$ and nonzero  integer coefficients $\{a_x\}_{x\in I}$, }\\ 
&\textstyle{\sum_{x\in I} a_x\w_x\ne 0.}
\end{split}\label{paths.2}  
\end{align} 
This condition implies the uniqueness of point-to-point geodesics mentioned under \eqref{G1}.

The following theorem summarizes previous knowledge of  the structure of semi-infinite geodesics under assumption \eqref{main-assump}.  These results were partly summarized in Section \ref{sec:geod-brief}. 

\begin{theorem}{\rm\cite[Theorems 2.1, 4.3, 4.5, and 4.6]{Geo-Ras-Sep-17-ptrf-2}}\label{thm1}
There exist $T$-invariant events $\Omega_0^3$ and $\Omega_\xi^3\subset\Omega_0^3$ for each $\xi\in\ri\Uset$, with $\P(\Omega_0^3)=1$, $\P(\Omega_\xi^3)=1$, and such that the following hold.
\begin{enumerate}[label={\rm(\alph*)}, ref={\rm\alph*}]  \itemsep=3pt 
	\item\label{thm1:exist}   For every $\w\in\Omega_0^3$,  for every $x\in\Z^2$, $\sigg\in\{-,+\}$, and  $\xi\in\ri\Uset$,
	$\geo{}{x}{\xi\sig}$ is $\Uset_{\xi\sig}$-directed and every 
	semi-infinite geodesic is $\Uset_\xi$-directed for some $\xi\in\Uset$.
	\item\label{thm1:coal} For every $\xi\in\ri\Uset$, for every $\w\in\Omega_\xi^3$, $x,y\in\Z^2$, and $\sigg\in\{-,+\}$, $\geo{}{x}{\xi\sig}$ and $\geo{}{y}{\xi\sig}$ coalesce,  i.e.\ there exists an integer $k\ge x\cdot\et\vee y\cdot\et$ such that $\geo{k,\infty}{x}{\xi\sig}=\geo{k,\infty}{y}{\xi\sig}$. 
	\item\label{thm1:biinf}  For every $\xi\in\ri\Uset$, $\w\in\Omega_\xi^3$,  $x\in\Z^2$, and $\sigg\in\{-,+\}$, there exist at most finitely many $z\in\Z^2$ such that $\geo{}{z}{\xi\sig}$ goes through $x$. 
	\item\label{thm1:direct}	If $\gpp$ is strictly concave,   then for any $\w\in\Omega_0^3$
		 every semi-infinite geodesic is $\xi$-directed for some $\xi\in\Uset$.
	\item\label{thm1:cont}
If $\xi\in\ri\Uset$ is such that  $\Uset_\xi=[\,\ximin,\ximax\,]$ satisfies $\ximin, \xi,\ximax\in\Diff$, then for any $\w\in\Omega_\xi^3$ and $x\in\Z^2$ we have $\geo{}{x}{\xi+}=\geo{}{x}{\xi-}$. This is the unique $\Uset_\xi$-directed  semi-infinite  geodesic out of $x$  and, by part \eqref{thm1:coal}, all these geodesics coalesce. By part \eqref{thm1:biinf}, there are no bi-infinite $\Uset_\xi$-directed geodesics.
\end{enumerate}
\end{theorem}

\begin{remark}[Ergodicity in the proof of Theorem \ref{thm1}] \label{rmk:erg3} 
As mentioned in Remark \ref{rmk:weak-lim}, 
  \cite{Geo-Ras-Sep-17-ptrf-2}   uses ergodicity of cocycles.  But  the results quoted above in Theorem \ref{thm1} can be obtained  with stationarity, which comes   from  \cite{Jan-Ras-20-aop} without the restrictive assumption $\w_x\ge c$. 
  
 The proof of  directedness (Theorem \ref{thm1}\eqref{thm1:exist} above) given  in \cite[Theorem 4.3]{Geo-Ras-Sep-17-ptrf-2}
uses the shape theorem of ergodic cocycles stated in \cite[Theorem A.1]{Geo-Ras-Sep-17-ptrf-2}. This shape theorem  holds also in  the stationary setting, as stated above in \eqref{eq:erg-coc}.  This result comes from  \cite[Theorem 4.4]{Jan-Ras-20-aop} and it is proved in detail in  Appendix B of \cite{Jan-Ras-18-arxiv}.   Now the  proof of \cite[Theorem 4.3]{Geo-Ras-Sep-17-ptrf-2} goes through line-by-line after switching its references and applying \cite[Lemma 4.5(c)]{Jan-Ras-20-aop} to identify the correct centering for the cocycle. 

 Similarly, the non-existence of directed bi-infinite geodesics (Theorem \ref{thm1}\eqref{thm1:biinf} above) proved  in  \cite[Theorem 4.6]{Geo-Ras-Sep-17-ptrf-2} needs only stationarity after minor changes. Essentially the same argument is given in \cite[Lemma 6.1]{Jan-Ras-20-aop} in positive temperature, assuming only stationarity. 
 \end{remark} 
 
We next record an easy consequence of the previous results, ruling out the existence of non-trivial semi-infinite geodesics which are either $e_1$- or $e_2$-directed.  
\begin{lemma}\label{lem:notriv}
For $\w \in \Omega_0^1 \cap \Omega_0^2 \cap \Omega_0^3$, if $\gamma^x$ is a semi-infinite geodesic emanating from $x$ with $\gamma^x_n/n \to e_i$ for some $i \in \{1,2\}$, then $\gamma^x = \geo{}{x}{e_i}$.
\end{lemma}
\begin{proof}
We consider the case of $i=1$, with the case of $i=2$ being similar. Call $x \cdot \ehat = k$ and fix a sequence $\zeta_n \in \ri \Uset$ with $\zeta_n \to e_1$ as $n\to\infty$. By Theorem \ref{thm1}\eqref{thm1:exist}, $\geo{}{x}{\zeta_n}$ is $\Uset_{\zeta_n}$-directed. \cite[Theorem 2.4]{Mar-04} implies that $e_1 \notin \Uset_{\zeta_n}$. Then, by \eqref{paths.2}, if $\gamma^x$ is as in the statement, we must have $\geo{\ell}{x}{\zeta_n} \preceq \gamma_\ell^x \preceq x + (k-\ell) e_1 = \geo{\ell}{x}{e_1}$ for all $n \in \bbN$ and $\ell \geq k$. But Theorem \ref{thm:Bus}\eqref{thm:cocyexist:h} implies that for each fixed $\ell \geq k$,  $\geo{\ell}{x}{\zeta_n} = x + (k-\ell) e_1 = \geo{\ell}{x}{e_1}$ holds for all large enough $n$. The result follows.
\end{proof}

Under the assumption that $\gpp$ is differentiable on $\ri\Uset$,  Theorem \ref{thm1}\eqref{thm1:cont} holds for all $\xi\in\ri\Uset$. An application of the Fubini-Tonelli theorem gives that the claims in Theorem \ref{thm1}\eqref{thm1:coal} and Theorem \ref{thm1}\eqref{thm1:biinf} in fact hold on a single full $\P$-measure event simultaneously for Lebesgue-almost all directions $\xi\in\ri\Uset$.
It is conjectured that the claim in part \eqref{thm1:biinf} holds in fact on a single full-measure event, simultaneously, for all $\xi\in\ri\Uset$.\smallskip

 The next result is a small extension of Lemma 4.4 of \cite{Geo-Ras-Sep-17-ptrf-2}, achieved by an application of the monotonicity in \eqref{path-ordering}. 

\begin{theorem}\label{thm:extreme}
Assume the regularity condition \eqref{g-reg}. For any $\w\in\Omega_0^1\cap\Omega_0^2$, \eqref{eq:geoorder} holds.
\end{theorem}

The next theorem says that there are multiple geodesics that are directed in the same asymptotic direction $\cid$ as the competition interface, which itself can be characterized using the Busemann process. See Figure \ref{fig:cif2}.

\begin{theorem}{\rm\cite[Equation (5.2) and Theorems 2.6, 2.8, and 5.3]{Geo-Ras-Sep-17-ptrf-2}}\label{thm:cif}
There exists a $T$-invariant event $\Omega_0^4$ such that $\P(\Omega_0^4)=1$ and the following hold for all $\w\in\Omega_0^4$.
\begin{enumerate}[label={\rm(\alph*)}, ref={\rm\alph*}]  \itemsep=3pt 
\item\label{thm:cif.a} There exists a unique point $\cid(\w)\in\ri\Uset$ such that \eqref{cid} holds.
\item\label{thm:cif.b} For any $\xi\in\ri\Uset$,   $\P(\cid=\xi)>0$ if and only if $\xi\in(\ri\Uset)\smallsetminus\Diff$.    
\item\label{thm:cif.c} For any $\zeta\prec\eta$ in $\ri\Uset$ with $\nabla\gpp(\zeta+)\ne\nabla\gpp(\eta-)$, for any $x\in\Z^2$, there exists $y\ge x$ such that $\cid(T_y\w)\in\,]\zeta,\eta[$. Consequently, any open interval outside  the closed  linear segments of   $\gpp$  contains $\cid$ with positive probability.  
\item\label{thm:cif.c1} For any $\xi\in(\ri\Uset)\smallsetminus\Diff$ and for any $x\in\Z^2$, there exists $y\ge x$ such that $\cid(T_y\w)=\xi$.
\end{enumerate}
	
If the regularity condition \eqref{g-reg} holds then the following also hold. 
	
\begin{enumerate}[label={\rm(\alph*)}, ref={\rm\alph*},resume]  \itemsep=3pt 
\item\label{cif-lln} We have the limit
\begin{align}\label{thm:cif:(i)}
\cid(\w)= \lim_{n\to\infty} n^{-1}\varphi_n^0(\w).  
\end{align}
\item\label{thm:cif:(ii)} $\cid(T_x\w)$ is the unique direction $\xi$ such that there are at least two $\Uset_\xi$-directed semi-infinite geodesics from $x$, namely $\geo{}{x}{\xi\pm}$, that separate at $x$ and never intersect thereafter.
\end{enumerate}
\end{theorem}

  Remark \ref{rmk:erg3} applies here as well. Ergodicity is invoked in the proofs of parts \eqref{thm:cif.b}, \eqref{thm:cif.c}  and \eqref{thm:cif.c1}  in \cite[Theorem 5.3(iii)--(iv)]{Geo-Ras-Sep-17-ptrf-2}   to apply the cocycle shape theorem.   In  our stationary setting this can be replaced with the combination of Theorem 4.4 and  Lemma 4.5(c) of \cite{Jan-Ras-20-aop}.

The following result for exponential weights, due to Coupier, states that there are no directions $\xi$ with three $\xi$-directed geodesics emanating from the same site.

\begin{theorem}\label{thm-Coupier}{\rm\cite[Theorem 1(2)]{Cou-11}}
Assume that under $\P$, the weights $\{\w_x:x\in\Z^2\}$ are exponentially distributed i.i.d.\ random variables. Then there exists a $T$-invariant event $\OmCou$ with $\P(\OmCou)=1$ and such that for any $\w\in\OmCou$, any $\xi\in\ri\Uset$, and any $x\in\Z^2$, there exist at most two $\xi$-directed semi-infinite geodesics out of $x$.
\end{theorem}

Fix a countable dense set $\Udense \subset \Diff$. 
The following event of full $\P$-probability is the basic setting for the proofs in Sections \ref{sec:bus}--\ref{sec:webpf}:  
	\begin{eqnarray}\label{Om0}
	&\Omega_0= \Omega_0^1 \cap \Omega_0^2\cap\Omega_0^4\cap\Bigl(\bigcap_{\xi\in\Udense}[\Omega_\xi^1\cap\Omega_\xi^3]\Bigr)\cap\Bigl(\bigcap_{\xi\in(\ri\Uset)\setminus\Diff}\Omega_\xi^3\Bigr).  
	\end{eqnarray} 
When additional assumptions are needed, $\Omega_0$ will be further restricted.


\section{Auxiliary lemmas}
\label{a:aux}

The next lemma follows from the shape theorem for cocycles \eqref{eq:erg-coc}. 

\begin{lemma}\label{lem:buselln}
Suppose $\gpp$ is differentiable on $\ri\Uset$. For any $\w \in \Omega_0$, $\xi\in\ri\Uset$, and  any $v\in\bbR^2$, $n^{-1}\B{\xi\pm}(0,\lfloor{nv}\rfloor)$ 
both converge to $v\cdot\nabla\gpp(\xi)$ as $n\to\infty$.
\end{lemma}

\begin{proof}
The claim is obvious for $v=0$. 
Suppose $v\in\bbR_+^2\setminus\{0\}$, the other cases being similar. 
Take $\w\in\Omega_0$ and $\zeta,\eta\in\Udense$ with $\zeta\cdot e_1<\xi\cdot e_1<\eta\cdot e_1$.
Let $x_n=\lfloor{nv}\rfloor=m_ne_1+\ell_ne_2$. Then
		\begin{align*}
		\B{\xi+}(0,x_n)&=\B{\xi+}(0,m_ne_1)+\B{\xi+}(m_ne_1,x_n)\\
		&\le \B{\eta}(0,m_ne_1)+\B{\zeta}(m_n e_1,x_n)\\
		&\le \B{\eta}(0,m_ne_1)+\B{\zeta}(0,x_n)-\B{\zeta}(0,m_ne_1).
		\end{align*}
Divide by $n$, take it to $\infty$, and apply the \eqref{eq:erg-coc} to $\B{\zeta}$ and $\B{\eta}$ to get 
	\[\varlimsup_{n\to\infty} n^{-1}\B{\xi+}(0,x_n)\le (v\cdot e_1)e_1\cdot\nabla\gpp(\eta)+v\cdot\nabla\gpp(\zeta)-(v\cdot e_1)e_1\cdot\nabla g(\zeta).\]
	Take $\zeta$ and $\eta$ to $\xi$ to get
	\[\varlimsup_{n\to\infty} n^{-1}\B{\xi+}(0,x_n)\le v\cdot\nabla\gpp(\xi).\]
The lower bound on the liminf holds similarly and so we have proved the claim for $\B{\xi+}$. 
The same argument works for $\B{\xi-}$.
\end{proof}

The lemma below is proved by calculus. 

\begin{lemma}\label{a:lm:t500}    
 Fix $c>0$.  Then for all $n\ge 1$ and all  $a,b$ such that $c\le a\le b\le a+\frac{c}n$, 
\be\label{a:t500} 
\frac{a^{n-1}b^n}{(a+b)^{2n-1}}   \ge \biggl(\frac12\biggr)^{2n-1} .  
\ee 
\end{lemma} 

\medskip 

\section{M/M/1 queues and  Busemann functions} \label{a:bus}

This appendix summarizes  results from \cite{Fan-Sep-20} that are needed for the proofs of the results of 
Section \ref{sec:exp}.   
Fix parameters  $0<\alpha<\beta$.   We formulate  a stationary M/M/1 queue in a particular way.   The inputs are two independent i.i.d.\ sequences:  an {\it inter-arrival process}  $I=(I_i)_{i\in\Z}$  with marginal distribution  $I_i\sim$ Exp$(\alpha)$ and   a {\it service process}  $Y=(Y_i)_{i\in\Z}$  with marginal distribution  $Y_i\sim$ Exp$(\beta)$.  Out of these inputs are produced two outputs: an {\it inter-departure process}   $\wt I=(\wt I_k)_{k\in\Z}$ and a {\it sojourn process}   $J=(J_k)_{k\in\Z}$,  through the following formulas.   Let $G=(G_k)_{k\in\Z}$ be any function on $\Z$ that satisfies $I_k=G_k-G_{k+1}$.   Define the function    $\wt G=(\wt G_k)_{k\in\Z}$ by 
\be\label{m:800}\begin{aligned} 
\wt G_k
=\sup_{m:\,m\ge k}  \Bigl\{  G_m+\sum_{i=k}^m Y_i\Bigr\}  
=   G_k+Y_k + \sup_{m:\,m\ge k}  \sum_{i=k}^{m-1} (Y_{i+1}-I_i)  . 
\end{aligned} \ee
The convention for the empty sum is $\sum_{i=k}^{k-1} =0$. 
 Under the assumption on $I$ and $Y$,  the supremum in \eqref{m:800} is almost  surely  assumed at some finite $m$.     Then define the outputs by 
 \be\label{m:801}  \wt I_k =  \wt G_k - \wt G_{k+1}  \ee 
and  
  \be\label{m:J} J_k=\wt G_k -  G_k = Y_k + \sup_{m:\,m\ge k}  \sum_{i=k}^{m-1} (Y_{i+1}-I_i). 
    \ee
The outputs satisfy  the useful  iterative equations
  \be\label{m:IJ5}  
 \wt I_k=Y_k+(I_k-J_{k+1})^+
 \quad\text{and}\quad 
 J_k=Y_k+(J_{k+1}-I_k)^+ .  
  \ee
In particular, this implies  the inequality $\wt I_k\ge Y_k$. 
    
 It is a basic fact about M/M/1 queues  that    $\wt I$ and $J$ are i.i.d.\ sequences with marginals $\wt I_k\sim$ Exp$(\alpha)$ and $J_k\sim$ Exp$(\beta-\alpha)$.  Furthermore, the three variables  $(Y_k, I_k, J_{k+1})$ on the right-hand sides of equations \eqref{m:IJ5}  are independent.   (See for example Appendix A in \cite{Fan-Sep-20}.)   But  $\wt I$ and $J$ are not independent of each other.  
    
The queueing interpretation goes as follows.  A  service station  processes a bi-infinite sequence of customers.   Queueing time runs backwards on the lattice $\Z$.   $I_i$ is the time between the arrivals of customers $i+1$ and $i$  ($i+1$ arrived before $i$)  and    $Y_i$ is the service time required by customer $i$.   $\wt I_k$ is the time between the departures of customers $k+1$ and $k$, with $k+1$ departing before $k$.  $J_k$ is the sojourn time of customer $k$, that is, the total time customer $k$ spent in the system from arrival to departure.  $J_k$ is the sum of the service time $Y_k$ and the waiting time of customer $k$, represented by the last member of \eqref{m:J}.  Because of our unusual convention with the backward indexing, even if $G_k$ is the moment of arrival of customer $k$, $\wt G_k$ is not the moment of departure.   The definition of $\wt G$  in \eqref{m:800} is natural in the present setting because it immediately ties in with LPP.    The convention in \cite{Fan-Sep-20} is different because in \cite{Fan-Sep-20} geodesics go south and west instead of north and east.
 
 The   joint distribution of successive  nearest-neighbor increments  of two Busemann functions on a horizontal or vertical line can now be described as follows. This is a special case of Theorem 3.2 in \cite{Fan-Sep-20}.  
 
 \begin{theorem}\label{th:B-q1}  Let   $\zeta\prec\eta$ in $\ri\Uset$ with parameters $\alpha=\alpha(\zeta)<\alpha(\eta)=\beta$ given by \eqref{u-a}.    Let $I=(I_i)_{i\in\Z}$    and    $Y=(Y_i)_{i\in\Z}$   be two independent i.i.d.\ sequences and define $\wt I=(\wt I_k)_{k\in\Z}$ as above through \eqref{m:800}--\eqref{m:801}.

 \begin{enumerate}  [label={\rm(\alph*)}, ref={\rm(\alph*)}]   \itemsep=3pt  
\item\label{B1-1} Let $I_i\sim$ {\rm Exp}$(\alpha)$ and $Y_i\sim$ {\rm Exp}$(\beta)$.  Then the sequence 
$(\B{\zeta}_{ke_1, (k+1)e_1}, \, \B{\eta}_{ke_1, (k+1)e_1})_{k\in\Z}$  has the same joint distribution as the pair $(\wt I, Y)$. 

  \item\label{B1-2}    Let    $I_i\sim$ {\rm Exp}$(1-\beta)$ and    $Y_i\sim$ {\rm Exp}$(1-\alpha)$.  Then the sequence  
  $(\B{\zeta}_{ke_2, (k+1)e_2}, \, \B{\eta}_{ke_2, (k+1)e_2})_{k\in\Z}$  has the same joint distribution as the pair $(Y, \wt I)$.      
  
  \end{enumerate} 

  \end{theorem}
 
Next we  derive a random walk representation for the sequence $\{\B{\zeta}_{ke_1, (k+1)e_1}- \B{\eta}_{ke_1, (k+1)e_1}\}_{k\in\Z}$ of (nonnegative) differences.   By Theorem \ref{th:B-q1} this sequence is equal in distribution to $\{\wt I_k-Y_k\}_{k\in\Z}$.   Define a two-sided randon walk $S$  with positive drift $E[I_{i-1}-Y_i]=\alpha^{-1}-\beta^{-1}$ by 
\be\label{BS}
S_n=\begin{cases}  -\sum_{i=n+1}^0 (I_{i-1}-Y_i), &n<0 \\[3pt]  0, &n=0\\[3pt] \sum_{i=1}^n (I_{i-1}-Y_i), &n>0.   \end{cases}
\ee
Then from \eqref{m:IJ5} and \eqref{m:J}, 
\begin{align*}
\wt I_k-Y_k = (I_k-J_{k+1})^+ = \biggl\{\,\inf_{n:\,n> k} \sum_{i=k+1}^{n} (I_{i-1}-Y_{i})\biggr\}^+ =  \bigl\{   \inf_{n:\,n> k}  (S_n-S_{k}) \bigr\}^+. 
\end{align*}
  From above we can record that  for $r>0$, 
\be\label{B4}
\P(  \B{\zeta}_{0, e_1}> \B{\eta}_{0, e_1}) =P(I_k>J_{k+1}) =  \frac{\beta-\alpha}{\beta}. 
\ee

  \begin{corollary}\label{cor:B-q2}  Let   $\zeta\prec\eta$ in $\ri\Uset$ with parameters $\alpha=\alpha(\zeta)<\alpha(\eta)=\beta$ given by \eqref{u-a}.    Let $S$ be the random walk in \eqref{BS} with step distribution {\rm Exp}$(\alpha)-{\rm Exp}(\beta)$.  Then the sequence $\{\B{\zeta}_{ke_1, (k+1)e_1}- \B{\eta}_{ke_1, (k+1)e_1}\}_{k\in\Z}$ has the same distribution as the sequence $ \bigl\{ \bigl( \,\ddd\inf_{n:\,n> k}   S_n-S_{k}\bigr)^+ \bigr\}_{k\in\Z}$.    \end{corollary} 
 
 \section{Random walk}\label{a:RW}
 
 Let $0<\alpha<\beta$ and let $\{X_i\}_{i\in\Z}$ be a doubly infinite sequence of i.i.d.\ random variables with marginal distribution $X_i \sim\text{Exp}(\alpha)-\text{Exp}(\beta)$ (difference of two independent exponential random variables).   Let $\theta$ denote the shift on the underlying canonical sequence space so that   $X_j=X_k\circ\theta^{j-k}$. 
  Let $\{S_n\}_{n\in\Z}$ be the  two-sided random walk   such that $S_0=0$ and $S_n-S_m=\sum_{i=m+1}^n X_i$ for all $m<n$ in $\Z$.  
 Let $(\lambda_i)_{i\ge 1}$ be the strict ascending ladder epochs of the forward walk.  That is, begin with $\lambda_0=0$,  and for $i\ge 1$  let 
\[  \lambda_i=\inf\{ n>\lambda_{i-1}:   S_{n} > S_{\lambda_{i-1}} \} . 
  \]
 The positive drift of $S_n$ ensures that these variables are finite almost surely.   For $i\ge 1$  define the increments   $L_i=\lambda_i-\lambda_{i-1}$ and $H_i=S_{\lambda_{i}} -S_{\lambda_{i-1}}$.   The variables $\{L_i, H_i\}_{i\ge 1}$ are mutually independent with marginal distribution
 \be\label{expH1L1} \begin{aligned}
P(L_1=n, H_1>r) 
= C_{n-1}\,  \frac{\alpha^{n-1}\beta^n}{(\alpha+\beta)^{2n-1}}e^{-\alpha r}, \quad n\in\N, \, r\ge 0. 
\end{aligned}\ee 
Above $C_n=\frac{1}{n+1}{2n\choose n}$ for $n\ge0$ are the  Catalan numbers. 
 A small extension of the proof of Lemma B.3 in \cite{Fan-Sep-20} derives  \eqref{expH1L1}.  

  Let 
\be\label{d:expU}   W_0=\inf_{m>0} S_m . \ee   
Note that $W_0\circ\theta^n>0$ if and only if $S_n < \inf_{m>n} S_m$, that is, $n$ is a last exit time for the random walk. 
Define   successive last exit times  (in the language of Doney \cite{Don-89}) by 
\be\label{exp-si1}\begin{aligned} 
 \sigma_0&=\inf\{ n\ge0:  S_n < \inf_{m>n} S_m \}  
 \\
\text{and for $i\ge 1$,}\qquad 
 \sigma_i&=\inf\{ n>\sigma_{i-1}:   S_n < \inf_{m>n} S_m \}. 
\end{aligned}\ee

\begin{proposition}\label{pr:exp9}   
Conditionally on $W_0>0$ {\rm(}equivalently, on $\sigma_0=0${\rm)},  the pairs  $\{(\sigma_i-\sigma_{i-1}, S_{\sigma_i}-S_{\sigma_{i-1}})\}_{i\ge 1}$ are i.i.d.\ with  marginal distribution 
\be\label{exp85} \begin{aligned}
P(\sigma_i-\sigma_{i-1}=n, \, S_{\sigma_i}-S_{\sigma_{i-1}}>r\,\vert\,W_0>0) 
= C_{n-1}\,  \frac{\alpha^{n-1}\beta^n}{(\alpha+\beta)^{2n-1}}e^{-\alpha r}
\end{aligned}\ee 
for all $i\in\N$,  $n\in\N$,   and  $r\ge 0$. 
\end{proposition} 
 
\begin{proof} 
Let $0=n_0<n_1<\dotsm< n_\ell$ and $r_1,\dotsc,r_\ell>0$.   
  The dual random walk   $S^*_k=S_{n_\ell}-S_{n_\ell-k}$ for $0\le k\le n_\ell$ (p.~394 in Feller II \cite{Fel-71}) satisfies   $(S^*_k)_{0\le k\le n_\ell} \deq (S_k)_{0\le k\le n_\ell}$ and is independent of  $W_0\circ\theta^{n_\ell}$.   
  \begin{align*}
&P\bigl(\forall i=1,\dotsc,\ell:  \sigma_i-\sigma_{i-1}=n_i -n_{i-1} \text{ and }  S_{\sigma_i}-S_{\sigma_{i-1}}> r_i, \, W_0>0\bigr)\\  
&=
P\bigl(\forall i=1,\dotsc,\ell:  \sigma_i=n_i \text{ and }  S_{\sigma_i}-S_{\sigma_{i-1}}>r_i, \, W_0>0\bigr)\\
&=P\bigl( \forall i=1,\dotsc,\ell:  S_k>S_{n_i}>S_{n_{i-1}}+r_i \text{ for }  k\in\,\rzb n_{i-1}, n_i\lzb\, , \, W_0\circ\theta^{n_\ell} >0\bigr) \\
&= P\bigl( \forall i=1,\dotsc,\ell:  S^*_j< S^*_{n_\ell-n_i}<S^*_{n_\ell-n_{i-1}}-r_i \text{ for }  j\in\,\rzb n_\ell-n_i, n_\ell-n_{i-1} \lzb\, , \, W_0\circ\theta^{n_\ell} >0\bigr) \\
&=  P\bigl( \forall k=1,\dotsc,\ell:  \lambda_k=n_\ell-n_{\ell-k}\text{ and }  H_k>r_{\ell-k+1} \bigr)  \, P(W_0>0) \\
&=  P\bigl( \forall k=1,\dotsc,\ell:  L_k=n_{\ell-k+1}-n_{\ell-k}\text{ and }  H_k>r_{\ell-k+1} \bigr)  \, P(W_0>0)   .
\end{align*} 
The claim follows from the independence of $\{L_k, H_k\}$ and \eqref{expH1L1}. 
\end{proof} 
 
  From $\sigma_0$ as defined in \eqref{exp-si1},  extend  $\sigma_i$   to negative indices by defining, for $i=-1, -2, -3,\dotsc$, 
\be\label{exp:si1.1} 
\sigma_i=\max\bigl\{k<\sigma_{i+1}:     S_k< S_{\sigma_{i+1}}\bigr\} . 
\ee
For each $k\in\Z$ set 
\be\label{Wk}   W_k=   \inf_{n:\,n> k}   S_n-  S_{k} . \ee 
Then one can check that $\sigma_{-1}<0\le \sigma_0$, and for all $i,k\in\Z$, 
\be\label{78-63}  
S_{\sigma_i}=     \inf_{n:\,n> \sigma_{i-1}}   S_n  ,
\ee
\be\label{78-72}  W_{\sigma_i} 
=   \inf_{n:\,n> \sigma_i}   S_n- S_{\sigma_i} =  S_{\sigma_{i+1}}- S_{\sigma_i}, 
\ee
and 
\be\label{78-65}  
  W_k>0   \ \iff\  k\in\{\sigma_i: i\in\Z\}. 
   \ee

\begin{theorem}\label{th:78-70} 
Conditionally on $\sigma_0=0$, equivalently, on $W_0>0$,   $\{\sigma_{i+1}-\sigma_i, W_{\sigma_i}: i\in\Z\}$ is an  i.i.d.\ sequence with marginal distribution 
\be\label{rw110} \begin{aligned} 
&P\bigl(\sigma_{i+1}-\sigma_i=n, \, W_{\sigma_i} >r\,\big\vert\,W_0>0\bigr) 
= C_{n-1}\,  \frac{\alpha^{n-1}\beta^n}{(\beta+\alpha)^{2n-1}}e^{-\alpha r} 
\end{aligned} \ee
for all $i\in\N$,  $n\in\N$,   and  $r\ge 0$. 
\end{theorem}


\begin{proof} 
 Define the processes $\Psi_+=\{\sigma_{i+1}-\sigma_i, W_{\sigma_i}: i\ge0\}$ and $\Psi_-=\{\sigma_{i+1}-\sigma_i, W_{\sigma_i}: i\le-1\}$.    $\Psi_+$ and  the conditioning  event    $W_0>0$ depend only on $(S_n)_{n\ge1}$,   while  $W_0>0$  implies   for $n<0$ that  $  \inf_{m:\,m>n} S_{m}= \inf_{m:\, n< m\le 0} S_{m}$.   $\Psi_+$ and $\Psi_-$ have been decoupled.  
 
Define another forward walk with the same step distribution  by $\wt S_k=-S_{-k}$ for $k\ge 0$.     Let $\lambda_0=0$,   $(\lambda_i)_{i\ge 1}$ be the successive ladder epochs and $H_i=\wt S_{\lambda_i}-\wt S_{\lambda_{i-1}}$   the successive ladder height increments for the $\wt S$ walk.  

We claim that on the event $\sigma_0=0$, 
\be\label{78-76} 
\lambda_{-i}=-\sigma_i \quad\text{and}\quad 
W_{\sigma_i}= H_{-i} \quad\text{for $i\le-1$.}  
\ee
First by definition, $\lambda_0=0=-\sigma_0$.  By the definitions and by induction, for $i\le -1$, 
\begin{align*}
\lambda_{-i} &=\min\{ k>\lambda_{-i-1}:  \wt S_k> \wt S_{\lambda_{-i-1}} \}
=\min\{ k>-\sigma_{i+1}:  S_{-k}< S_{\sigma_{i+1}} \} \\
&=-\max\{ n<\sigma_{i+1}:  S_{n}< S_{\sigma_{i+1}} \} 
= -\sigma_i 
\end{align*} 
where the last equality came from \eqref{exp:si1.1}.   Then from \eqref{78-72}, 
\begin{align*}  W_{\sigma_i} 
  =  S_{\sigma_{i+1}}- S_{\sigma_i}  =    -\wt S_{-\sigma_{i+1}}+\wt S_{-\sigma_i}
  =  -\wt S_{\lambda_{-i-1}}+\wt S_{\lambda_{-i}}= H_{-i}. 
\end{align*}
Claim \eqref{78-76} has been verified.  

Let   $\Psi'=\{ \lambda_{-i}-\lambda_{-i-1}, H_{-i}: i\le-1\}$,  a function of $(S_n)_{n\le -1}$.  
 By \eqref{78-76},    $\Psi_-=\Psi'$ on the event $\sigma_0=0$. 

Let $A$ and $B$ be suitable measurable sets of infinite sequences.  
\begin{align*}
&P( \Psi_+\in A, \Psi_-\in B\,\vert\,  W_0>0) =  \frac1{P(W_0>0)} P( \Psi_+\in A, \Psi'\in B, W_0>0) \\
&\qquad 
=\frac{P( \Psi_+\in A, W_0>0)}{P(W_0>0)} \, P(\Psi'\in B) 
=P( \Psi_+\in A\,\vert\,W_0>0)  P(\Psi'\in B). 
\end{align*} 
The conclusion follows.  By Proposition \ref{pr:exp9}, conditional on $W_0>0$, 
 $\Psi_+$ has the i.i.d.\ distribution \eqref{rw110}, which is the same as the i.i.d.\ distribution  \eqref{expH1L1} of $\Psi'$. 
\end{proof}

\bigskip
\footnotesize
\noindent\textit{Acknowledgments.}
The authors thank Yuri Bakhtin for helpful conversations.
C.\ Janjigian was partially supported by National Science Foundation grant DMS-2125961.
F.\ Rassoul-Agha was partially supported by NSF grants DMS-1407574 and DMS-1811090.
T.\ Sepp\"al\"ainen was partially supported by  National Science Foundation grants  DMS-1602486  and DMS-1854619, and by the Wisconsin Alumni Research Foundation.



\bibliographystyle{plain}
\bibliography{GeoWebPaper}


\end{document}